\documentclass[10pt]{article}
\usepackage{amssymb, amsfonts, amsmath}
\usepackage{amsthm, thmtools, mathrsfs, comment}
\usepackage{enumitem, yhmath, accents}
\usepackage{imakeidx}
\usepackage[colorlinks=true,linkcolor=blue]{hyperref}
\usepackage[retainorgcmds]{IEEEtrantools}
\makeindex

\declaretheorem[name=Definition,style=definition,qed=$\dashv$,numberwithin=section]
{dfn}

\declaretheorem[name=Theorem,style=plain,sibling=dfn]{tm}
\declaretheorem[name=Lemma,style=plain,sibling=dfn]{lem}
\declaretheorem[name=Corollary,style=plain,sibling=dfn]{cor}
\declaretheorem[name=Remark,style=definition,sibling=dfn]{rem}
\declaretheorem[name=Case,style=definition,sibling=dfn]{case}
\declaretheorem[name=Proposition,style=plain,sibling=dfn]{prop}
\declaretheorem[name=Claim,style=plain,numberwithin=tm]{clm}
\declaretheorem[name=Subclaim,style=plain,numberwithin=clm]{sclm}
\numberwithin{equation}{section}

\newcommand{\Xx}{\mathcal{X}}
\newcommand{\Yy}{\mathcal{Y}}
\newcommand{\Shift}{\mathrm{Shift}}
\newcommand{\opt}{\mathrm{opt}}
\newcommand{\dfnemph}{\textbf}
\newcommand{\rest}{\mathord{\upharpoonright}}
\newcommand{\tu}{\textup}

\newcommand{\univ}[1]{\lpole #1\rpole}

\newcommand{\dom}{{\rm dom}}

\newcommand{\lh}{{\rm lh}}

\newcommand{\wins}{\ins}
\newcommand{\wpins}{\pins}
\newcommand{\nwins}{\nins}

\newcommand{\Jins}{\ins^\J}
\newcommand{\Jpins}{\pins^\J}

\newcommand{\card}{\mathrm{card}}

\renewcommand{\models}{\vDash}
\newcommand{\powerset}{\mathcal{P}}

\def\M{{\mathcal{M}}}
\def\N{{\mathcal{N}}}

\def\T{{\mathcal{T}}}

\def\Ss{{\mathcal{S}}}

\newcommand{\ins}{\trianglelefteq}
\newcommand{\nins}{\ntrianglelefteq}
\newcommand{\pins}{\triangleleft}
\newcommand{\npins}{\ntriangleleft}
\newcommand{\Hull}{\mathrm{Hull}}
\newcommand{\cHull}{\mathrm{cHull}}
\newcommand{\crit}{\mathrm{crit}}
\newcommand{\RR}{\mathbb{R}}
\newcommand{\rSigma}{{\mathrm{r}\Sigma}}
\newcommand{\OR}{\mathrm{Ord}}
\newcommand{\J}{\mathcal J}
\newcommand{\un}{\cup}
\newcommand{\sub}{\subseteq}
\newcommand{\om}{\omega}
\newcommand{\ZF}{\mathsf{ZF}}
\newcommand{\AD}{\mathsf{AD}}
\newcommand{\ZFC}{\mathsf{ZFC}}
\newcommand{\CC}{\mathbb{C}}
\newcommand{\es}{\mathbb{E}}
\newcommand{\Tt}{\mathcal{T}}

\newcommand{\PP}{\mathbb{P}}
\newcommand{\Vv}{\mathcal{V}}
\newcommand{\Uu}{\mathcal{U}}

\newcommand{\conc}{\ \widehat{\ }\ }

\newcommand{\Lp}{\mathrm{Lp}}
\newcommand{\sats}{\vDash}
\newcommand{\cut}{\backslash}
\newcommand{\Ult}{\mathrm{Ult}}
\newcommand{\core}{\mathfrak{C}}
\newcommand{\pow}{\mathcal{P}}
\newcommand{\inter}{\cap}

\newcommand{\lex}{\mathrm{lex}}

\newcommand{\Coll}{\mathrm{Col}}
\newcommand{\rg}{\mathrm{rg}}
\newcommand{\lpole}{\left\lfloor}
\newcommand{\rpole}{\right\rfloor}

\renewcommand{\l}{\mathit{l}}
\newcommand{\Hh}{\mathcal{H}}

\newcommand{\nth}{\mathrm{th}}
\newcommand{\Ord}{\mathrm{Ord}}
\newcommand{\Th}{\mathrm{Th}}
\newcommand{\com}{\circ}
\newcommand{\HC}{\mathrm{HC}}
\newcommand{\cof}{\mathrm{cof}}

\newcommand{\Ll}{\mathcal{L}}
\newcommand{\all}{\forall}
\newcommand{\ex}{\exists}
\newcommand{\Ttvec}{\vec{\Tt}}

\newcommand{\id}{\mathrm{id}}
\newcommand{\cross}{\times}

\newcommand{\elem}{\preccurlyeq}

\newcommand{\pred}{\text{pred}}

\newcommand{\C}{\mathcal{C}}
\newcommand{\opbk}{\mathscr{B}}
\newcommand{\her}{\mathscr{H}}
\newcommand{\plI}{\mathrm{I}}
\newcommand{\plII}{\mathrm{II}}
\newcommand{\ZFmin}{\mathsf{ZF}^-}
\newcommand{\rank}{\mathrm{rank}}
\newcommand{\psub}{\subsetneq}

\newcommand{\Fop}{\mathcal{F}}
\newcommand{\base}{\mathrm{base}}

\newcommand{\sq}{\mathrm{sq}}

\newcommand{\trancl}{\mathrm{trancl}}

\newcommand{\xvec}{\vec{x}}

\newcommand{\G}{\mathcal{G}}

\newcommand{\maps}{\colon}

\newcommand{\Abar}{\bar{A}}

\newcommand{\bfrSigma}{\undertilde{\rSigma}}

\newcommand{\Ww}{\mathcal{W}}

\mathchardef\mhyphen="2D

\newcommand{\DC}{\mathsf{DC}}

\newcommand{\partialto}{\dashrightarrow}
\newcommand{\op}{\mathrm{op}}

\newcommand{\Cc}{\mathcal{C}}
\newcommand{\pvec}{\vec{p}}
\newcommand{\Mtilde}{\widetilde{\M}}

\newcommand{\eqdef}{=_{\mathrm{def}}}
\newcommand{\onto}{\stackrel{\mathrm{onto}}{\to}}
\newcommand{\cofto}{\stackrel{\mathrm{cof}}{\to}}

\newcommand{\alphavec}{\vec{\alpha}}
\newcommand{\ph}{\mathfrak{P}}
\renewcommand{\L}{\mathcal{L}}

\newcommand{\witri}{\widetriangle}
\newcommand{\hmE}{E}
\newcommand{\hmP}{P}
\newcommand{\hmS}{S}
\newcommand{\hmX}{X}
\newcommand{\hmb}{cb}
\newcommand{\hmp}{cp}
\newcommand{\hmbv}{b}
\newcommand{\hmpv}{p}
\newcommand{\hmmu}{\mu}
\newcommand{\hme}{e}
\newcommand{\hmEdot}{\dot{\hmE}}
\newcommand{\hmPdot}{\dot{\hmP}}
\newcommand{\hmSdot}{\dot{\hmS}}
\newcommand{\hmXdot}{\dot{\hmX}}
\newcommand{\hmbdot}{\dot{\hmb}}
\newcommand{\hmpdot}{\dot{\hmp}}
\newcommand{\hmmudot}{\dot{\hmmu}}
\newcommand{\hmedot}{\dot{\hme}}
\newcommand{\actext}{F}
\newcommand{\hmPvec}{{\vec{\hmP}}}

\newcommand{\Pvec}{\vec{P}}
\newcommand{\cvec}{\vec{c}}
\newcommand{\Dd}{\mathcal{D}}
\newcommand{\dirlim}{\mathrm{dirlim}}
\newcommand{\wl}{\l}
\newcommand{\wupto}{|}
\newcommand{\wdoubleupto}{||}
\newcommand{\Jupto}{|^{\J}}
\newcommand{\Lim}{\mathrm{Lim}}
\newcommand{\R}{\mathcal{R}}
\newcommand{\Q}{\mathcal{Q}}
\newcommand{\K}{\mathcal{K}}
\renewcommand{\P}{\mathcal{P}}

\begin{document}
\title{The fine structure of operator mice
\renewcommand{\thefootnote}{\fnsymbol{footnote}} 
\renewcommand{\thefootnote}{\arabic{footnote}}}
\author{Farmer Schlutzenberg\footnote{Institute of Discrete Mathematics and Geometry, TU Vienna, Email: farmer.schlutzenberg@tuwien.ac.at}\\Nam 
Trang\footnote{Department of Mathematics, University of North Texas, USA, Email: nam.trang@unt.edu}}
\maketitle

\begin{abstract}
We develop the fine structure theory of  \emph{operator-premice}.
These are a generalization of standard premice, in which an abstract operator $\Fop$ is used to form the successor steps in the internal hierarchy of the premouse, instead of Jensen's $\J$-operator (which computes rudimentary closure). Such notions have seen applications in core model induction arguments, but their theory has not previously been developed in detail.
We define \emph{fine condensation} for operators $\Fop$ and show that fine 
condensation and iterability together ensure that $\Fop$-mice have the fundamental fine 
structural properties including universality and solidity of the standard parameter.
\end{abstract}

\tableofcontents

\section{Introduction}\label{sec:intro}

The core model induction is the most general and successful method for computing lower bounds for the consistency strengths of strong theories, like $\sf{PFA}$, $\sf{MM}$ and many others. It is used to construct  models of $\sf{AD}^+$ of high complexity -- which themselves contain inner models for large cardinals -- from such theories;  the papers
  \cite{PFA_implies_ADLR}, \cite{cmi}, \cite{wilson2012contributions}, \cite{trang2016pfa}, \cite{adolf2024ideals}, \cite{trang2021determinacy}, \cite{sargsyan2016tame}, \cite{sargsyan2014non}, \cite{busche2009strength} develop the general theory of the core model induction and give various applications. 
The construction is inductive in structure.
Roughly, one proves that the universe $V$ (or some part thereof, such as the hereditarily countable sets $\HC$)
is closed under certain  kinds of operators (functions) $\Fop$ which yield inner models for large cardinals. The proof is by induction on the complexity of such $\Fop$, as measured in terms of the Wadge hierarchy in models of $\AD^+$.
 Suppose we have constructed an operator $\Fop$ of the right kind, which captures a pointclass $\undertilde{\Gamma}$ of $\sf{AD}^+$.\footnote{In the same sense that $\M_1^\sharp(x)$ captures $\Sigma^1_2(x)$.} We would like to construct an operator of higher complexity, and a pointclass $\undertilde{\Omega}$ of $\sf{AD}^+$ that strictly contains $\undertilde{\Gamma}$. In one case, we can construct such objects by  constructing the operator $x\mapsto\M_1^{\Fop,\sharp}(x)$; this is the sharp for the canonical iterable $\Fop$-closed model for one Woodin cardinal over $x$, analogous to $\M_1^\#(x)$. This is usually achieved by showing that some fully backgrounded construction or partially backgrounded construction ($K^c$-construction) over $x$,
 which is built relative to $\Fop$, reaches $\M_1^{\Fop,\sharp}(x)$. Here ``built relative to $\Fop$'' means that  the successor stages of the construction are applications of $\Fop$, instead of  Jensen's operator $\J$. (Given a set $X$, $\J(X)$
 is the  closure of $X\un\{X\}$ under the rudimentary set functions. Standard premice are constructed using $\J$ to extend the model at successor levels,
 and  extenders are added at certain limit levels.)  In order for this kind of construction to work, $\Fop$ should satisfy special properties, generalizing many of those that $\J$ satisfies. This paper defines precisely these concepts and generalizes fine-structural and iterability results from ordinary mice to $\mathcal{F}$-mice.

In an \emph{$\Fop$-premouse},
 $\Fop$ is used to extend the model at successor levels, instead of $\J$.
The operator $\Fop$ can be used to feed different kinds of information into a model. For example, ordinary mice,
or an iteration strategy,
or the specification of term relations for a self-justifying system, are some examples of the kind of information that might be fed in. 
We will define \emph{$\Fop$-premice} for a
fairly
wide class of operators
$\Fop$ with nice condensation properties, and develop their basic theory. 
Versions of this theory have been outlined and used by others
(see particularly \cite[\S1.3]{cmi} and \cite[\S2.1]{wilson2012contributions}), but
without supplying a very thorough development of the theory. We give here a
more thorough development.
Aside from providing more details,
some of the basic definitions we use here differ from those in \cite{cmi} and \cite{wilson2012contributions}
in important ways.
But other than in Remark \ref{rem:problems_with_condenses_well} (which can be omitted),
this paper has no formal dependence on those  two papers, though they do provide significant motivation for what we do here.

If $X$ is a transitive set in the domain of an operator $\Fop$ of the kind in which we are interested, then $\N=\Fop(X)$ will be a transitive structure with $X\in\N$, and $\Fop(X)$ will be a (very) simple instance of an \emph{$\Fop$-premouse over $X$}. More generally,
if $\R$ is a sound $\Fop$-premouse over $X$,
then $\Fop(\R)$ will be an $\Fop$-premouse over $X$, with $\R\pins\Fop(\R)$ (that is, $\R$ is a proper initial segment of $\Fop(\R)$, in fact
the largest such).
An essential feature
of the operators suitable for our purposes
is their behaviour under \emph{condensation},
which should be reasonably
 analogous to that of the $\J$-operator.
For example, if $X$ is a transitive set in the domain of $\Fop$, $\R,\Ss$ are $\Fop$-premice
over $X$, $\M$ is a transitive structure
with $\R\in\M$ and $\pi:\M\to\Fop(\Ss)$ is elementary with $\pi(\R)=\Ss$
and $\pi\rest X\cup\{X\}=\id$,
then we might want to know that $\M=\Fop(\R)$.
In fact, we will also want to consider such condensation with respect to partially elementary maps that show up in fine structural contexts.
We will formulate such properties -- \emph{coarse condensation} and \emph{fine condensation} --
in Definitions \ref{dfn:condenses_coarsely}
and \ref{dfn:condenses_finely} respectively.
But before we can discuss those properties,
we need to describe the more basic and general properties of \emph{operator-premice},
which will be the abstract form for $\Fop$-premice, irrespective of the features of any particular $\Fop$.
This is the subject of \S\ref{SigmaMiceOverR}.
After having developed this theory, we will prove in Theorem \ref{thm:k+1-fineness} that if $\Fop$ is an operator with fine condensation, then $\Fop$-iterable $\Fop$-premice satisfy the fundamental fine structural properties that are essential to our understanding of standard premice.

As mentioned above,
some basic notions we use  differ significantly from their analogues in \cite{cmi} and \cite{wilson2012contributions}.
Let us now try to convey some idea about this.
The fine version of condensation for operators
(described roughly in the previous paragraph) which is employed in \cite{cmi} and \cite{wilson2012contributions} is that of \emph{condenses well} (see \cite[Definition 1.3.2]{cmi} and \cite[Definition 2.1.10]{wilson2012contributions}), and this notion is central to the theory in those papers.
As we explain in Remark \ref{rem:problems_with_condenses_well}, this property does not fully function as one would like, and in particular,
 in the terminology of \cite{wilson2012contributions},
 the operators $F_G$ derived from mouse operators $G$ typically do not condense well, contrary to \cite[Lemma 2.1.12]{cmi}.\footnote{Actually, there is a  minor further issue in \cite[Lemma 2.1.12]{cmi},
 or more to the point, in \cite[Definition 2.1.8]{cmi},
 upon which \cite[2.1.12]{cmi} relies; $F_G$ (as specified in \cite[2.1.8]{cmi}) is typically not well-defined in the first place. There is a natural correction to this, but employing the correction, \cite[Lemma 2.1.12]{cmi} fails.} We will show that the variant we introduce, \emph{condenses finely} (see \ref{dfn:condenses_finely}),
 behaves as desired. Because  Remark \ref{rem:problems_with_condenses_well} motivates some key aspects of the definition of \emph{condenses finely}, we have placed it just prior to the formulation of that definition.
But \ref{rem:problems_with_condenses_well} does not rely on the material in the paper prior to where it appears, and the reader who wishes to start with it should have no difficulty in doing so.

A second key definition of \cite{cmi} and \cite{wilson2012contributions}
is that of \emph{model operator}
(see \cite[Definition 1.3.1]{cmi}, and the similar \cite[Definition 2.1.4]{wilson2012contributions}).
The authors were not able to develop the theory at the level of generality of model operators with condensation properties,
 because we could not see how to define appropriate $\Sigma_1$-Skolem functions, nor prove facts such as the preservation of the $1$st standard parameter under iteration maps.
Thus, we make stronger hypotheses on the kinds of structures we work with, ensuring more properties familiar from $\J$-structues;\footnote{A \emph{$\J$-structure} is one of form $(\J^A_\alpha,B)$ for some  $A$, where $\alpha$ is an ordinal or $\alpha=\OR$, and some predicate $B\sub\J_\alpha^A$.
Here $\J_\alpha^A$ refers to the $\alpha$th iterate of the $\J$-operator relativized to $A$.}  see especially  Definitions \ref{dfn:potential opm}, \ref{dfn:adequate} and \ref{dfn:opm}.

Our proof of the solidity of the standard parameter, part
of Theorem \ref{thm:k+1-fineness}, is based on that in the union of  \cite{FSIT}, \cite{steel2010outline} and \cite{deconIMT}. But we provide some  details which are not discussed explicitly in those papers,  which are also relevant in the case of ordinary premice (as opposed to operator-premice), and which the authors believe are non-trivial. In the introduction to \S\ref{sec:solidity},
we isolate the point in the proof at which the details are relevant. 

In the paper,
we will mostly focus on  material that is new, skipping certain
parts which are immediate
transcriptions of the theory of standard premice, although for purposes of readability and self-containment, we do include some fairly standard material.

We have tried to
develop the theory in a manner {\color{black}that its content}   is mostly compatible  with the literature. This is part of our motivation for developing the
theory of $\Fop$-premice abstractly,
dealing with operators $\Fop$ more general than those given by $\J$-structures; cf.~the developments in \cite{cmi} and \cite{wilson2012contributions}, which are abstract.
Of course the abstract development also makes the work more general,
and has the advantage of showing
which properties of $\J$-structures are most essential to the theory.
But it does incur the cost of increasing  complexity somewhat.
A reasonable alternative would have been to give a more concrete development by restricting attention to operators given by
$\J$-structures, and in the end, all applications known to the authors are of this form.
Also, if one deals exclusively  with $\J$-structures, one can more naturally formulate
fine structural condensation properties regarding \emph{all} $\J$-initial segments of the model.
But at least the most straightforward analogues of condensation for abstract $\Fop$-mice apply only to $\Fop$-initial
segments of the model.\footnote{That is, given a reasonably closed $\Fop$-mouse $\M$, it is straightforward to formulate
condensation properties with respect to embeddings $\Hh\to\M$, or $\Hh\to\Fop(\M)$,
or $\Hh\to\Fop(\Fop(\M))$, etc, but it is not so clear how this should be done with respect to embeddings $\Hh\to\N$ when $\M\in\N\in\Fop(\M)$.} This
seems to be a significant complication for abstract $\Fop$-mice.\footnote{For example, strategy mice can
either
be defined as an instance of the general theory here, or as $\J$-structures.
The latter approach
is
taken in \cite{scales_in_hybrid_mice_over_R}, and that approach is more convenient, as it gives us
the right
notation to prove strong condensation properties like
\cite[Lemma 4.1(***)]{scales_in_hybrid_mice_over_R}.
 If one defines strategy mice as an
instance of the
general theory here, one would then need to define new notation to refer to arbitrary $\J$-initial
segments in order to prove the analogue of
\cite[Lemma 4.1(***)]{scales_in_hybrid_mice_over_R}.
But then one might as well have defined strategy mice as in \cite{scales_in_hybrid_mice_over_R} to
begin with.} Also, there are important operators, like $x\mapsto C_\Gamma(x)$ for pointclasses $\Gamma$, which are not known to be given by $\J$-structures. We hope the work here will fuel the developments of a more general theory of operators that can accommodate those like $C_\Gamma$ in the future. ($C_\Gamma$ probably need not yield the kind of operator that is appropriate to define operator-mice as we define them. But maybe some appropriate variant can be worked out.)

The paper proceeds as follows.
In \S\ref{SigmaMiceOverR} we define precursors to $\Fop$-premice, culminating in \emph{operator
premice}. We analyse these structures and cover basic fine structure and iteration theory.
In \S\ref{sec:operators}, we introduce \emph{operators} $\Fop$, and \emph{$\Fop$-premice},
which will be instances of operator premice. We define \emph{fine condensation} for operators;
this notion is integral to the paper. We describe \emph{mouse operators} in Definition \ref{dfn:mouse_op} (a basic example of abstract operators),
and show in Proposition \ref{prop:mouse_op_con_finely} that  mouse operators  condense finely.
 We
then prove, in \ref{thm:k+1-fineness}, the main result of the paper -- that the fundamental fine
structural facts (such as solidity of the standard parameter) hold for $\Fop$-iterable
$\Fop$-premice, given that $\Fop$ condenses finely.

\subsection{Conventions and Notation}\label{sec:notation}
We work {\color{black}in} $\ZF$ throughout the paper, indicating choice assumptions where we use
them. We write $\Ord$
for the class of
ordinals. Given a transitive set $M$, $\Ord^M=\Ord(M)$ denotes $\Ord\inter M$. We
write $\card(X)$ for the
cardinality of $X$, $\pow(X)$ for the power set of $X$, and for $\theta\in\Ord$,
$\powerset({<\theta})$ is the set of bounded subsets of $\theta$ and $\her_\theta$ the set
of sets hereditarily of cardinality ${<\theta}$.
We write $f:X\partialto Y$ to denote a partial function.

We identify $[\Ord]^{<\om}$ with the strictly decreasing sequences of ordinals,
so given $p,q\in[\Ord]^{<\om}$, $p\rest i$ denotes the upper $i$ elements of $p$,
and $p\ins q$ means that $p=q\rest i$ for some $i$,
and $p\pins q$ iff $p\ins q$ but $p\neq q$.
The default ordering of $[\Ord]^{<\om}$ is lexicographic,
with $p<q$ iff $p\neq q$ and $\max(p\Delta q)\in q$.

Given a first-order structure $\M=(X,A_1,\ldots)$ with universe $X$ and predicates,
constants, etc,
$A_1,\ldots$, we write $\univ{\M}=X$.
A \dfnemph{transitive structure} is a first-order structure with
transitive universe. We sometimes blur the distinction between the terms
\emph{transitive} and
\emph{transitive structure}. For example, when we refer to a transitive structure as being
\dfnemph{rud closed}, it means that its universe is closed under rudimentary functions. For $\M$ a
transitive structure, $\OR(\M)=\OR(\univ{\M})$. An arbitrary transitive set $X$ is also considered
as the transitive structure $(X)$.
We write $\trancl(X)$ for the transitive closure of $X$.
We say that $\M$ is \emph{amenable}
if for predicate $A$ of $\M$,
we have $X\cap A\in\M$
for all $X\in\univ{\M}$.

Given a transitive structure $\M$, we write
$\J_\alpha(\M)$ for the
$\alpha^\nth$ step in Jensen's $\J$-hierarchy over $\M$ (so for example,
$\J_1(\M)$ is the rud closure of $\trancl(\{\M\})$). We similarly use $\Ss$ to
denote the function giving Jensen's more refined $\Ss$-hierarchy,
so $\Ss_\om(\M)=\J_1(\M)$.
And $\J(\M)=\J_1(\M)$.

We take (standard) \dfnemph{premice} as in \cite{steel2010outline}, except that
we allow superstrong extenders on their sequence, as discussed in Remark \ref{rem:superstrong}. Our definition and theory
of
\emph{operator premice} is mostly modelled on \cite{steel2010outline} and \cite{FSIT},
and fine structure is mostly in those papers,
but adopting the simplifications in \cite[\S5]{V=HODX}.
For discussion of generalized solidity witnesses, see \cite{zeman}.

Our notation pertaining to iteration trees is fairly standard, but here are some points. Let $\Tt$
be a putative iteration tree. We write $<^\Tt$ for the tree order of $\Tt$ and $\pred^\Tt$ for
the
$<^\Tt$-predecessor function.
Let $\alpha+1<\lh(\Tt)$ and $\beta=\pred^\Tt(\alpha+1)$.
Then $M^{*\Tt}_{\alpha+1}$ denotes the $\N\ins M^\Tt_\beta$ such that
$M^\Tt_{\alpha+1}=\Ult_n(\N,E)$, where $n=\deg^\Tt(\alpha+1)$ and $E=E^\Tt_\alpha$,
and $i^{*\Tt}_{\alpha+1}=i^{\N,n}_{E}$ denotes the corresponding ultrapower embedding. And for
$\alpha+1\leq_\Tt\gamma$, $i^{*\Tt}_{\alpha+1,\gamma}=i^\Tt_{\alpha+1,\gamma}\com
i^{*\Tt}_{\alpha+1}$. Also let $M^{*\Tt}_0=M^\Tt_0$ and $i^{*\Tt}_0=\id$.
If $\lh(\Tt)=\gamma+1$ then $M^{\Tt}_\infty=M^\Tt_\gamma$, etc,
and $b^\Tt$ denotes $[0,\gamma]_\Tt$.
For $\alpha<\lh(\Tt)$, $\base^\Tt(\alpha)$
denotes the least $\beta\leq_\Tt\alpha$ such that $(\beta,\alpha]_\Tt$ does not
drop in model or degree. (Therefore either $\beta=0$ or $\beta$ is a
successor.)

A premouse $\P$ is \dfnemph{$\eta$-sound} iff for every $n<\om$, if
$\eta<\rho_n^\P$ then $\P$ is $n$-sound, and if $\rho^\P_{n+1}\leq\eta<\rho_n^\P$ then letting
$p=p_{n+1}^\P$, $p\cut\eta$ is $(n+1)$-solid for $\P$, and $\P=\Hull_{n+1}^\P(\eta\un \{p,\pvec_n^{\P}\})$, where  $p^\P_i$ is the $i$-th standard parameter of $\P$,  $\pvec_n^{\P} =\{p^\P_1, \dots, p^\P_n\}$, and $\Hull_{n+1}$ is defined via the union of \ref{dfn:hulls} and \ref{dfn:fine_structure}.

{\color{black}\label{c21}Let $\M$ be a first order structure
and $\Gamma$ a set of formulas in the signature of $\M$. Let $X\sub\M$. Then
$\Th_\Gamma^\M(X)$ denotes the set of pairs $(\varphi,\vec{x})$
such that $\varphi\in\Gamma$,
$\vec{x}\in X^{<\om}$ and $\M\sats\varphi(\vec{x})$.}

\section{The fine structural framework}
\label{SigmaMiceOverR}

In this section, we introduce and analyse an increasingly focused sequence of
approximations to \emph{${\Fop}$-premice} (which were outlined in the introduction, but will be defined formally later). We first define \emph{hierarchical model}, which describes
the most basic structure of
${\Fop}$-premice.
We refine this by defining \emph{adequate model}, adding some
semi-fine-structural  requirements (such as \emph{acceptability}).
We then develop some basic facts regarding adequate models and their cardinal
structure. From there we can define \emph{potential
operator premouse \tu{(}potential opm\tu{)}}, which are analogous to potential premice; this definition makes new restrictions
on the information encoded by the predicates (most significantly that the predicate
$\hmEdot$ encodes extenders analogous to those of premice), and adds some pre-fine structural
requirements.
Using the latter, we can define the central fine structural concepts for
potential opms.
We then define \emph{Q-operator premouse \tu{(}Q-opm\tu{)}} by requiring that every proper segment
be fully sound,
and show that the first-order content of Q-opm-hood is \emph{almost} expressed by a
Q-formula.\footnote{As in \cite{FSIT}, we consider two cases: type 3, and non-type 3. For example,
the property of being a non-type 3 Q-opm is expressed by a Q-formula
modulo transitivity and the Pairing Axiom.}
We then define \emph{operator premouse} (analogous to \emph{premouse}).
We prove various fine structural facts regarding operator premice,
and discuss the basic iterability theory.

Later in \S\ref{sec:operators}, we will introduce \emph{operators}
${\Fop}$, and \emph{${\Fop}$-premice}. In order to motivate the language $\Ll_0$ of hierarchical models (see Definition \ref{dfn:language_Ll_0}),
we mention now the basic setup for ${\Fop}$-premice. \label{c8}In an ${\Fop}$-premouse $\M$, the predicate $\hmEdot$ will be used to
encode an extender,
$\hmPdot$ to encode auxiliary information given by ${\Fop}$ (for example, if ${\Fop}$ codes  an iteration strategy $\Sigma$
and $\T\in \M$ is a tree according to $\Sigma$, then $\dot{P}$ could code a branch $b$ of $\T$ according to $\Sigma$), $\hmSdot$ to encode the sequence of proper initial segments of $\M$, $\hmXdot$
to encode the extensions
of all (not just proper) segments of $\M$, $\hmbdot$ to refer to the coarse
\emph{base} of $\M$ (a coarse,
transitive
set at the bottom of the structure), and $\hmpdot$ to refer to a coarse
\emph{parameter}, {\color{black}which will be useful if there is some special element of the coarse base to which we want to be able to refer to directly with the language (continuing the same example of $\mathcal{F}$ coding an iteration strategy, $\hmpdot$ might specify the structure for which $\mathcal{F}$ is an iteration strategy). The choice of symbols has the following linguistic justification: \emph{E} stands for \emph{extender},
\emph{P} for \emph{predicate}, \emph{S} for \emph{segments}, \emph{X} for
\emph{extensions}, \emph{cb} for \emph{coarse base}, \emph{cp} for \emph{coarse parameter}.\label{c5} We  use
\emph{cp} \label{c9}instead of \emph{p} to avoid conflict with notation for standard
parameters. We use \emph{cb} instead of \emph{b} to avoid conflict with notation
associated to strategy mice. For better readability, we will
typically use the variable $A$ to represent $\hmb^\M$.} An
${\Fop}$-premouse $\M$ is
\emph{over} its base $A=\hmbdot^\M$.
Here $A\in\M$ and
$A$ is in all proper segments of $\M$.
When we form fine structural cores, all elements of $A\un\{A\}$ will be {\color{black}in the relevant hulls\label{comment_6}}.
But
in some contexts we will also be interested in hulls
which do not include all elements of $A$.

\subsection{Hierarchical models}

\begin{dfn}
Let $Y$ be transitive. Then \index{$\varrho_Y$}$\varrho_Y:Y\to\rank(Y)$ denotes the rank function.
And \index{$\hat{Y}$}$\hat{Y}$ denotes $\trancl(\{(Y,\om,\varrho_Y)\})$.
For $M$ transitive, we say that $M$ is \index{rank closed}\dfnemph{rank closed} iff for every $Y\in M$,
we have $\hat{Y}\in M$ and $\hat{Y}^{<\om}\in M$.\footnote{We take finite sequences
	over $Y$ as functions $f:n\to Y$ for $n<\omega$, so if $Y$ is infinite then \index{$Y^{<\om}$}$\rank(\hat{Y}^{<\omega})<\rank(Y)+\om$.} Note that if $M$ is rud closed and rank closed
then $\rank(M)=\Ord\inter M$.
\end{dfn}

\begin{dfn}[Hulls]\index{hull}\label{dfn:hulls} Let $\Ll=\{\dot{B},\Pvec,\cvec\}$ be a finite first-order language,\footnote{\color{black}We include an equality symbol in all first-order languages by default,
interpreted as true equality.} where
$\dot{B}$ is a binary predicate,
$\Pvec=\left<\dot{P}_i\right>_{i<m}$ is a tuple of unary predicates and
$\cvec=\left<\dot{c}_i\right>_{i<n}$ a tuple of constants.
Let $\N$ be a first-order $\Ll$-structure
and $B=\dot{B}^\N$, etc.
Let $\Gamma$ be a collection of $\Ll$-formulas with ``$x=\dot{c}_i$'' in $\Gamma$ for each $i<n$.
Let $X\sub\univ{\N}$. Then\index{$\Hull^\N_\Gamma(X)$}
\[ \Hull^\N_\Gamma(X)\eqdef(H,B\inter H^2,P_0\inter H,\ldots,P_{m-1}\inter
H,c_0,\ldots,c_{n-1}), \] where $H$ is the set of all $y\in\univ{\N}$ such that for some
$\varphi\in\Gamma$ and $\xvec\in
X^{<\om}$, $y$ is the unique $y'\in\N$ such that $\N\sats\varphi(\xvec,y')$.
If $\N$ is transitive {\color{black}and $H$ is extensional,\label{c7}} then
\index{$\cHull^\N_\Gamma(X)$}$\C=\cHull^\N_\Gamma(X)$
denotes the $\L$ structure which is the transitive collapse of $\Hull^\N_\Gamma(X)$. (That is,
$\univ{\C}$ is the transitive collapse of $H$, and letting $\pi:\univ{\C}\to H$ be the uncollapse,
$P_i^\C=\pi^{-1}``P_i$, etc.)
\end{dfn}

\begin{dfn}\label{dfn:language_Ll_0}
Let \index{$\Ll_0$}$\Ll_0$ be the language of set theory augmented with unary predicate symbols
$\hmEdot$, $\hmPdot$, $\hmSdot$, $\hmXdot$, and constant symbols $\hmbdot$,
$\hmpdot$. Let \index{$\Ll_0^+$}$\Ll_0^+$ be $\Ll_0$ augmented with constant symbols $\hmmudot$, $\hmedot$.\footnote{$\mu$ is for
	\emph{measurable}, and will represent the critical point of an active extender,
	and \emph{e} is for \emph{extender}, and will represent the largest witness to the Initial Segment Condition for a type 2 active extender.} Let
\index{$\Ll_0^-$}$\Ll_0^-=\Ll_0\cut\{\hmEdot,\hmPdot\}$.
\end{dfn}

\begin{dfn}\label{dfn:hm}
\label{model}
A \index{hierarchical model}\index{hm}\dfnemph{hierarchical model}
is an $\Ll_0$-structure
\[ \M = (\univ{\M};{\color{black}{\in}\rest\univ{\M}^2}, \hmE, \hmP, \hmS, \hmX, \hmbv, \hmpv),\]
where {\color{black}$\dot{\in}^{\M}={\in}\rest\univ{\M}^2$,} $\hmEdot^\M=\hmE$, etc, $\hmbv=\hmbdot^\M$ and $\hmpv=\hmpdot^\M$, and such that for some
ordinal $\lambda>0$, the following conditions hold:
\begin{enumerate}
\item \index{coarse base}\index{coarse parameter} (Base, Parameter) $\hmbv=\hat{Y}$ for some transitive $Y$, and $\hmpv\in\J(\hmbv)$; we say
that
$\M$ is \dfnemph{over} the
\dfnemph{\tu{(}coarse\tu{)} base} $\hmbv$
and has \dfnemph{\tu{(}coarse\tu{)} parameter} $\hmpv$.

 \item (Segments) \index{$\hmS$}$\hmS = \left<\hmS_\xi\right>_{\xi<\lambda}$ where $\hmS_0=\hmbv$ and for each
$\xi\in[1,\lambda)$, {\color{black}$\hmS_\xi$ is
a  $\Ll_0$-structure
with $\hmbdot^{\hmS_\xi}=\hmbv$,
 $\hmpdot^{\hmS_\xi}=\hmpv$, and
$\hmSdot^{\hmS_\xi} = \hmS\rest \xi$.
Write $\hmS_\lambda=\M$.}
  \item For each $\xi\in[1,\lambda]$,
  $\univ{\hmS_\xi}$ is transitive, rud closed and
rank closed,
and $\hmS_\xi$ is amenable
(note that this includes in particular $\M=\hmS_\lambda$).\footnote{\color{black}Note that it follows
that $\hmS_\alpha\in\univ{\hmS_\beta}$  and $\univ{\hmS_\alpha}\psub\univ{\hmS_\beta}$ for all $\alpha<\beta\leq\lambda$.}

\item\label{item:limit_union} (Continuity)  $\univ{\hmS_\xi}=\bigcup_{\alpha<\xi}\univ{\hmS_\alpha}$
for each limit $\xi\leq\lambda$.
\item\label{item:extensions} (Extensions) \index{$\hmX$}$\hmX^{\hmS_\xi}\colon\univ{\hmS_\xi}\to\xi$, and $\hmX^{\hmS_\xi}(x)$ is the
least $\alpha$ such that $x\in\univ{\hmS_{\alpha+1}}$.
\end{enumerate}
Let \index{$\wl(\M)$}\index{length}$\wl(\M)$ denote
$\lambda$, the \dfnemph{length} of $\M$.
For $\alpha\leq\lambda$ let \index{$\M\wupto\alpha$}$\M\wupto\alpha=\hmS_\alpha$.
A hierarchical model $\M$ is a \index{successor}\dfnemph{successor} iff $\wl(\M)$ is a successor ordinal
$\xi+1$; in this case let \index{$\M^-$}$\M^-=\M\wupto\xi$. If $\wl(\M)$ is a limit ordinal,
let \index{$\M^-$}$\M^-=\M$. We say that $\N$ is an
\index{initial segment (of hm)}\dfnemph{\tu{(}initial\tu{)} segment} of $\M$,
and write \index{$\ins$}$\N\wins\M$, iff
$\N=\M\wupto\alpha$ for some $\alpha\in[1,\lambda]$, and say that $\N$ is a \index{proper segment (of hm)}\dfnemph{proper
\tu{(}initial\tu{)} segment} of $\M$, and write \index{$\wpins$}$\N\wpins\M$, iff $\N\wins\M$ and $\N\neq\M$.
(Note that $\M\wupto0=\hmbv\nwins\M$.) We write $E^\M=E$, etc.  For any transitive $Y$, let
\index{$\hmb^{\hat{Y}}$}$\hmb^{\hat{Y}}=\hat{Y}$; so $\hmb^{\M\wupto\alpha}=\M\wupto0$ for all $\alpha$.\footnote{\color{black}That is,
we have $\hmb^{\M\wupto\alpha}=b=S_0=\M\wupto 0$
for all $\alpha\in(0,\lambda]$ by definition. But recall that $b=\hat{Y}$ for some transitive $Y$,
so $\hmb^{\M\wupto 0}=\hmb^{b}=\hmb^{\hat{Y}}=\hat{Y}=b=\M\wupto 0$.}
\end{dfn}

{\color{black}The first observation follows easily from the definition:
\begin{lem}\label{lem:ins_of_hm}
Let $\M$ be a hierarchical model and $\N\ins\M$. Then $\N$ is a hierarchical model.
\end{lem}}
\begin{rem}\label{rem:def_over_core_0(M)}
For the most part, definability over hierarchical models $\M$ will literally be computed over $\core_0(\M)$ (to be defined later), which will be an $\Ll_0^+$-structure. But for successors $\M$, we will have $\core_0(\M)=(\M,\hmmudot^{\core_0(\M)},\hmedot^{\core_0(\M)})$ and $\hmmudot^{\core_0(\M)}=\emptyset=\hmedot^{\core_0(\M)}$. So in this case, definability over
$\M$ (using $\Ll_0$) will be equivalent to that over $\core_0(\M)$ (using $\Ll_0^+$).
\end{rem}
\begin{dfn}\label{dfn:adequate}
Let $\M$ be a hierarchical model over $A$.

If $\M$ is a successor,
then for $p\in[\OR^\M]^{<\om}$, we say that $\M$ is \index{solid}\dfnemph{$(1,p)$-solid}
iff for every $\alpha\in p$,
we have\footnote{Clearly this implies that $\Th_{\Sigma_1}^\M(X)\in\M$ also, 
where $X=A\cup\alpha\cup(p\cut(\alpha+1))$. 
	Recall that for standard premice $\M$, when defining the solidity of $p_{1}^\M$, it does not matter whether we demand that the relevant $\Sigma_{1}$-hulls are in $\M$,
    or their corresponding $\Sigma_{1}$-theories are in $\M$; the two requirements are equivalent.
	But this does not seem clear for the structures we consider.
	It is important that we use the stronger condition.}
\[ \Hull_1^{\M}\big(A\cup\alpha\cup(p\cut(\alpha+1))\big)\preccurlyeq_1\M \]
and
\[\cHull_1^{\M}\big(A\cup\alpha\cup(p\cut(\alpha+1))\big)\in\M.\]
Note that it follows that
$\Th_{\Sigma_1}^\M\big(A\cup\alpha\cup(p\cut(\alpha+1))\big)\in\M$, and
note that the $\Sigma_1$-elementarity ensures that the uncollapsed hull is extensional (in fact $\Sigma_0$ suffices),
and hence the transitive collapse $\cHull$ is well-defined.

We say that $\M$ is \index{soundly projecting}\dfnemph{soundly projecting} iff
for every successor $\N\wins\M$, there is $p\in[\OR(\N)]^{<\om}$ such that $\N$ is $(1,p)$-solid
and
\[ \N=\Hull_{\Sigma_{1}}^{\N}(\N^-\un\{\N^-, p\}). \]

We say that $\M$ is \index{acceptable}\dfnemph{acceptable} iff
for every successor $\N\wins\M$,
for every $\tau\in\OR(\N^-)$,
if there is some $X\in\pow(A^{<\om}\cross\tau^{<\om})$ such that $X\in\N\cut
\N^-$ then in $\N$ there is a map $A^{<\om}\cross\tau^{<\om}\onto\N^-$.

We say that $\M$ is an \index{adequate}\dfnemph{adequate model} iff $\M$ an acceptable hierarchical model
and every \emph{proper} segment of $\M$ is soundly projecting.

An \index{model-plus}\dfnemph{adequate model-plus} is an $\Ll_0^+$-structure $\M$ such that the $\Ll_0$-reduct of $\M$ is an
adequate model.
\end{dfn}

{\color{black}In the end we will be primarily interested in structures for which every initial segment is soundly projecting, not just the proper segments. For certain kinds of operators
$\Fop$, such as the usual operators used to encode an iteration strategy in a hybrid mouse or strategy mouse, the successor structures $\N$ produced (as $\Fop$-premice) 
will in fact have the stronger property
that $\N=\Hull_{\Sigma_1}^\N(\N^-\cup\{\N^-\})$.
Of course, this is the case when $\Fop$ is the usual $\J$-operator.
For mouse operators $\Fop$, $\N=\Fop(\N^-)$
will be equivalent to a sound mouse over $\N^-$ which projects to $\N^-$.  By coding that mouse via its $n$th reduct for the relevant $n$ (with $\rho_{n+1}^\N\leq\OR^{\N^-}<\rho_n^{\N^-}$),
we will get a structure which is soundly projecting, with the $p$ of the definition being $p_{n+1}^\N$.
}

{\color{black}As in \cite{FSIT}, etc, it is useful to consider
what can be expressed with \emph{Q-formulas} and variants thereof,
as they are preserved well downward under $\Sigma_1$-elementary maps, and upward under ultrapower embeddings:}

\begin{dfn}\label{dfn:Ll_0-Q-formula} Given a language $\Ll$ extending the language of set
theory, an \index{Q-formula}\dfnemph{$\Ll$-simple-Q-formula} is a formula of the form
\[ \varphi(v_0,\ldots,v_{n-1})\ \iff\ \all x\ex y[x\sub y\ \&\ \psi(y,v_0,\ldots,v_{n-1})], \]
for some $\Sigma_1$ formula $\psi$ of $\Ll$. (Here all free variables are displayed; hence, $x$ is
not free in $\psi$.)

Let $\varphi_{\mathrm{pair}}$ be the Pairing Axiom.
\end{dfn}

It is easy to see that neither $\varphi_{\mathrm{pair}}$, nor rud closure, can be expressed, modulo
transitivity, by a simple-Q-formula.\footnote{If $\Ll$ is a first-order language extending the
language of
set theory, and $X,Y$ are rud closed transitive $\Ll$-structures such that $c^X=c^Y$ for each
constant symbol $c\in\Ll$,
and $P^X=P^Y$ for each predicate symbol $P\in\Ll$ with
$P\neq\dot{\in}$, then any $\Ll_0$-Q-formula true in both
$X,Y$ is also true in the ``union'' of $X,Y$.} However:

\begin{lem}\label{lem:adequate model_almost_Q}
There is an $\Ll_0$-simple-Q-formula $\varphi_{\mathrm{am}}$ such that for all transitive
$\Ll_0$-structures
$\M$, $\M$ is an adequate model iff $\M\sats[\varphi_{\mathrm{pair}}\ \&\
\varphi_{\mathrm{am}}]$.
\end{lem}
\begin{proof}[Proof Sketch]
 This is a routine calculation, which we omit. (First find an $\Ll_0$-Q-formula
$\varphi_{\mathrm{rud}}$ such that $\varphi_{\mathrm{pair}}\wedge \varphi_{\mathrm{rud}}$
expresses rud closure;
this uses the the finite basis for rud functions.)
\end{proof}

If $\M$ is an adequate model over $A$ and $\xi<\wl(\M)$ then $\M$ has a map
\[ A^{<\om}\cross\xi^{<\om}\onto\M\wupto\xi.\]
In fact, by the following lemma, this is true uniformly.

\begin{lem}\label{lem:canonical_surjection}
There is a $\Sigma_1$ formula $\psi$ of $\Ll_0^-$, of two free variables,
such that for all $A$
and adequate models $\M$ over $A$,
$\psi$ defines a map $F:\wl(\M)\to\M$, and for $\xi<\wl(\M)$, letting $h_\xi=F(\xi)$,
we have
\[ h_{\xi}:A^{<\om}\cross\xi^{<\om}\onto\M\wupto\xi\]
and for all
$\alpha\leq\xi$, we have $h_\alpha\sub
h_\xi$.
\end{lem}
{\color{black}\begin{proof}\label{c11}
The proof is
quite routine, using the sound-projection of proper segments of $\M$, much like in the proof
of the
corresponding fact for $L$. At the referee's request, we provide details. We will define  $h_\xi$ for $\xi<\wl(\M)$, by recursion on $\xi$. We leave it to the reader to see that the definitions are sufficiently uniform and local that one can write down a $\Sigma_1$ formula $\psi$ witnessing the lemma.

Recall that $\M|0=A$. Note $0^{<\om}=\{\emptyset\}$. We define $h_0: A^{<\om}\cross 0^{<\om}\onto A$ by setting $h_0(\vec{x},\emptyset)=\vec{x}(0)$, in case $\lh(\vec{x})>0$,
and $h_0(\emptyset,\emptyset)=\emptyset$.

Given a soundly projecting successor $\N$, let $g^{\N}$
be the least $g\in[\OR(\N)]^{<\om}$
such that $\N=\Hull_1^{\N}(\N^-\cup\{\N^-,g\})$.

Now suppose $1<\wl(\M)$.
We define $h_1:A^{<\om}\cross 1^{<\om}\onto\M|1$.
Set $h_1(\vec{x},\emptyset)=h_0(\vec{x},\emptyset)$,
so $h_0\sub h_1$.
For $k<\om$,
let $\vec{0}_k$ denote the sequence $(0,\ldots,0)$ of $0$'s of length $k$. So $1^{<\om}=\{\vec{0}_k\bigm|k<\om\}$.
Let us now define $h_1(\vec{x},\vec{0}_{n+1})$ for $n<\om$.
Let $\varphi$ be a  $\Sigma_1$ formula
with free variables
exactly $v_0,\ldots,v_{k+2}$, where $k<\om$. Let $n$ be the G\"odel number of $\varphi$,
and $\vec{x}\in A^k$.
If there is a unique $y\in\M|1$
such that $\M\sats\varphi(\vec{x},A,g^{\M|1},y)$, then
 we define $h_1(\vec{x},\vec{0}_{n+1})=$ that unique $y$;
 otherwise define $h_1(\vec{x},\vec{0}_{n+1})=\emptyset$.
 By the definition of $g^{\M|1}$
 and since $A=(\M|1)^-$,
 $h_1:A^{<\om}\cross 1^{<\om}\onto\M|1$,
 and since $h_1$ is definable (in fact without parameters) over $\M|1$, we have $h_1\in\M$.

 Now suppose  $0<\gamma<\gamma+1<\wl(\M)$
 and we have defined $h_\gamma:A^{<\om}\cross \gamma^{<\om}\onto\M|\gamma$.
 We define $h_{\gamma+1}:A^{<\om}\cross (\gamma+1)^{<\om}\onto\M|(\gamma+1)$.
We start by setting $h_\gamma\sub h_{\gamma+1}$. It remains to define $h(\vec{x},\vec{\alpha})$ in case $\vec{\alpha}\in(\gamma+1)^{<\om}\cut \gamma^{<\om}$. Here we will use $(\vec{x},\vec{\alpha})$ to determine some formula $\varphi$ (with G\"odel code $n$)
and some elements $y_0,\ldots,y_{m-1}\in\M|\gamma$,
and use these data, along with the parameters $\M|\gamma$ and $g^{\M|(\gamma+1)}$, to attempt to define
some $y\in\M|(\gamma+1)$;
if this attempt is successful,
we will set $h_{\gamma+1}(\vec{x},\vec{\alpha})=y$. Here we use $y_i=h_\gamma(\vec{x}_i,\vec{\ell}_i)$, where $\vec{x}_i$ is a certain substring of $\vec{x}$, and $\vec{\ell}_i$ a certain substring of $\vec{\alpha}$, determined as follows.
Suppose $\vec{\alpha}$
has form
\begin{equation}\label{eqn:code_seq} (\gamma,\vec{0}_{n},1,\vec{0}_{j_0},1,\vec{0}_{k_0},1,\vec{\ell}_0,\vec{0}_{j_1},1,\vec{0}_{k_1},1,\vec{\ell}_1,\ldots,\vec{0}_{j_{m-1}},1,\vec{0}_{k_{m-1}},1,\vec{\ell}_{m-1}) \end{equation}
where $\vec{\ell}_{i}$ has length
$\lh(\vec{\ell}_i)=k_i$ for each $i<m$.\footnote{The sequence in line (\ref{eqn:code_seq}) has length $1+n+1+(j_0+1+k_0+1+\lh(\vec{\ell}_0))+\ldots+(j_{m-1}+1+k_{m-1}+1+\lh(\vec{\ell}_{m-1}))$.
The first entry is $\gamma$, followed by $n$-many $0$s, one $1$,  $j_0$-many $0$s,  one $1$,  $k_0$-many $0$s,  one $1$, then $\lh(\vec{\ell}_0)$-many ordinals $\leq\gamma$, etc.} Note that  any such sequence is uniquely readable, in that the form above is uniquely determined by the sequence. Suppose that $\vec{\ell}_i\in \gamma^{<\om}$ for each $i<m$. Let $\varphi$ be the $\Sigma_1$ formula with G\"odel code $n$. Suppose that the free variables of $\varphi$ are exactly $v_0,\ldots,v_{m+2}$. Let $\vec{x}\in A^{<\om}$
have length $j_0+j_1+\ldots+j_{m-1}$, and write
\[ \vec{x}=\vec{x}_0\conc\ldots\conc\vec{x}_{m-1} \]
where $\lh(\vec{x}_i)=j_i$.
Let $y_i=h_\gamma(\vec{x}_i,\vec{\ell}_i)$, so $y_i\in\M|\gamma$. If there is a unique $y\in\M|(\gamma+1)$
such that
\[ \M|(\gamma+1)\sats\varphi(y_0,\ldots,y_{m-1},\M|\gamma,g^{\M|(\gamma+1)},y),\]
then define $h_{\gamma+1}(\vec{x},\vec{\alpha})=y$,
and otherwise define $h_{\gamma+1}(\vec{x},\vec{\alpha})=\emptyset$.
For all other $(\vec{x},\vec{\alpha})$, define $h_{\gamma+1}(\vec{x},\vec{\alpha})=\emptyset$.
Then $h_{\gamma+1}$ is surjective and definable over $\M|(\gamma+1)$, so $h_{\gamma+1}\in\M$.

Given $h_\alpha$
for all $\alpha<\xi$ where $\xi<\wl(\M)$ is a limit, (we must)
set $h_\xi=\bigcup_{\alpha<\xi}h_\alpha$. By the uniformity of the definitions, $h_\xi$ is $\Sigma_1^{\M|\xi}$,
so $h_\xi\in\M$.
\end{proof}}
\begin{dfn}\label{dfn:h^M}
 Given an adequate model $\M$ over $A$ and $\xi<\wl(\M)$, let
\index{$h^\M$} $h^\M_\xi$ be the function $h_\xi$ of the preceding lemma.
 Let $h^\M=\bigcup_{\xi<\wl(\M)}h^\M_\xi$.
\end{dfn}

\begin{rem}\label{rem:h^M} So $h^\M$ is $\Sigma_1^{\M}$ via a formula in $\Ll_0^-$, uniformly in adequate $\M$,
and
\[ h^\M:A^{<\om}\cross\wl(\M^-)^{<\om}\onto\M^- \]
(recall that if $\M$ is a limit then $\M^-=\M$), and if $\M$ is a successor then
$h^\M\in\M$.
\end{rem}

{\color{black}We now want to analyse somewhat the cardinal structure
of adequate models. This will be useful when we come to defining \emph{potential operator-premice}, in particular regarding the properties of extenders on their sequence.}
\begin{dfn}\label{dfn:cardinals}
 Let $\M$ be an adequate model over $A$ and $\lambda=\wl(\M)$.
Let $\rho<\OR(\M)$. Then $\rho$ is an \index{$A$-cardinal}\dfnemph{$A$-cardinal} of $\M$ iff $\M$
has no map $A^{<\om}\cross\gamma^{<\om}\onto\rho$ where $\gamma<\rho$.
We let \index{$\Theta^\M$}$\Theta^\M$ denote the least $A$-cardinal of $\M$, if such exists.
We say that $\rho$ is \index{$A$-regular}\dfnemph{$A$-regular} in $\M$ iff $\M$ has no map
$A^{<\om}\cross\gamma^{<\om}\cofto\rho$ where $\gamma<\rho$.
We say that $\rho$ is an \index{ordinal-cardinal}\dfnemph{ordinal-cardinal} of $\M$ iff $\M$ has
no map $\gamma^{<\om}\onto\rho$ where $\gamma<\rho$.
We say that $\rho$ is \index{relevant}\dfnemph{relevant} iff $\rho\leq\OR(\M^-)$.
\end{dfn}

\begin{lem}
 Let $\M$ be an adequate model over $A$ and $\lambda=\wl(\M)>\xi>0$.
 Let $\kappa$ be an $A$-cardinal of $\M$ such that $\kappa\leq\OR(\M\wupto\xi)$.
 Then $\rank(A)<\kappa\leq\xi$ and $\kappa=\OR(\M\wupto\kappa)$.
\end{lem}

\begin{lem}\label{lem:H_Theta^M_etc}
 There is a $\Sigma_1$ formula $\varphi$ in $\Ll_0^-$ such that,
 for any $A$ and adequate model $\M$ over $A$, if $\Theta=\Theta^\M$ exists and is relevant then:
 \begin{enumerate}
  \item\label{item:power(a)_unbounded} $\Theta$ is the least
$\alpha$ such that $\pow(A^{<\om})^\M\sub\M\wupto\alpha$.
  \item\label{item:M_Theta_is_H_Theta} $\univ{\M\wupto\Theta}$ is the set of all $x\in\M$ such that
$\trancl(x)$ is the surjective image of $A^{<\om}$ in $\M$.
  \item\label{item:uniform_surjections} Over $\M\wupto\Theta$, $\varphi(0,\cdot,\cdot)$ defines
a function $G:\Theta\to\M\wupto\Theta$ such that for all $\alpha<\Theta$,
we have $G(\alpha):A^{<\om}\onto\M\wupto\alpha$.
 \item\label{item:Theta_A-regular} $\Theta$ is $A$-regular in $\M$.
\end{enumerate}

 Let $\kappa_0<\kappa_1$ be consecutive relevant $A$-cardinals of $\M$. Then:
 \begin{enumerate}[resume*]
\item $\kappa_1$ is the least $\alpha$ such that
$\pow(A^{<\om}\cross\kappa_0^{<\om})^\M\sub\M\wupto\alpha$.
\item $\univ{\M\wupto\kappa_1}$ is the set of all $x\in\M$ such that $\trancl(x)$ is the surjective
image of $A^{<\om}\cross\kappa_0^{<\om}$ in $\M$.
\item Over $\M\wupto\kappa_1$,
$\varphi(\kappa_0,\cdot,\cdot)$ defines a map $G:\kappa_1\to\M\wupto\kappa_1$ such that for all
$\alpha<\kappa_1$,
we have $G(\alpha):A^{<\om}\cross\kappa_0^{<\om}\onto\M\wupto\alpha$.
\item $\kappa_1$ is $A$-regular in $\M$.
\end{enumerate}
\end{lem}
\begin{proof}
	We just prove parts \ref{item:power(a)_unbounded}--\ref{item:Theta_A-regular}; the others are
	similar.
	Let $\gamma\in[1,\wl(\M)]$ be least such that $\pow(A^{<\om})\inter\M\sub\M|\gamma$. For part \ref{item:power(a)_unbounded},
	we must see that $\gamma=\Theta$.

	Let us first observe that $\gamma$ is a limit ordinal.
	{\color{black}Suppose $\gamma=\xi+1$ for some $\xi$. By acceptability,
	there is a surjection $\pi:A^{<\om}\to\M|\xi$ with $\pi\in\M|\gamma$. So $\OR(\M|\xi)<\Theta$. So if $\gamma=\wl(\M)$ then $\OR(\M^-)=\OR(\M|\xi)<\Theta$, contradicting the assumption that $\Theta$ is relevant.  So $\gamma<\wl(\M)$. But then because $\M$ is adequate, $\M|\gamma$ is soundly projecting, so there is a surjection $\pi':(\M|\xi)^{<\om}\to\M|\gamma$ with $\pi'\in\M$,
	so $\OR(\M|\gamma)<\Theta$.
	But then the usual diagonalization gives a contradiction.}

	Now by acceptability, for every $\alpha<\gamma$,
	$\M|\gamma$ has a map $A^{<\om}\onto\M|\alpha$.

	We now claim that $\gamma=\Theta$.
	For $\gamma\leq\Theta$ by acceptability.
	So suppose $\gamma<\Theta$, and let $g:A^{<\om}\onto\gamma^{<\om}$ be in $\M$.
	Let $h=h^{\M|\gamma}$. Then because $g,h\in\M$, clearly $\M$ has a map
	$f:A^{<\om}\onto\M|\gamma$, so $\M$ has a map
	$A^{<\om}\onto\pow(A^{<\om})^\M$, again a contradiction.

	So $\gamma=\Theta$, giving part \ref{item:power(a)_unbounded}.
	{\color{black} Part \ref{item:M_Theta_is_H_Theta}:
As mentioned above, for every $\alpha<\Theta$, $\M|\Theta$
	has a surjection $A^{<\om}\onto\M|\alpha$. So letting $Y\in\M$ be transitive and $\pi:A^{<\om}\onto Y$ with $\pi\in\M$, it suffices to see that $Y\in\M|\Theta$.
	Let $X\sub A^{<\om}$
	be the code for $Y$ determined by $\pi$. Then $X\in\M|\gamma=\M|\Theta$. But then $Y\in L_\kappa(X)$ where $\kappa$ is least such that $L_{\kappa}(X)$ is admissible, and note that $\kappa<\Theta$, so $Y\in\M|\Theta$.}

	Part \ref{item:uniform_surjections}: Let $\alpha<\Theta$.
	We will define
	$g:A^{<\om}\cross A^{<\om}\onto\M|\alpha$,
	and the uniformity in the definition will yield the result.
	Let $\beta\in[\alpha,\Theta)$ be least such that
	\[ \pow(A^{<\om})\inter\M|\beta\not\sub\M|\alpha.\]
	Let $h=h^{\M|\beta}$. Let $x\in A^{<\om}$ be such that for some $y$, $f=h(x,y)$ is
	such that $f:A^{<\om}\to\M|\alpha$ is a surjection (such $x$ exists by acceptability). Let $y_x$
	be the least such $y$,
	and $f_x=h(x,y_x)$. For all such $x$ and for $z\in A^{<\om}$,
	define $g(x,z)=f_x(z)$. For all other $(x,z)$, set $g(x,z)=\emptyset$. This completes the definition
	of $g$, and the uniformity is clear.

	Part \ref{item:Theta_A-regular} now follows.
\end{proof}
\begin{cor}
 Let $\M$ be an adequate model over $A$ and let $\gamma$ be a relevant $A$-cardinal of $\M$.
 If $\gamma$ is a limit of $A$-cardinals of $\M$ then
$\M\wupto\gamma$ satisfies Separation and Power Set.
If $\gamma$ is not a limit of $A$-cardinals of $\M$
then $\M\wupto\gamma\sats\ZFmin$. In particular, $\M\wupto\Theta^\M\sats\ZFmin$.
\end{cor}

\begin{lem}\label{lem:cardinals_equiv}
Let $\M$ be an adequate model over $A$ such that $\Theta^\M$ exists and is relevant.
Let $\kappa\in[\Theta^\M,\OR(\M))$ be relevant.
Then $\kappa$ is an $A$-cardinal of $\M$ iff $\kappa$ is an ordinal-cardinal of $\M$.
\end{lem}
\begin{proof}
	Suppose $\kappa>\Theta=\Theta^\M$ and $\kappa$ is an ordinal-cardinal,
	but $\M$ has a map
	\[ f:A^{<\om}\cross\gamma^{<\om}\onto\kappa \] where $\gamma<\kappa$.
	For each $y\in\gamma^{<\om}$, let $f_y:A^{<\om}\to\kappa$ be $f_y(x)=f(x,y)$,
	and let $g_y$ be the norm on $A^{<\om}$ associated to $f_y$
	(that is, $g_y:A^{<\om}\to\OR$,
	$\rg(g_y)$ is an ordinal, and $g_y(x)<g_y(x')$ iff $f_y(x)<f_y(x')$).
	Then $g_y\in\M$ and $\rg(g_y)<\Theta$, because the associated prewellorder
	on
	$A^{<\om}$ is in $\M|\Theta$ and $\M|\Theta\sats\ZFmin$. Similarly,
	the
	function $y\mapsto(f_y,g_y)$ is in $\M$.
	Let
	\[ h:\Theta\cross\gamma^{<\om}\onto\kappa \] be as follows.
	Let $(\alpha,y)\in\Theta\cross\gamma^{<\om}$. If $\alpha\notin\rg(g_y)$ then $h(\alpha,y)=0$;
	otherwise $h(\alpha,y)=f(x,y)$ where $g_y(x)=\alpha$. Then $h\in\M$, a
	contradiction.
\end{proof}

\begin{dfn}\label{rem:cardinal_successors}
 Let $\M$ be an adequate model over $A$ and let $\kappa<\OR(\M)$.
Then \index{$\kappa^{+\M}$}$\kappa^{+\M}$ denotes either the least ordinal-cardinal $\gamma$ of
$\M$ such that $\gamma>\kappa$, if there is such, and denotes $\OR(\M)$ otherwise.
By \ref{lem:cardinals_equiv}, if $\M$ is a limit and $\Theta^\M\leq\kappa$, then
$\kappa^{+\M}$
is the least $A$-cardinal $\gamma$ of $\M$ such that $\gamma>\kappa$,
if there is such, or is $\OR(\M)$ otherwise. This applies when $\hmE^\N\neq\emptyset$
in \ref{dfn:potential opm} below.
\end{dfn}

\begin{dfn}
 Let $\M$ be an adequate model over $A$. Then \index{$\rho^\M$}$\rho^\M$ denotes the least $\rho\in\Ord$ such
that $\rho\geq\om$ and $\pow(A^{<\om}\cross\rho^{<\om})\inter\J(\M)\not\sub\M$.
\end{dfn}

\subsection{Potential operator-premice}
\begin{rem}\label{rem:popm_preview}
 We now proceed to the definition of \emph{potential operator-premouse}. {\color{black}This will lay out the main first order properties we demand of $\Fop$-premice. The properties for segments with an active extender are very close to those for standard premice as in \cite{steel2010outline}, generalized to  allow superstrong extenders. The properties for successor levels are new, and they consist of four clauses. Let us first give some
motivation for these. \emph{Projectum amenability}
generalizes the fact that in an ordinary premouse $\N$, if $\M\pins\N$ then there are no new bounded subsets of $\rho_\om^\M$ which are in $\J(\M)$.} It ensures that we record all essential segments of a
potential operator-premouse
$\N$ in its history $\hmS^\N$. For example, suppose we are forming an $n$-maximal iteration tree
and
we wish to apply an extender $E$ to some piece of $\N$, but $E$ is not $\N$-total.
Projectum amenability will ensure that there is some
$\M\wpins\N$ such that $E$ is $\M$-total and $\M$ projects to $\crit(E)$.
The property of \emph{$\Sigma_1$-ordinal-generation} is used in making sense of fine structure; it
ensures for
example
that
the 1st standard parameter $p_1$ is well-defined. The \emph{stratification} of $\N$ {\color{black}lets us define $\Sigma_1$-Skolem functions in the manner usual for $\J$-structures, thereby ensuring that $\Hull_{\Sigma_1}^\N(\hmb^\N\un Y)\elem_1\N$ for any $Y\sub\N$,} and also allows us to
establish facts regarding the preservation of fine structure (including the preservation of $p_1$,
assuming $1$-solidity) under degree $0$ ultrapower maps.
 And
the existence of
\emph{$\hmb^\N$-ordinal-surjections}, together with stratification, will be used in proving that
$\Sigma_1$-ordinal-generation is propagated under degree $0$ ultrapower maps.
\end{rem}

\begin{dfn}\label{dfn:potential opm}We say that $\N$ is a \index{potential operator-premouse (opm)}\dfnemph{potential
operator-premouse}
\dfnemph{\tu{(}potential opm\tu{)}}
iff $\N$ is an adequate model, over $A$, such that
for every $\M\wins\N$,
\begin{enumerate}
\item ($\hmP$-goodness) If $\hmP^\M\neq\emptyset$ then $\M$ is a successor and
$\hmP^\M\sub\M\cut\M^-$.\footnote{The requirement that $\hmP^\M\sub\M\cut\M^-$ does not restrict
the information that can be encoded in $\hmP^\M$, because given any $X\sub\M$, one can
always replace it with $\{\M^-\}\cross X$.}
\item\label{item:E-goodness} {\color{black}($\hmE$-goodness) If
$\hmE^\M\neq\emptyset$ then
$\M$ is a limit and letting $F=\bigcup\hmE^\M$,
then $F$ is an extender over $\M$ such that, letting $\hmS=\hmS^\M$,
$\hmE=\hmE^\M$ and $\kappa=\crit(F)$:
\begin{enumerate}
\item $F$ is $A^{<\om}\times\gamma^{<\om}$-complete for all $\gamma < \kappa$.
\item The (following) premouse axioms hold for
$(\univ{\M},\hmS,\hmE)$:
letting $U=\Ult_0(\M|\kappa^{+\M},F)$ and $\nu=\nu(F)$, we have:
\begin{enumerate}
\item (Mitchell-Steel indexing)
$\OR^{\M}=\nu^{+U}$.
\item (Coherence) $S=S^{U|\OR(\M)}$.
\item (Initial Segment Condition)
For every $\eta$ such that $\kappa^{+\M}<\eta<\nu$ and $\eta=\nu(F\rest\eta)$ (that is,
either $\eta=\gamma+1$ for some generator of $F$, or $\eta$ is a limit of generators of $F$)
and $F\rest\eta$ is not of type Z, either:
\begin{enumerate}
\item there is $\P\pins\M$ such that $\bigcup E^{\P}$ is the trivial completion of $F\rest\eta$, or
\item $E^{\M|\eta}\neq\emptyset$
and there is $\P\pins\Ult_0(\M|\eta,F')$,
where $F'=\bigcup E^{\M|\eta}$,
such that $\bigcup E^{\P}$ is the trivial completion of $F\rest\eta$.
\end{enumerate}
\item (Amenable encoding) $\hmE$ is the
amenable code for $F$;
that is,
$\hmE$ is the set of all $e\in\M$
such that for some $\xi\in(\kappa,\kappa^{+\M})$, we have  $e=F\cap((\M|\xi)\cross[\nu]^{<\om})$.
\end{enumerate}
\end{enumerate}
(It follows that $\M$ has a largest cardinal
$\delta$, and $\delta\leq i_F(\kappa)$, and $\delta\leq\nu<\OR(\M)=\delta^{+U}$.)}
\item\label{item:successor_potential opm} If $\M$ is a successor then:
\begin{enumerate}
 \item\label{item:no_early_proj} \index{projectum amenability}(Projectum amenability) If $\wl(\M)>1$ and
$\om,\alpha<\rho^{\M^-}$ then
\[ \pow(A^{<\om}\cross\alpha^{<\om})\inter\M\sub\M^-.\]
 \item\label{item:surjections} \index{$A$-ordinal-surjections}($A$-ordinal-surjections) For every $x\in\M$ there is $\alpha<\OR^\M$
{\color{black}and a map\label{c14}}
$A^{<\om}\cross\alpha^{<\om}\onto x$ in $\M$.
\item\label{item:M=Hull_1^M(OR)} \index{$\Sigma_1$-ordinal-generation}($\Sigma_1$-ordinal-generation)
$\M=\Hull_{\Sigma_1}^\M(\M^-\un\{\M^-\}\un\OR^\M)$.
 \item\label{item:stratification} \index{stratification}(Stratification) There is a limit $\gamma\in\Ord$ and sequence
$\Mtilde=\left<\Mtilde_\alpha\right>_{\alpha<\gamma}$ such that:
\begin{enumerate}
\item for each $\alpha<\gamma$, $\Mtilde_\alpha$ is an $\Ll_0$-structure such that
$\univ{\Mtilde_\alpha}$ is transitive and $\Mtilde_\alpha=\M\rest\univ{\Mtilde_\alpha}$;
that is, $\hmb^{\Mtilde_\alpha}=A$ and $\hmp^{\Mtilde_\alpha}=\hmp^\M$ and
$\hmP^{\Mtilde_\alpha}=\hmP^{\M}\inter\univ{\Mtilde_\alpha}$, etc,
\item $\Mtilde$ is a continuous, strictly increasing sequence with $\M^-\in\Mtilde_0$ and
$\M=\bigcup_{\alpha<\gamma}\Mtilde_\alpha$,
 \item\label{item:Ntilde_def} $\Mtilde\rest\alpha\in\M$ for every $\alpha<\gamma$, and the function
$\alpha\mapsto\Mtilde\rest\alpha$,
with domain $\gamma$, is
$\Sigma_1^\M(\{\M^-\})$.\qedhere
\end{enumerate}
\end{enumerate}
\end{enumerate}
\end{dfn}

\begin{rem}
Let $\N$ be a potential opm over $A$. Suppose $\hmE^\N$ codes an extender $F$. Clearly
$\kappa=\crit(F)>\Theta^\M>\rank(A)$. By \cite[Definition
2.2.1]{wilson2012contributions}, we have
$\kappa^{+\M}<\OR(\M)$; cf. \ref{rem:cardinal_successors}. Note that
\emph{we allow $F$ to be of superstrong type} (see
\ref{someNotations}) in accordance with
\cite{wilson2012contributions}, not \cite[Definition 2.4]{steel2010outline}.\footnote{The main
point of permitting superstrong
extenders is that it simplifies
certain things. But it complicates others. If the reader prefers, one could
instead require that $F$ \emph{not} be superstrong, but
various
statements throughout the paper regarding condensation would need to be modified, along the lines
of \cite[Lemma 3.3]{FSIT}.} Mitchell-Steel fine structure
for premice with superstrongs has been developed in \cite{premouse_inheriting} and \cite{fsfni}, and of course, \cite{zeman} developed it for Jensen fine structure.
\end{rem}

\begin{dfn}
\label{someNotations}
Let $\M$ be a potential opm over $A$.
We say that $\M$ is \dfnemph{$\hmE$-active} iff
$\hmE^\M\neq\emptyset$, and \dfnemph{$\hmP$-active} iff
$\hmP^\M\neq\emptyset$. \index{active}\dfnemph{Active} means either
$\hmE$-active or $\hmP$-active. \dfnemph{$\hmE$-passive} means not
$\hmE$-active. \dfnemph{$\hmP$-passive} means not $\hmP$-active. \index{passive}\dfnemph{Passive}
means not active. \index{type}\dfnemph{Type 0} means passive.
\dfnemph{Type 4} means $\hmP$-active. \dfnemph{Type 1, 2} or \dfnemph{3}
mean $\hmE$-active, with  numerology as in \cite{FSIT}.

We write \index{$\actext^\M$}$\actext^\M$ for the extender $F$ coded by $\hmE^\M$ (where $F=\emptyset$ if
$\hmE^\M=\emptyset$).
We write \index{$\es^\M$}$\es^\M$ for the function with domain $\wl(\M)$, sending
$\alpha\mapsto\actext^{\M\wupto\alpha}$.
Likewise for \index{$\es_+^\M$}$\es^\M_+$, but with domain $\wl(\M)+1$.

{\color{black}For $\alpha\leq\lh(\M)$, we define \index{$\M\wupto\wupto\alpha$}$\M||\alpha$ as follows.
If $\M|\alpha$ is $E$-passive
then $\M||\alpha=\M|\alpha$.
If $\M$ is $E$-active
then $\M||\alpha$
denotes the (unique) $E$-passive
potential opm
$\N$ with $S^{\N}=S^{\M|\alpha}$ (so $\lh(\M)=\alpha$ and $\N$ has  the same universe as has $\M|\alpha$).}

Suppose $F=F^\M\neq\emptyset$ and $\kappa=\crit(F)$. {\color{black}As usual, \index{$\nu(F)$}$\nu(F)=\max(\kappa^{+\M},\nu'(F))$,
where $\nu'(F)$ is the strict sup of generators of $F$.} We say $\M$, or $F$, is \index{superstrong}\dfnemph{superstrong} iff
$i_F(\crit(F))=\nu(F)$.

Suppose $\M$ is a successor.
A \index{stratification}\dfnemph{stratification} of $\M$ is a sequence $\Mtilde$ witnessing
\ref{dfn:potential opm}(\ref{item:stratification}) for $\M$. For a $\Sigma_1$ formula
$\varphi\in\Ll_0$,
we say that $\M$ is \index{$\varphi$-stratified}\dfnemph{$\varphi$-stratified} iff $\varphi(\M^-,\cdot)^\M$ defines the set of
all
proper restrictions $\Mtilde\rest\alpha$ of a
stratification $\Mtilde$ of $\M$.\footnote{The $\varphi$-stratification of $\M$ need not imply that
every successor $\N\wpins\M$ is $\varphi$-stratified.}
\end{dfn}

\begin{lem}\label{lem:local_hulls}
Let $\M$ be a successor potential opm, over $A$. Let
$\Mtilde=\left<\Mtilde_\alpha\right>_{\alpha<\gamma}$ be a stratification of $\M$.
For $\alpha<\gamma$ let
\[ H_\alpha=\Hull_1^{\Mtilde_\alpha}({\color{black}A\un\OR(\Mtilde_\alpha)}). \]
Then for every $x\in\M$ there is $\alpha<\gamma$ such that $x\sub H_\alpha$.
\end{lem}
\begin{proof} Use $\Sigma_1$-ordinal-generation and
$A$-ordinal-surjections.
\end{proof}

\begin{lem}\label{lem:popm_is_Hull_1(A_cup_OR)}
	Let $\M$ be a potential opm $\M$  over coarse base $A$. Then  $\M=\Hull_1^{\M}(A\cup\OR^{\M})$,
	and
	$\Hull_1^{\M}(A\cup X)\preccurlyeq_1\M$ for all $X\sub\M$.
\end{lem}
\begin{proof}
	Use $\Sigma_1$-ordinal-generation
	and Lemma \ref{lem:canonical_surjection}.
\end{proof}

\subsection{Fine structure}
{\color{black}We now want to define the various fine structural notions for the structures under consideration, such as soundness, projecta, etc. We will want these definitions to apply not just to potential opms, but also the squash of a type 3 potential opm, once we have defined these. The following definition abstracts out what properties we want for this.}

\begin{dfn}\label{dfn:pre-fine}
Let $\N$ be a structure for a finite first-order language $\Ll$ {\color{black}consisting of constants and relation symbols}. We say that $\N$ is
\index{pre-fine}\dfnemph{pre-fine} (for $\Ll$) iff:
\begin{enumerate}[label=(\roman*)]
 \item $\Ll$ is a finite and $\{\dot{\in},\hmbdot\}\sub\Ll$, where
$\dot{\in}$ is a binary relation symbol and $\hmbdot$ is a constant symbol.
 \item $\N$ is an amenable $\Ll$-structure with transitive, rud closed, rank closed universe
$\univ{\N}$ and $\dot{\in}^\N={\in}\inter\univ{\N}^2$ and $\hmbdot^\N$ is transitive.
 \item\label{item:pre-fine_OR_Sigma_1-gen} $\N=\Hull_{\Sigma_1}^\N(\hmbdot^\N\un\OR(\N))$ (note the language here is $\Ll$),
 \item\label{item:pre-fine_stratification} {\color{black} there is a $\Sigma_1$ formula $\varphi$ of $\Ll$,
 a limit ordinal $\eta$ and a sequence $\left<\N_\alpha\right>_{\alpha<\eta}$ such that each $\N_\alpha$ is an $\Ll$-structure, $\univ{\N_\alpha}$ is transitive, $\N_\alpha\in\N$,
 \[ \N_\alpha=\N\rest\univ{\N_\alpha},\]
 $\univ{\N_\alpha}\sub\univ{\N_\beta}$ for $\alpha<\beta$, $\univ{\N}=\bigcup_{\alpha<\eta}\univ{\N_\alpha}$, and
 for all $X\in\N$,
 \[ X\in\{\N_\alpha\}_{\alpha<\eta}\iff\N\sats\varphi(X).\]
 (Since $\hmbdot^{\N_\alpha}=\hmbdot^{\N}$ and $\N_\alpha$ is transitive, it follows that $\hmbdot^{\N}\sub\N_\alpha$.)}\qedhere
\end{enumerate}
\end{dfn}

\begin{dfn}[Fine structure]\label{dfn:fine_structure}
Let $\N$ be pre-fine for the language $\Ll$. We sketch a description of the \index{fine structural notions}\dfnemph{fine
structural
notions}
for $\N$. For details refer to \cite{FSIT},\cite{steel2010outline};
we also adopt the simplifications explained in \cite[\S5]{V=HODX}.\footnote{The simplifications
involve dropping the parameters $u_n$, and replacing the use of generalized theories with pure
theories. These changes are not important, and if the reader prefers, one could redefine things
more analogously to \cite{FSIT},\cite{steel2010outline}.} Let $A=\hmb^\N$.

We say that $\N$ is \dfnemph{$0$-sound} and let \index{$\rho_n^\M$}\index{$p_n^\M$}$\rho_0^\N=\OR(\N)$,
$p_0^\N=\emptyset$, \index{$\core_n(\M)$}$\core_0(\N)=\N$ and
\index{$\rSigma_n^\M$}$\rSigma_{1}^{\N}=\Sigma_1^{\core_0(\N)}$ (here and in what follows, definability is with respect
to $\Ll$). Let \index{$T_n^\M$}$T^\N_0 = \N$.

Now let $n<\om$ and {\color{black}suppose we have already defined $\rho_n^\N$, $\pvec_n^\N=(p_1^\N,\ldots,p_n^\N)$,
$\core_n(\N)$, and $n$-soundness.
Suppose that $\N$ is $n$-sound,
which will imply that $\core_n(\N)=\N$. Suppose we have also defined the class of $\rSigma_{n+1}^\N$ relations, and that every $\Sigma_1^\N$ relation is $\rSigma_{n+1}^\N$.}

{\color{black}Define \index{$\rho_n^\M$}$\rho_{n+1}^\N$} as the least ordinal $\rho\geq\om$ such that for some
$X\sub A^{<\om}\cross\rho^{<\om}$, $X$ is $\bfrSigma_{n+1}^{\N}$ but $X\notin\univ{\N}$.
Define \index{$p_n^\M$}$p_{n+1}^\N$ as the least  $p\in[\Ord]^{<\om}$ such that some such $X$ is
\[ \rSigma_{n+1}^\N(A\un{\color{black}\rho_{n+1}^{\N}}\un\{p,\pvec^\N_n\}).\]
Here $p_{n+1}^\N$ is well-defined
by {\color{black}condition \ref{item:pre-fine_OR_Sigma_1-gen} of pre-fineness}.
For $X\sub\N$, let\index{$\Hull_{n}^\M(X)$}
\[ \Hull_{n+1}^\N(X)=\Hull_{\rSigma_{n+1}}^\N(X),\]
and \index{$\cHull_{n}^\M(X)$}$\cHull_{n+1}^\N(X)$ be
its transitive collapse, if this hull is extensional. If $A\subseteq X$ then the hull is indeed extensional,
as then $\Hull_1^{\N}(X)\preccurlyeq_1\N$, {\color{black}since
$\rSigma_1^\N\sub\rSigma_{n+1}^\M$,
and
using conditions \ref{item:pre-fine_OR_Sigma_1-gen} and \ref{item:pre-fine_stratification} of pre-fineness, we can define $\Sigma_1$ Skolem functions in a standard manner.}\footnote{\label{ftn:Sigma_1_Skolem}\color{black}That is,
fix $\left<\N_\alpha\right>_{\alpha<\eta}$ as in \ref{item:pre-fine_stratification}.
Let $\psi$ be $\Sigma_0$
and $\vec{x}\in X^{<\om}$
with $\N\sats\exists y\ \psi(y,\vec{x})$. We want to see that there is $y\in\Hull_1^{\N}(X)$
such that $\N\sats\psi(y,\vec{x})$.
Using \ref{item:pre-fine_OR_Sigma_1-gen}, let $\vec{a}\in A^{<\om}$
and $\tau$  be a $\Sigma_1$ formula
 such that there is $\vec{\beta}\in[\OR^\N]^{<\om}$
and  $y\in\N$ such that \[ \big[\N\sats\psi(y,\vec{x})\big]\text{
and }
\big[y\text{ is the unique }y'\in\N\text{
such that }\N\sats\tau(\vec{a},\vec{\beta},y')\big].\]
Then there is $\alpha<\eta$ such that $\vec{x}\in(\N_\alpha)^{<\om}$
and there is $\vec{\beta}\in[\OR^{\N_\alpha}]^{<\om}$
and $y\in\N_\alpha$ such that
\[ \big[\N_\alpha\sats\psi(y,\vec{x})\big]\text{
and }
\big[y\text{ is the unique }y'\in\N_\alpha\text{
such that }\N_\alpha\sats\tau(\vec{a},\vec{\beta},y')\big].\]

Let $\alpha_0$ be least such $\alpha$. Let $\vec{\beta}_0\in[\OR(\N_{\alpha_0})]^{<\om}$
be the least witness to the choice of $\alpha_0$.
Let $y_0$ be the unique $y\in\N_{\alpha_0}$
such that
$\N_{\alpha_0}\sats\tau(\vec{a},\vec{\beta}_0,y)$.
Then $\N\sats\psi(y_0,\vec{x})$ and $\alpha_0,\vec{\beta}_0,y_0\in\Hull_1^\N(X)$, which suffices.} Let\index{$\Th_{n}^\M(X)$}
\[ \Th_{n+1}^\N(X)=\Th_{\rSigma_{n+1}}^\N(X). \]
(Recall that $\Th^\N_{\rSigma_{n+1}}(X)$
was specified at the end of \S\ref{sec:notation}.
Note that this
theory is analogous to the \emph{pure} $\rSigma_{n+1}$ theory, as opposed to the
\emph{generalized}
$\rSigma_{n+1}$ theory,
in the sense of \cite{FSIT}.\footnote{As in \cite[\S2]{FSIT}, it does not matter which we
use.}.) Then we let\index{core}\index{$\core_{n}(\M)$}
\[ \core_{n+1}(\N)=\cHull_{n+1}^\N(A\un\rho_{n+1}^\N\un\pvec_{n+1}^\N), \]
and the uncollapse map $\pi:\core_{n+1}(\N)\to\N$ is the associated \index{core embedding}\dfnemph{core embedding}.
{\color{black}We say that $\N$ is \index{universal}$(n+1)$-\dfnemph{universal} iff
\[ \N\cap\pow\big(A^{<\om}\cross(\rho_{n+1}^\N)^{<\om}\big)\sub\core_{n+1}(\N).\]
For $\alpha\in p_{n+1}^{\N}$, define
the
 \index{solidity witness}\dfnemph{$(n+1)$-solidity witness $\mathcal{W}^\N_{n+1}(\alpha)$ at $\alpha$} by setting
\[ \mathcal{W}^\N_{n+1}(\alpha)=\cHull_{n+1}\big(A\cup\alpha\cup(p_{n+1}^{\N}\cut(\alpha+1))\cup\pvec_n^{\N}\big).\]
We say that $\N$ is \index{solid}\dfnemph{$(n+1)$-solid} iff $\mathcal{W}^{\N}_{n+1}(\alpha)\in\N$ for each $\alpha\in p_{n+1}^{\N}$.} (This follows the form of the definition of solidity in \cite{zeman}. In \cite{FSIT} and \cite{steel2010outline},
 $\rSigma_{n+1}$-theories are used as solidity witnesses, instead of transitive structures. Note that also
 following Zeman \cite{zeman} but not Steel \cite{steel2010outline},
	we do not incorporate $(n+1)$-universality into $(n+1)$-solidity.)
We say that $\N$ is \index{sound}$(n+1)$-\dfnemph{sound} iff $\N$ is $(n+1)$-solid and $\core_{n+1}(\N)=\N$ and the core embedding $\pi=\id$.

Now suppose that $\N$ is $(n+1)$-sound and $\rho_{n+1}^\N>\om$ (so $\rho_{n+1}^\N>\rank(A)$).\footnote{{\color{black}If $A\cap\OR<\rank(A)$,
we can still definably refer to each $\alpha<\rank(A)$ via elements $x\in A$ of rank $\alpha$. This is because  $\N$ is rud closed and rank closed,
so the rank function $r_A:A\to\OR$
is in $\N$, and since $A=\hmbdot^\N$, note that $\{r_A\}$ is $\rSigma_1^{\N}$. It easily follows that we can't have $\om<\rho_{n+1}^\N\leq\rank(A)$.}}
Define \index{$T_n^\M$}$T=T_{n+1}^\N\sub\N$ by letting $t\in T$ iff
for some $q\in\N$ and $\alpha<\rho_{n+1}^\N$,
\[ t=\Th_{{n+1}}^\N(A\un\alpha\un\{q\}).
\]
Define \index{$\rSigma_n^\M$}$\rSigma_{n+2}$ from $T_{n+1}$ as
usual: an $\rSigma_{n+2}$ formula $\varphi(\vec{v})$ (in free variables $\vec{v}$) is one of form
\[\exists t\ \big(T_{n+1}(t) \wedge \psi(t,\vec{v})\big) \]
where $\psi$ is $\rSigma_1$. The $\rSigma_{n+2}^\N$ relations are then given by interepreting these formulas over $\N$, with $T_{n+1}$ interpreted as $T_{n+1}^\N$. This completes the definitions.
\end{dfn}

\begin{dfn}
Let $\N$ be a potential opm.

If $\N$ is $\hmE$-active then \index{$\hmmu^\M$}$\hmmu^\N\eqdef\crit(F^\N)$, and otherwise $\hmmu^\N\eqdef\emptyset$.\footnote{Recall the language $\Ll_0^+=\Ll_0\cup\{\dot\mu,\dot e\}$ was specified in Definition \ref{dfn:language_Ll_0}.}

If $\N$ is $\hmE$-active type $2$ then \index{$\hme^\M$}$\hme^\N$ denotes the trivial completion of the largest
non-type $Z$ proper segment of $F$; otherwise $\hme^\N\eqdef\emptyset$.\footnote{In \cite{FSIT},
the (analogue of) $\hme$ is referred to by its
code $\gamma^\M$. We use $\hme$ instead because this does not depend
on having (and selecting) a wellorder of $\M$.}

If $\N$ is non-type 3 then \index{$\core_n(\M)$}$\core_0(\N)=\N^\sq$ denotes the $\Ll_0^+$-structure
$(\N,\hmmu^\N,\hme^\N)$ (with $\hmmudot^\N=\hmmu^\N$ etc).

If $\N$ is type 3 then define the $\Ll_0^+$-structure $\core_0(\N)=\N^\sq$, called the \index{squash}\index{$\M^\sq$}\dfnemph{squash} of $\N$, essentially as in
\cite{FSIT}; so
 \[ \N^\sq=(R,\hmE',\hmP',\hmS',\hmX';\hmb^\N,\hmp^\N,\hmmu^\N,\hme^\N)\]
where $\nu=\nu(F^\N)$, $R=\univ{\N|\nu}$, $\hmE'=F^\N\rest\nu$ (which is the usual coding
of $F^\N$ over the squash), $\hmP'=\emptyset$, $\hmS'=\hmS^\N\inter R$ and $\hmX'=\hmX^\N\inter R$. Note $e^{\N^\sq}=e^\N=\emptyset$ here.

We define the \index{fine structural notions}\dfnemph{fine structural notions} for $\N$ ($n$-\index{sound}soundness,
 \index{$\rho_n^\M$}$\rho_{n+1}^{\N}$, \index{$\Hull_n^\M(X)$}$\Hull^\N_{n+1}$, $\Th_{n+1}^\N$, etc) as those for
$\core_0(\N)$.\footnote{Thus, when we write, say, $\M=\cHull_{n+1}^\N(X)$,
we will have $X\sub\core_0(\N)$ and literally mean that
$\core_0(\M)=\R$ where $\R=\cHull_{n+1}^{\core_0(\N)}(X)$.
So if $\N$ is type 3, then $\M$ is produced by first squashing $\N$, forming the transitive collapse $\R$ of the hull of $X$ in $\N^\sq$, and then unsquashing $\R$ to reach $\M$. However, if $\N$ is type 3 and $n=0$ it is possible that
unsquashing $\R$ produces an illfounded structure $\M$, in which case $\core_0(\M)$ has not
literally been defined. In this case, we define $\M$ to be this illfounded structure,
and define $\core_0(\M)=\R$.}
\end{dfn}

In the proof of the solidity, etc, of iterable opms, one must also deal with
structures which are almost active opms, except that they may fail the ISC.
The details are immediate modifications of the standard notions, so we leave them to the reader.

The following definition is just the direct adaptation of the usual one:
\begin{dfn}\label{weak0}
Let $\M$ be a potential opm. Let $\R$ be an $\Ll_0^+$-structure (possibly illfounded). Let
$\pi:\R\to\core_0(\M)$.

We say that $\pi$ is a \index{weak $n$-embedding}\dfnemph{weak $0$-embedding} iff $\pi$ is
$\Sigma_0$-elementary (therefore $\R$ is extensional and wellfounded, so assume $\R$ is
transitive) and there is $X\sub\R$ such that $X$ is $\in$-cofinal in $\R$ and
$\pi$ is $\Sigma_1$-elementary on elements of $X$, and if $\M$ is type 1 or 2, then letting
$\mu=\hmmu^\R$,
there is $Y\sub\R\wupto\mu^{+\R}\cross\R$ such that $Y$ is ${\in}\cross{\in}$-cofinal in
$(\R\wupto\mu^{+\R})\cross\R$ and $\pi$ is $\Sigma_1$-elementary on
elements of $Y$.\end{dfn}

The following definition of \emph{\tu{(}near\tu{)} $k$-embedding} is analogous
to that in \cite[Definition 2.20, Remark 4.3]{steel2010outline},
and \emph{weak $k$-embedding} analogous
to that introduced in \cite{copy_con}
(the change in the definition in \emph{weak $k$-embedding} between \cite{FSIT} and the one we use here is due to Steve Jackson).\footnote{\label{ftn:steve}Jackson noticed a difficulty in  the proof of the Shift Lemma for weak $k$-embeddings as defined in \cite{FSIT} and \cite{steel2010outline}, when $0<k<\om$. Jackson suggested the definition we use here as a replacement. (The problem is that it does not seem obvious that a weak $k$-embedding $\pi$ in the sense of \cite{FSIT} is always such that $\pi``T_k^M\sub T_k^N$.
But we do not actually know of an example of a weak $k$-embedding
as defined in \cite{FSIT}
for which this or the Shift Lemma fails.}
\begin{dfn}\label{dfn:k-emb}
{\color{black}\index{embedding}Let $k\leq\om$
and let $\M,\N$ be $k$-sound opms.
Let \[\pi:\core_0(\M)\to\core_0(\N).\]

We say that $\pi$ is a \index{near $n$-embedding}\emph{near $k$-embedding}
iff either $k=\om$ and $\pi$
is fully elementary,
or  $k<\om$ and:
\begin{enumerate}\item $\pi$ is $\rSigma_{k+1}$-elementary,
\item $\pi(\pvec_k^\M)=\pvec_k^\N$,
\item for $i<k$, we have:
\begin{enumerate}[label=--]\item if $\rho_i^\M=\rho_0^\M$ then $\rho_i^\N=\rho_0^\N$, and
\item if $\rho_i^\M<\rho_0^\M$
then $\pi(\rho_i^\M)=\rho_i^\N$.
\end{enumerate}
\end{enumerate}

We say $\pi$ is a \index{$n$-embedding}\emph{$k$-embedding}
iff $\pi$ is a near $k$-embedding
and if $k<\om$
then $\pi``\rho_k^\M$ is cofinal in $\rho_k^\N$.

If $0<k\leq\om$,
we say $\pi$ is a \index{weak $n$-embedding}\emph{weak $k$-embedding} iff $k=\om$ and $\pi$ is fully elementary,
or $k<\om$ and has the properties of a near $k$-embedding,
except that instead of $\rSigma_{k+1}$-elementarity,
we only demand that $\pi$ is $\rSigma_k$-elementary and there is a set $X\sub\rho_k^\M$ such that
{\color{black}$X$ is cofinal in $\rho_k^\M$ and}  $\pi$ is $\rSigma_{k+1}$-elementary on parameters in
$\Hull_{k+1}^\M(X\cup\{\pvec_k^\M\})$. Therefore $\pi `` T_k^{\M} \subseteq T^{\N}_k$.
(Note that
this definition of \emph{weak $k$-embedding} diverges slightly from the definitions given in
\cite[p.~52]{FSIT} and \cite[Definition 4.1]{steel2010outline}; see Footnote \ref{ftn:steve}.)}

We say that $\pi$ is \index{$n$-good}\index{good}\index{nearly $n$-good}\index{weakly $n$-good}\dfnemph{(weakly, nearly) $k$-good} iff $\pi$ is a (weak, near)
$k$-embedding and $\hmb^\M=\hmb^\N$ and $\pi\rest \hmb^\M=\id$.
\end{dfn}

{\color{black}The following definition is not intended to be a comprehensive statement of fine structural condensation.
It simply encompasses some basic instances of condensation, which for example arise in the basic proofs of fine structure, such as solidity, etc. The definition is also of low enough complexity that it is preserved by the relevant hulls and ultrapower maps.}

\begin{dfn}\label{dfn:<om-condensing}
Let $\N$ be an $\om$-sound potential opm. We say that $\N$ is \index{condensing}\index{${<\om}$-condensing}\dfnemph{${<\om}$-condensing}
iff for every $k<\om$, for every soundly projecting, $(k+1)$-sound potential opm $\M$,
for every near $k$-embedding
$\pi:\M\to\N$ such that $\rho=\rho_{k+1}^\M\leq\crit(\pi)$ and
$\rho<\rho_{k+1}^\N$, we have the following. {\color{black}\label{c23}If $\N\wupto\rho$ is $\hmE$-passive let $\Q=\N$,
and otherwise let $\Q=\Ult(\N\wupto\rho,F^{\N\wupto\rho})$.} Then either:
\begin{enumerate}[label=--]
\item $\M\wpins\Q$, or
\item $\M^-\wpins\Q$, and $\M\in\R$ where $\R\wpins\Q$ is such
that $\R^-=\M^-$.\qedhere
\end{enumerate}
\end{dfn}

{\color{black}The inclusion of the second option above (where $\M^-\pins\Q$ and $\M\in\R$) might appear to diverge from the usual kind of conclusion for condensation, as we do not have $\M\ins\Q$ here; instead,
$\M$ is strictly ``between'' $\R^-$ and $\R$.}{\color{black}\label{c24} Note that if this clause attains then by projection amenability for $\R$, $\M$
is sound and
 $\rho_\om^\M=\rho_1^\M=\rho_\om^{\M^-}$, the soundness following from the fact that $\rho_\om^\M=\rho_\om^{\M^-}$.}

 \subsection{Q-operator-premice}

\begin{dfn}
 A \index{Q-opm}\dfnemph{Q-operator-premouse \tu{(}Q-opm\tu{)}}\footnote{\emph{Q} is for \emph{Q-formula}. We
will
see that the first-order aspects
of Q-opm-hood are expressible with Q-formulas and P-formulas.} is a potential operator-premouse
$\M$ such that every $\N\wpins\M$ is $\om$-sound and ${<\om}$-condensing.
\end{dfn}

{\color{black}Q-operator-premice
are basically analogous to \emph{premice}
in \cite{FSIT}. However, we will soon refine things one step further, defining \emph{operator-premice}, which will be the primary objects of interest; these are just Q-operator-premice,
all of whose segments are soundly projecting (including the top one), and they are also analogous to premice. Of course for premice,
there is no distinction between these two notions.}
The analogy with the premice of
 \cite{FSIT} does fail, however, in a minor regard: \cite{FSIT} makes no condensation demands of proper segments of
premice. We
make this requirement here so that we can avoid stating it as an explicit axiom at certain points later
(and it holds for the structures we care about).

{\color{black}Much as in \cite{FSIT},
modulo wellfoundedness
(of the relevant objects),
we can capture the property of being a Q-opm with formulas of the following forms:

\begin{dfn}\label{dfn:Q,P-formula}
The class of (non-simple) \index{Q-formula}\dfnemph{$\Ll_0^+$-Q-formulas}  is defined as in \cite[Definition 2.3.9]{FSIT}; that is, these are the formulas $\varphi(\vec{u})$
of form
\[ \all x\ \all \xi<\dot{\mu}^{+}\  \exists y\ \exists\eta<\dot{\mu}^{+}\ \big[x\sub y\wedge \xi\leq\eta\wedge\psi(y,\eta,\vec{u})\big],\]
where $\psi(v_0,v_1,\vec{u})$
is an $\rSigma_1$ formula
of $\Ll_0^+$,
which
has only $v_0,v_1,\vec{u}$ free.
We define
the class of
\index{P-formula}\dfnemph{$\Ll_0^+$-P-formulas}
just like the conventional notion of Q-formulas (instead of following \cite[Definition 3.1.4]{FSIT});
these are the formulas $\varphi(\vec{u})$
of form
\[\all x\ \exists y\ \big[x\sub y\wedge\psi(y,\vec{u})\big] \]
where $\psi(v,\vec{u})$ is an $\rSigma_1$ of $\Ll_0^+$ and has only $v,\vec{u}$ free.
\end{dfn}}

{\color{black}Recall from Definition \ref{dfn:adequate} that an
adequate model-plus is an $\Ll_0^+$-structure $\N$ such that $\N\rest\Ll_0$ is an
adequate model.}

\begin{lem}\label{lem:Q-opm_type_Q-formulas} There are $\Ll_0^+$-Q-formulas
$\varphi_1,\varphi_2$,
an $\Ll_0^+$-P-formula $\varphi_3$, an $\Ll_0^+$-simple-Q-formula $\varphi_{0,\mathrm{limit}}$,
and for each $\Sigma_1$ formula $\psi\in\Ll_0$ there are $\Ll_0^+$-simple-Q-formulas
$\varphi_{0,\psi},\varphi_{4,\psi}$,
{\color{black}obtained recursively from $\psi$,} such that for any adequate model-plus $\N'$:
\begin{enumerate}
\item\label{item:passive_limit} $\N'\sats\varphi_{\mathrm{0,limit}}$ iff $\N'=\core_0(\N)$ for some
limit passive Q-opm $\N$.
\item\label{item:B-active} $\N'\sats\varphi_{4,\psi}$ iff $\N'=\core_0(\N)$ for some
$\psi$-stratified $\hmP$-active Q-opm $\N$.
\item\label{item:passive} $\N'\sats\varphi_{0,\psi}$ iff $\N'=\core_0(\N)$ for some passive Q-opm
$\N$ which is either a limit or is $\psi$-stratified.
 \item\label{item:active_1,2} $\N'\sats\varphi_1$
\tu{(}respectively, $\N'\sats\varphi_2$\tu{)} iff $\N'=\core_0(\N)$ for some type 1
\tu{(}respectively, type 2\tu{)}
Q-opm $\N$.
 \item\label{item:active_3} If $\N'=\core_0(\N)$ for
some type 3 Q-opm $\N$ then $\N'\sats\varphi_3$. If $\N'\sats\varphi_3$ then $\hmE^{\N'}$
codes an extender
$F$ over $\N'$ such that if $\Ult(\N',F)$ is wellfounded then
$\N'=\core_0(\N)$ for some type 3 Q-opm $\N$.
\end{enumerate}
\end{lem}
\begin{proof}
Part \ref{item:passive_limit} is routine. Part \ref{item:active_1,2} is a straightforward adaptation of its analogue
\cite[Lemma 2.5]{FSIT}.
{\color{black}Part
\ref{item:active_3} is likewise adaptation of \cite[Lemma 3.3]{FSIT},
with the following two remarks.
Firstly,
the P-formulas of \cite[Definition 3.1.4]{FSIT}
are more liberal than
$\Ll_0^+$-P-formulas.
But note that each of
the  sentences $\theta_1,\ldots,\theta_5$
of the proof of \cite[Lemma 3.3]{FSIT}, adapted to our context, are expressible with an $\Ll_0^+$-P-formula
(for $\theta_1$,
this is as in part \ref{item:passive_limit} of the current lemma).
Now the rest of the proof of \cite[Lemma 3.3]{FSIT} goes through, noting that we have been able to
drop the clause ``or $\N$ is of superstrong type'' from the statement of \cite[Lemma
3.3]{FSIT}, because we allow extenders of superstrong type as the active extenders of Q-opms.}

Part \ref{item:B-active} is an easy adaptation of part \ref{item:passive}, using the fact that if
$\N$ is $\hmP$-active then $\hmP^\N\sub\N\cut\N^-$.
So it just remains to consider part \ref{item:passive};
we just sketch the proof of this.

Consider an adequate model-plus $\N'$ and $\N=\N'\rest\Ll_0$. We leave it to the reader to verify
that here is an $\Ll_0$-simple-Q-formula
asserting (when interpreted over $\N'$)
that every $\M\wpins\N$ is a ${<\om}$-condensing $\om$-sound potential opm,
and an
$\Ll_0^+$-simple-Q-formula asserting
that
$\hmP^{\N}=\hmE^{\N}=\hmmu^{\N}=\hme^{\N}=\emptyset$.
It remains to see that we can assert that \ref{dfn:potential opm}(\ref{item:successor_potential
opm}) holds for $\M=\N$
(the assertion will include the possibility that $\N$ is a limit). For
\ref{dfn:potential opm}(\ref{item:no_early_proj}), use the formula
``$\all x\ex y\big[x\sub y\wedge\varphi(y)\big]$'',
where $\varphi(y)$ asserts ``either there is $s\in \hmS^\M$ such that $y\in s$ or there are
$\hmS,A$
such that $\hmS=y\inter \hmS^\M$ and $A=\hmb^\M$ and $\hmS$ has a largest element $\P$ and
for each $\tau<\OR(\P)$, if there is $X\in y\cut\P$ such that $X\sub
A^{<\om}\cross\tau^{<\om}$, then there is $n<\om$ such that
$\rho_{n+1}^\P\leq\tau$, as witnessed by a satisfaction relation in
$y$'' (use the fact that $\N$ is rud closed).

Clause \ref{dfn:potential opm}(\ref{item:surjections}) is easy, and it is fairly straightforward to
assert that either $\N$ is a limit or $\N$ is $\psi$-stratified, identifying candidates for
$\N^-$ as in the previous paragraph.
We can therefore assert \ref{dfn:potential opm}(\ref{item:M=Hull_1^M(OR)}) as ``$\all x\ex y\big[x\sub
y$
and there is $\alpha<\gamma$ such that $y\sub H_\alpha\big]$'', where
$\gamma,H_\alpha$ are defined as in \ref{lem:local_hulls}, using the stratification given by $\psi$.
\end{proof}

{\color{black}
The natural adaptations of \tu{\cite[Lemmas 2.4, 3.2]{FSIT}}
hold, and the proofs are straightforward:

\begin{lem}\label{lem:Qformula_pres}
Let $\R,\Ss$ be transitive $\Ll_0^+$-structures. Let $\pi:\R\to\Ss$.
Let $\vec{x}\in\R^{<\om}$. Then:
\begin{enumerate}
\item
Let $\varphi$ be an $\Ll_0^+$-P-formula. Then:
\begin{enumerate}[label=\tu{(}\alph*\tu{)}]
\item If $\pi$ is $\rSigma_1$-elementary and $\Ss\sats\varphi(\pi(\vec{x}))$
then $\R\sats\varphi(\vec{x})$.
\item If $\pi$ is $\Sigma_0$-elementary and $\sub$-cofinal
and $\R\sats\varphi(\vec{x})$
then $\Ss\sats\varphi(\pi(\vec{x}))$.
\end{enumerate}
\item Let $\varphi$ be an $\Ll_0^+$-Q-formula. Suppose $\dot{\mu}^{\R}\in\OR^{\R}$. Then:
\begin{enumerate}[label=\tu{(}\alph*\tu{)}]
\item If $\pi$ is $\rSigma_1$-elementary
and $\Ss\sats\varphi(\pi(\vec{x}))$
then $\R\sats\varphi(\vec{x})$.
\item If $\pi$ is $\Sigma_0$-elementary and $\rg(\pi)$
is ${\sub}\cross{\sub}$-cofinal in $(\Ss|\dot{\mu}^{+\Ss})\cross\Ss$
and $\R\sats\varphi(\vec{x})$ then $\Ss\sats\varphi(\pi(\vec{x}))$.
\end{enumerate}
\end{enumerate}
\end{lem}}

There is also a version of this lemma for weak $0$-embeddings,
but here we only consider statements $\varphi$ without parameters:

\begin{lem}\label{lem:Qformula_pres_weak0}
Let $\N$ be a Q-opm, let $\R$ be an $\Ll_0^+$-structure and let $\pi:\R\to\core_0(\N)$ be a weak
$0$-embedding. Then:
\begin{enumerate}
\item\label{item:weak_0-emb_Q-formula} Suppose $\N$ is type $i\neq 3$. Then:
\begin{enumerate}[label=\tu{(}\alph*\tu{)}]
\item For every $\Ll_0^+$-Q-formula $\varphi$, if $\core_0(\N)\sats\varphi$ then $\R\sats\varphi$.
\item\label{item:weak_0-emb_R_is_Q-opm_same_type}  $\R=\core_0(\M)$ for some type $i$ Q-opm $\M$.
\footnote{Possibly $\M,\N$ are passive and $\N$ is a successor
but $\M$ a limit.}
\end{enumerate}
\item\label{item:weak_0-emb_P-formula} Suppose $\N$ is type 3. Then:
\begin{enumerate}[label=\tu{(}\alph*\tu{)}]\item For every $\Ll_0^+$-P-formula $\varphi$,
if $\core_0(\N)\sats\varphi$ then $\R\sats\varphi$.
{\color{black}\item\label{item:weak_0-emb_wfd_goes_down} If $\Ult(\core_0(\N),F^\N)$ is wellfounded then $\Ult(\R,F^\R)$ is wellfounded.}
\item\label{item:weak_0-emb_R_Q-opm_type_3} If $\Ult(\R,F^\R)$ is wellfounded then
$\R=\core_0(\M)$ for some type 3 Q-opm $\M$.
\end{enumerate}
\end{enumerate}
\end{lem}
The proof is routine,
{\color{black}using  Lemma \ref{lem:Q-opm_type_Q-formulas}
for parts \ref{item:weak_0-emb_Q-formula}\ref{item:weak_0-emb_R_is_Q-opm_same_type} and \ref{item:weak_0-emb_P-formula}\ref{item:weak_0-emb_R_Q-opm_type_3}}.

\begin{lem}\label{lem:hull_emb_properties}
 Let $n<\om$ and $\M$ be an $n$-sound Q-opm over $A$ with $\om<\rho_n^\M$. Let $X\sub\core_0(\M)$, let
$\R=\cHull_{n+1}^\M(A\un X\un\pvec_n^\M)$
and let $\pi:\R\to\core_0(\M)$ be the uncollapse. Then:
\begin{enumerate}
 \item\label{item:N_is_Q-opm} If either $n>1$ or $\M$ is non-type 3 or
$\Ult(\core_0(\M),F^\M)$ is wellfounded then $\R=\core_0(\N)$
for some Q-opm $\N$.
\item\label{item:pi_nearly_n-good} If $\R=\core_0(\N)$ for some Q-opm $\N$ then $\N$ is $n$-sound and $\pi$ is nearly $n$-good.
\end{enumerate}
\end{lem}
\begin{proof}
 Suppose $n=0$ and $\M$ is a successor. Then {\color{black}by Lemmas \ref{lem:Q-opm_type_Q-formulas} and \ref{lem:Qformula_pres},} it suffices to see that $\pi$ is
$\rSigma_1$-elementary,
or in other words, that $\rg(\pi)\preccurlyeq_{1}\M$.
But using  stratification
(as in part \ref{item:successor_potential opm}(\ref{item:stratification}) of Definition \ref{dfn:potential opm})
we can define appropriate $\Sigma_1$ Skolem functions over $\M$,
much as was done in Footnote \ref{ftn:Sigma_1_Skolem},
and thereby verify that $\rg(\pi)\preccurlyeq_1\M$.

If $n=0$ and $\M$ is a limit it is similar, but easier.
(If $\M$ is type 3, then by the hypothesis in part \ref{item:N_is_Q-opm},
$\Ult(\core_0(\M),F^{\M})$ is wellfounded,
which implies that $\Ult_0(\R,F^{\R})$ is wellfounded,
and $\R=\core_0(\N)$
for a type 3 Q-opm $\N$.)

If $n>0$, then the proof for standard premice adapts routinely.\footnote{The fine structural setup here is a little different
from that in \cite{FSIT}, as we have dropped the use of $u_i^\M$. See \cite[\S5]{V=HODX}
for calculations which deal with this difference.} (If $\M$ is type 3 and $n>1$, there is
$(a,f)\in\rg(\pi)$ such that $\nu(F^\M)=[a,f]^\M_{F^\M}$, which easily gives that {\color{black}$\R=\core_0(\N)$
for a type 3 Q-opm $\N$,
even if we don't know that $\Ult_0(\M,F^{\M})$ is wellfounded.)}\end{proof}

Using stratifications and standard calculations, we also have:

\begin{lem}\label{lem:core_embedding}
Let $\N,\M$ be $n$-sound Q-opms over $A$. Then:
\begin{enumerate}
\item Let $\pi:\core_0(\N)\to\core_0(\M)$ be nearly $n$-good.
Suppose that $\N\notin\M$ and
\[\N=\cHull_{n+1}^\N(A\un\rho\un\{q\}),\] where $\rho\in\Ord$ and $\rho\leq\crit(\pi)$
{\color{black}and $q\in[\rho_0^\N]^{<\om}$}.
Then $\pi$ is $n$-good.

\item If $\N=\core_{n+1}(\M)$ and $\pi$ is the core embedding, then $\pi$ is $n$-good.
\end{enumerate}
\end{lem}

\subsection{Operator-premice}
We finally reach the ultimate notion prior to introducing an actual operator $\Fop$ from which to build our premice:

\begin{dfn}\label{dfn:opm}
 An \index{operator-premouse}\index{opm}\dfnemph{operator-premouse \tu{(}opm\tu{)}} is a soundly projecting Q-opm. For an opm $\M$,
let \index{$q^\M$}$q^\M=p_1^\M\inter(\OR(\M^-),\OR(\M))$
(so if $\M$ is a limit then $q^\M=\emptyset$).
\end{dfn}
\begin{dfn}\label{dfn:(k+1,q)-solid}
 Let $\M$ be a $k$-sound opm over $A$ and $q\in(\rho_k^\M)^{<\om}$. We say that $\M$ is
\index{solid}\dfnemph{$(k+1,q)$-solid}
 iff for each $\alpha\in q$, letting $q'=q\cut(\alpha+1)$ and $X=A\un\alpha\un q'\un\pvec_k^\M$, we
have
$\cHull_{k+1}^\M(X)\in\M$.
\end{dfn}
\begin{lem}\label{lem:succ_opm_top_param} Let $\M$ be a successor opm and $\wl(\M)=\xi+1$. Let
$\rho=\rho_\om^{\M^-}$ and $p=p_1^\M\cut\rho$. Then $\M$ is $\rho$-sound and
$\rho_1^\M\leq\rho$ and either:
\begin{enumerate}[label=--]
		\item  $p=q^\M$ \tu{(}in other words, $p_1^{\M}\cap[\rho,\OR(\M^-)]=\emptyset$\tu{)}, or
	\item $q^\M=\emptyset$ and $p=\{\alpha\}$
for some $\alpha\in[\rho,\xi]$.
\end{enumerate}
Therefore either $\M$ is $\om$-sound and
$\rho_\om^\M=\rho=\rho_\om^{\M^-}$,
or there is $k<\om$ such that $\M$ is $k$-sound and $\rho_{k+1}^\M<\rho\leq\rho_k^\M$ {\color{black}\tu{(}so if $k>0$ then $\rho=\rho_k^\M$\tu{)}}.
\end{lem}
\begin{proof} We have $\rho_1^{\M}\leq\rho$ as $\M$ is soundly projecting. If
$q^\M\neq\emptyset$ then $p\inter[\rho,\OR(\M^-)]=\emptyset$,
as letting $A=\hmb^\M$, we have
\[ \M^-\un\{\M^-\}\sub\Hull_1^\M(A\un\rho\un q^\M), \]
as $\hmX^\M$ is
$\Sigma_1^\M$, and
this suffices since $\M$ is soundly projecting. So suppose $q^\M=\emptyset$.
Let $r$ be least in $(\xi+1)^{<\om}$ such that
\[ \M^-\in H=\Hull_1^\M(A\un\rho\un r).\]
Note that $H=\M$, since $\M$ is soundly projecting and $q^\M=\emptyset$.
So if $r=\emptyset$ then $p=\emptyset=q^\M$, so we are done.
So suppose $r\neq\emptyset$.

Suppose for a contradiction
that $\card(r)>1$ and let $\alpha_0>\alpha_1$ be the top 2 elements.
Let
\[ C=\cHull_1^{\M}(A\cup\alpha_1\cup\{\alpha_0\}) \]
and $\pi:C\to\M$ the uncollapse. So $\crit(\pi)=\alpha_1<\OR^{C}$.
So $\alpha_1$ is an $A$-cardinal of $C$, and so $\pi(\alpha_1)$
is an $A$-cardinal of $\M$. Since $\rho\leq\alpha_1<\pi(\alpha_1)$
and $\M^-$ projects to $\rho$, therefore $\pi(\alpha_1)>\OR(\M^-)$.
But then $\M^-\in\rg(\pi_1)$,
contradicting the minimality of $r$.

So $r=\{\alpha_0\}$ for some $\alpha_0$.
It remains to verify that $\M$ is $(1,r)$-solid.
Let
\[ C=\cHull_1^{\M}(A\cup\alpha_0) \]
and $\pi:C\to\M$ the uncollapse. Then $\rg(\pi)\cap\OR=\alpha_0$,
because otherwise  again  $\crit(\pi)$ is an $A$-cardinal of $C$, etc.
Note that because $r\neq\emptyset$,
we have $P^\M=\emptyset$ (since otherwise $P^\M\sub\M\cut\M^-$).
Therefore $C=\M||\alpha_0$,\footnote{This denotes the passivization of $\M|\alpha_0$;
that is, the passive opm $\P$
such that $S^P=S^{\M|\alpha_0}$.} so $C\in\M$, which gives that $\M$ is $(1,\{\alpha_0\})$-solid, as desired.\end{proof}

\begin{lem}\label{lem:Sigma_1_hull_reflects_opm}
 Let $\N$ be a successor opm and $\pi:\M\to\N$. Suppose  either
\begin{enumerate}[label=\tu{(}\roman*\tu{)}]
	\item\label{item:q_empty} $\pi$ is $\Sigma_1$-elementary and
$q^\N=\emptyset$, or
\item\label{item:pi_Sigma_2-elem} $\pi$ is $\Sigma_2$-elementary and
$q^\N\in\rg(\pi)$.
\end{enumerate}
Then $\M$ is a successor opm   of the same type as $\N$, and $\pi(q^\M)=q^\N$.
\end{lem}
\begin{proof}
By \ref{lem:Q-opm_type_Q-formulas}, $\M$ is a Q-opm and we may assume
$\N^-\in\rg(\pi)$, so $\M$ is a successor, $\pi(\M^-)=\N^-$, and $\M$ is $\psi$-stratified
where $\N$ is $\psi$-stratified.
Now
\[ \N=\Hull_1^{\N}(\N^-\cup\{\N^-\}\cup q^\N),\]
 which with $\Sigma_1$-elementarity gives
 \[ \M=\Hull_1^\M(\M^-\un\{\M^-\}\cup q) \]
  where $\pi(q)=q^\N$.
This suffices for part \ref{item:q_empty}. For part \ref{item:pi_Sigma_2-elem} use
generalized solidity witnesses to see that $\M$ is $(1,\bar{q})$-solid,
which is enough.
\end{proof}

However, in the context above,
if $\pi$ is just $\Sigma_1$-elementary and $q^\N\neq\emptyset$,
$\M$  might not be soundly projecting, even if $q^\N\in\rg(\pi)$. Such embeddings arise
when we take $\Sigma_1$ hulls like in the proof of $1$-solidity for $(0,\omega_1+1)$-iterable premice.

{\color{black}
\begin{lem}\label{lem:k+1-Hull}Let $n,\M,\R,\pi$ be as in Lemma \ref{lem:hull_emb_properties},
with $n>0$,
and suppose that $\M$ is an opm.
Suppose that either $n>1$
or $\M$ is non-type 3 or $\Ult(\core_0(\M),F^{\M})$
is wellfounded.
Then $\R=\core_0(\N)$ for an $n$-sound opm $\N$, and $\pi$ is nearly $n$-good.
\end{lem}

\begin{proof}
By Lemma \ref{lem:hull_emb_properties},
we know that $\R=\core_0(\N)$
 for some $n$-sound Q-opm $\N$
 and that $\pi$ is nearly $n$-good.
So we just need to see that $\N$ is an opm. So we may assume that $\N$ is a successor, so $\M$ is also. Since $n>0$, $\pi$ is $\rSigma_2$-elementary,
so by Lemma \ref{lem:Sigma_1_hull_reflects_opm},
it suffices to see that $q^\M\in\rg(\pi)$. But $q^\M\sub p_1^\M\in\rg(\pi)$, so we are done.
	\end{proof}}

Let $X$ be transitive. Then $X^\#$
determines naturally an opm $\M$ over $\hat{X}$ of length $1$, so $U=\Ult_0(\M,F^{X^\#})$ is also a
Q-opm over $\hat{X}$ of length $1$, but $U$ is not an opm.\footnote{$U$ is not soundly projecting.} So opm-hood is not expressible
with Q-formulas. However, given a successor opm $\N$,
we will only form ultrapowers of $\N$ with extenders $E$ such that $\crit(E)<\OR(\N^-)$,
and
under these circumstances, opm-hood is preserved. In fact, we will only form ultrapowers and
fine structural hulls under further fine structural assumptions:

\begin{dfn}
Let $k\leq\om$. An opm $\M$ is \index{$n$-relevant}\index{relevant}$k$-\dfnemph{relevant} iff
$\M$ is $k$-sound, and either $\M$ is a limit or $k=\om$ or
$\rho_{k+1}^\M<\rho_\om^{\M^-}$.

A Q-opm $\M$ which is not an opm (so $\M$ is a successor) is \dfnemph{$k$-relevant} iff $k=0$ and
$\rho_{1}^\M<\rho_\om^{\M^-}$.
\end{dfn}

For the development of the basic fine structure theory of opms, one only needs to iterate
$k$-relevant
opms (and phalanxes of such structures, and bicephali and pseudo-premice); see
\ref{dfn:iteration_trees}. For instance, the following lemma follows from
\ref{lem:succ_opm_top_param}:

\begin{lem}\label{lem:non_k-rpm_sound}
Let $k<\om$ and $\M$ be a $k$-sound operator-premouse which is not $k$-relevant. Then $\M$ is
$(k+1)$-sound.
\end{lem}

\subsection{Fine structure and iterations}

Now that we have introduced operator-premice and studied how they  behave under forming elementary hulls, we want to
consider forming iteration trees on them.
In the following lemma we establish the preservation of fine structure under degree $k$
ultrapowers, for $k$-relevant opms.
The proof involves a key use of stratification.

\begin{lem}\label{lem:Ult_k_pres_fs}
 Let $\M$ be a $k$-relevant opm over $A$ and $E$ an extender over $\M$, weakly amenable to
$\M$,
with
$\rank(A)<\crit(E)<\rho_k^\M$, and $\crit(E)<\rho_\om^{\M^-}$ if $\M$ is a successor.\footnote{Note that if $\M$ is a successor
and $k>0$,
then $\rho_k^\M\leq\rho_\om^{\M^-}$, but $\rho_0^{\M}>\rho_\om^{\M^-}$,
so the hypothesis that $\crit(E)<\rho_\om^{\M^-}$
is just needed when $k=0$.} Let $\N=\Ult_k(\M,E)$
and $j=i^{\M,k}_E$ be the ultrapower embedding.
Suppose $\N$ is wellfounded.
 Then:
 \begin{enumerate}
  \item\label{item:pres_relevant_opm_type} $\N$ is a $k$-relevant opm of the same type as
$\M$.
  \item\label{item:ult_pres_psi-strat} $\N$ is a successor iff $\M$ is. If
$\M$ is a successor then $j(\wl(\M))=\wl(\N)$ and if $\M$
is $\psi$-stratified then $\N$ is $\psi$-stratified.
  \item\label{item:i_E_k-embedding} $j$ is $k$-good.
  \item\label{item:pres_solidity} For any $q\in(\rho_k^\M)^{<\om}$, if $\M$ is $(k+1,q)$-solid
then
$\N$ is
$(k+1,j(q))$-solid.
\item\label{item:half_sup_pres_rho} $\rho_{k+1}^\N\leq\sup
j``\rho_{k+1}^\M$.
  \item\label{item:E_close_sup_pres_rho} If $E$ is close to $\M$ and $\M$ is $(k+1)$-solid then
$\rho_{k+1}^\N=\sup j``\rho_{k+1}^\M$ and $p_{k+1}^\N=j(p_{k+1}^\M)$ and $\N$ is
$(k+1)$-solid.
\end{enumerate}
\end{lem}
\begin{proof}
The fact that $\N$ is a Q-opm of the same type as $\M$ is by \ref{lem:Q-opm_type_Q-formulas}.
Part \ref{item:i_E_k-embedding} is standard and part
\ref{item:ult_pres_psi-strat} follows easily. We now verify that $\N$ is soundly projecting; we may
assume that
$\M,\N$ are successors. If $k>0$, use elementarity and stratification. Suppose $k=0$. Let
$\rho=\rho_\om^{\M^-}$ and $q=j(q^\M)$. The fact that $\N$ is $(1,q)$-solid follows by
an easy adaptation of the usual proof of preservation of the standard parameter, using
stratification (where in
the usual proof, one uses the natural stratification of the $\J$-hierarchy).
So it suffices to see that
$\N=\Hull_1^\N(\N^-\un\{\N^-,q\})$.
But this holds because $\M$ is an opm and
\[ \N=\Hull_1^\N(\rg(j)\un\nu_E) \]
and $\nu_E\sub\N^-$, the latter because $\crit(E)\leq\OR(\N^-)$ (in fact,
$\crit(E)<\rho_\om^{\N^-}$).

Parts \ref{item:pres_solidity}--\ref{item:E_close_sup_pres_rho}: If $k>0$ the proof for standard
premice works. See, for example, \cite[Lemmas 4.5, 4.6]{FSIT}, and if
$\kappa<\rho_{k+1}^\M$, see
\cite[Corollary 2.24]{extmax} and its proof
(that result is formally below superstrong, but essentially the same proof works) and/or
%conf
\cite[Lemma 3.8]{fsfni};
these arguments are related to the calculations in \cite[Claim 5 of Theorem 6.2]{FSIT}.
If $k=0$, again use stratification to adapt the usual proof. (In the case that
$\wl(\M)$ is a limit, $\M$ is of course ``stratified'' by its proper segments.)

By part \ref{item:half_sup_pres_rho}, it follows that $\N$ is $k$-relevant, completing part
\ref{item:pres_relevant_opm_type}.
\end{proof}

We next want to define (fine, such as $k$-maximal, etc) \emph{iteration trees} $\Tt$ on opms, following
the general form of \cite[\S3.1]{steel2010outline}:

\begin{dfn}\label{dfn:iteration_trees}
Let $k<\om$ and let $\M$ be a $k$-sound opm. The \index{iteration}\index{$n$-maximal}\index{maximal}\dfnemph{$k$-maximal iteration
game}
\index{$\G_n(\M,\theta)$}$\G_k(\M,\theta)$ \dfnemph{on $\M$, of length $\theta$} is defined completely analogously to the game $\G_k(\N,\theta)$
for $k$-sound premice $\N$,
as defined in
\cite[\S3.1]{steel2010outline}, except for the
following differences.
Let $\Tt$ be a partial play
(so $\Tt$ will be a putative tree). Then:
\begin{enumerate}[label=--]
\item It is player I's responsibility that for all $\beta+1<\alpha+1<\lh(\Tt)$, we have
$\lh(E^\Tt_\beta)\leq\lh(E^\Tt_\alpha)$
(as opposed to
$\lh(E^\Tt_\beta)<\lh(E^\Tt_\alpha)$, as is the requirement
in \cite{steel2010outline}).\footnote{\color{black}The weakening of this requirement is needed because we allow extenders of superstrong type in $\es_+^\M$. See Remark \ref{rem:superstrong} for details.}
\item It is player II's responsibility that for each $\alpha+1<\lh(\Tt)$, $M^\Tt_\alpha$ is an opm (as opposed to a premouse, as is the usual responsibility).
\end{enumerate}
The rest is as in \cite[\S3.1]{steel2010outline}. In particular, the game stops as soon as either player breaks a rule,  and
it is player I's responsibility that $E^\Tt_\alpha\in\es_+(M^\Tt_\alpha)$ for each $\alpha+1<\lh(\Tt)$.
Recall here that if $\M$ is a successor opm and $E\in\es_+^\M$
then $E\in\es_+^{\M^-}$.)

A \index{$n$-maximal}\index{iteration}\dfnemph{$k$-maximal iteration tree on $\M$} is a partial play of $\mathcal{G}_k(\M,\infty)$
in which neither player has yet lost. A \index{putative}\dfnemph{putative $k$-maximal iteration tree on $\M$} is as for a $k$-maximal iteration tree $\Tt$,
except that if $\Tt$ has successor length, then the main branch $b^\Tt$ of $\Tt$ is allowed to drop infinitely often,
and if it does not, there are no demands on the nature of the  last model $M^\Tt_\infty$
(other than that it be formed via the direct limit along the branch in the usual manner).

A \index{strategy}\index{$(n,\theta)$-iteration strategy}\dfnemph{$(k,\theta)$-iteration strategy} for $\M$ is
a winning strategy for player $\plII$ in $\G_k(\M,\theta)$.

{\color{black}The \index{$n$-maximal stack}\emph{$k$-maximal stack} iteration game
\index{$\G^*_n(\M,\alpha,\theta)$}$\G^*_k(\M,\alpha,\theta)$ is defined by analogy with the game $\G^*_k(\N,\alpha,\theta)$ for $k$-sound premice $\N$,
essentially as defined in \cite[prior to Corollary 1.10]{cmwmwc}, and see also \cite[\S4.1]{steel2010outline}.}\footnote{\color{black}Here are more details. For $\gamma<\alpha$,
after the first $\gamma$ rounds have been played, both players having met their commitments so far,
we have a $\gamma$-sequence $\Ttvec=\left<\Tt_\delta\right>_{\delta<\gamma}$ of iteration trees, with wellfounded final model
$M^{\Ttvec}_\infty$ (formed by direct limit if $\gamma$ is a limit); it follows that this
model is an $n$-sound opm where $n=\deg^{\Ttvec}_\infty$. At the beginning of
round $\gamma$, player $\plI$ chooses some
$(\Q,q)\wins(\M^{\Ttvec}_\infty,n)$, and round $\gamma$ is a (possibly partial) run of
$\G_q(\Q,\theta)$, producing a putative tree $\Tt_\gamma$. Player I is allowed to terminate this run
at any stage, after producing $\Tt_\gamma$
of length some $\xi+1<\theta$.
If player I wins the round of $\G_q(\Q,\theta)$ at any stage before terminating, then player I wins the full run of $\G^*_k(\M,\alpha,\theta)$.
If player I terminates with $\Tt_\gamma$ of length $\xi+1<\theta$
and $\Tt_\gamma$ is not a win for player I in $\G_q(\Q,\theta)$,
then they go on to round $\gamma+1$ of $\G^*_k(\M,\alpha,\theta)$, assuming $\gamma+1<\alpha$.
Suppose player I never terminates at any $\xi+1<\theta$, and
 $\Tt_\gamma$ is not a win for player I in $\G_q(\Q,\theta)$; so $\Tt_\gamma$ is an iteration tree of length $\theta$. Then
player II wins the full run of $\G^*_k(\M,\alpha,\theta)$.
Note that in case $\theta$ is a successor ordinal,
this last rule
is as in the games denoted $\G^*_k(\M,\alpha,\theta)$ in \cite{cmwmwc},
but differs from those denoted $\G_k(\M,\alpha,\theta)$ in \cite{steel2010outline}.} {\color{black}As in \cite[\S1.1.5]{iter_for_stacks}, the \index{$n$-optimal}\emph{$k$-optimal stack}
iteration game \index{$\G^*_{\opt,n}$}$\G=\G^*_{\opt,k}(\M,\alpha,\theta)$ is
defined likewise, except that we do not allow player $\plI$ to drop in model or degree at the
beginnings of rounds.} That is, (i) round $0$ of $\G$ is a run of $\G_k(\M,\theta)$,
and (ii) letting $0<\gamma<\alpha$ and $\Ttvec=\left<\Tt_\beta\right>_{\beta<\gamma}$ be the
sequence of trees played in rounds $\beta<\gamma$ and $\N=M^{\Ttvec}_\infty$ and
$n=\deg^{\Ttvec}_\infty$, round $\gamma$ of $\G$ is a run of
$\G_n(\N,\theta)$.

{\color{black}A \index{$(n,\alpha,\theta)^*$-optimal}\index{iteration strategy}\dfnemph{$(k,\alpha,\theta)^*$-optimal iteration strategy} for $\M$ is a winning strategy
for player $\plII$ in $\G^*_{\opt,k}(\M,\alpha,\theta)$, and a
\index{$(n,\alpha,\theta)^*$-iteration strategy}\dfnemph{$(k,\alpha,\theta)^*$-iteration strategy} is likewise for
$\G^*_k(\M,\alpha,\theta)$.}
{\color{black}See \cite[Lemma 9.8]{iter_for_stacks} for some basics on the connection
between $(k,\alpha,\theta)^*$- and $(k,\alpha,\theta)^*$-optimal iteration strategies.}

Now \index{iterability}\dfnemph{$(k,\theta)$-iterability}, {\color{black}\index{optimal iterability}\dfnemph{$(k,\alpha,\theta)^*$-optimal iterability},} etc, are
defined by the existence of the appropriate winning
strategy.
\end{dfn}

\begin{rem}\label{rem:superstrong}
\index{superstrong}The requirement, in $\G_k(\M,\theta)$, that $\lh(E^\Tt_\beta)\leq\lh(E^\Tt_\alpha)$ for
$\beta<\alpha$, is weaker than requiring that $\lh(E^\Tt_\beta)<\lh(E^\Tt_\alpha)$,
because opms may have superstrong extenders in their sequence. For example, we might have that
$E^\Tt_0$ is type 2 and $E^\Tt_1$ is
superstrong with $\crit(E^\Tt_1)$ the largest cardinal of $\M^\Tt_0\wupto\lh(E^\Tt_0)$, in which
case
$\M^\Tt_2$ is active but $\OR(\M^\Tt_2)=\lh(E^\Tt_1)$, and therefore if $E^\Tt_2$ is defined then $\lh(E^\Tt_2)=\lh(E^\Tt_1)$. Because $E^\Tt_2$ is type 2, however,
if $E^\Tt_3$ is defined then $\lh(E^\Tt_2)<\lh(E^\Tt_3)$.

The preceding example is essentially general. It is easy to show that if $\Tt$ is
$k$-maximal and $\alpha<\beta<\lh(\Tt)$ then either
$\lh(E^\Tt_\alpha)<\OR(M^\Tt_\beta)$ and $\lh(E^\Tt_\alpha)$
is a cardinal of $M^\Tt_\beta$, or $\beta=\alpha+1$ and
$\lh(E^\Tt_\alpha)=\OR(M^\Tt_{\alpha+1})$ and $E^\Tt_\alpha$ is
superstrong and $M^\Tt_{\alpha+1}$ is type 2. Therefore if $\alpha+1<\beta+1<\lh(\Tt)$ then
$\nu(E^\Tt_\alpha)<\nu(E^\Tt_\beta)$, and if $\alpha+1\leq\beta<\lh(\Tt)$ then
$E^\Tt_\alpha\rest\nu(E^\Tt_\alpha)$ is not an initial segment of any extender on
$\es_+(M^\Tt_\beta)$.

The \index{comparison}comparison algorithm needs to be modified slightly. Suppose we are comparing models $\M,\N$,
via padded $k$-maximal trees $\Tt,\Uu$, respectively, and we have produced $\Tt\rest(\alpha+1)$
and $\Uu\rest(\alpha+1)$. Let $\gamma$ be least such that
$\M^\Tt_\alpha\wupto\gamma\neq\M^\Uu_\alpha\wupto\gamma$;
{\color{black}let us assume that $\gamma$ is a limit, so that this distinction is due to differing extenders indexed at $\gamma$.}
If only one of these models is active, then we use that active extender next. Suppose both are
active. If one active extender is type 3 and one is type 2, then we use only the type 3 extender
next. Otherwise we use both extenders next. With this modification, and with the remarks in the
preceding paragraph, the usual proof that comparison succeeds goes through. {\color{black}The reason we make
this modification is as follows.
Suppose we use the usual process, so that if both sides are active at height $\gamma$ (where the least disagreement was), we automatically use both of the disagreeing extenders. Let us use padding in the usual way for comparison. It might be that $E^\Tt_\alpha$ is superstrong and $E^\Uu_\alpha$ is type 2 (with $\lh(E^\Tt_\alpha)=\gamma=\lh(E^\Uu_\alpha)$),
and the situation described in the previous paragraph occurs in $\Tt$ at that stage, so that $\gamma=\OR(M^\Tt_{\alpha+1})$
and $M^\Tt_{\alpha+1}$ is active type 2. But then it might also be that $F(M^\Tt_{\alpha+1})=E^\Uu_\alpha$.
On the other hand, $\gamma<\OR(M^\Uu_{\alpha+1})$
(since $E^\Uu_\alpha$ is type 2),
so $E^\Tt_{\alpha+1}=F(M^\Tt_{\alpha+1})=E^\Uu_\alpha$.
So $\Tt,\Uu$ use identical extenders at stages $\alpha+1,\alpha$ respectively, which breaks the usual comparison arguments.
With the modified algorithm, we would set $E^\Tt_\alpha$ to be the superstrong extender (indexed at $\gamma$) and $E^\Uu_\alpha=\emptyset$, and if indeed $F(M^\Tt_{\alpha+1})=F^{M^\Uu_\alpha|\gamma}$, then the comparison would terminate at stage $\alpha+1$.}

{\color{black}Note that everything mentioned in this remark applies to standard premice with superstrong extenders, not just opms.}
\end{rem}

\begin{lem}\label{lem:k-stack-max_pres_rpm}
Let $\M$ be a $k$-relevant opm and $\Tt$ be {\color{black}a partial play of $\G^*_{\opt,k}(\M,\infty,\infty)$
of  successor length, with
$M^\Tt_\infty$ well-defined and wellfounded}. Then $M^\Tt_\infty$ is a $\deg^\Tt_\infty$-relevant
opm.
\end{lem}
\begin{proof}
Given the result for $k$-maximal trees $\Tt$, the generalization to {\color{black}stacks of the kind dealt with in the lemma} is routine.
For $k$-maximal $\Tt$, the result follows by a straightforward induction on $\lh(\Tt)$,
using Lemma
\ref{lem:Ult_k_pres_fs}, together with the following observation.
Suppose $\beta+1\in\mathscr{D}_{\deg}^\Tt$ and let $\alpha=\pred^\Tt(\alpha+1)$
and $\N=M^{*\Tt}_{\beta+1}$ and $n=\deg^\Tt(\beta+1)$.
We claim that $M^{*\Tt}_{\beta+1}$ is $n$-relevant,
which together with \ref{lem:Ult_k_pres_fs} suffices.
So suppose that $\N$ is a successor.
We have  $\kappa=\crit(E^\Tt_\beta)<\nu(E^\Tt_\alpha)$ and $\rho_{n+1}^{\N}\leq\kappa<\rho_n^{\N}$. But
$\lh(E^\Tt_\alpha)\leq\OR^{\N^-}$,
since $\es_+^\N$ has no elements indexed in the interval $(\OR^{\N^-},\OR^\N]$.
Since $\N^-\pins M^{*\Tt}_{\beta+1}$, we have $\kappa<\rho_\om^{\N^-}$.
Therefore $\rho_{n+1}^\N\leq\kappa<\rho_\om^{\N^-}$, so $\N$ is $n$-relevant.
\end{proof}

In \ref{lem:k-stack-max_pres_rpm}, it is important that we are considering $\G^*_{\opt,k}(\M,\infty,\infty)$; {\color{black}$M^\Tt_\infty$ can fail to be $\deg^\Tt_\infty$-relevant for trees produced by $\G^*_k(\M,\infty,\infty)$. (For example,  after the first round, producing last model/degree $(M^\Tt_\infty,\deg^\Tt_\infty)$,
player 1 could drop in model/degree
to some $(\Q,q)\ins(M^\Tt_\infty,\deg^\Tt_\infty)$
such that $\Q$ is a successor
and $\rho_{q+1}^\Q=\rho_\om^{\Q^-}$, and then in the second round, use just one extender $E$ applied to $Q$
with $\crit(E)<\rho_\om^{\Q^-}$, forming $\Ult_q(\Q,E)$. If the ultrapower map is continuous at $\rho_\om^{Q^-}$, then $\Ult_q(\Q,E)$
is not $q$-relevant.)}

\section{${\Fop}$-mice for operators ${\Fop}$}

\label{sec:operators}
Operator-premice $\M$ are generally considered in the case that  successor steps are taken by some
\emph{operator}
${\Fop}$ (Definition \ref{dfn:operator}); that is, in which $\N={\Fop}(\N^-)$ for each successor $\N\wins\M$.
We call such an $\M$ an \emph{${\Fop}$-premouse}. A key example  is that of \emph{mouse operators}, for which we have some formula $\varphi$
and  $\Fop(\N^-)$ is, roughly, the least mouse $\R$ over $\N^-$  such that either $\R\sats\varphi$ or $\R$ projects $<\rho_\om(\N^-)$ (but $\R$ must be coded appropriately so that $\Fop(\N^-)$ is an opm; see Definition \ref{dfn:mouse_op} for details). One can also use the operator framework to
define (iteration) \emph{strategy mice}, although a different approach is taken in
\cite{scales_in_hybrid_mice_over_R} (to give a more refined hierarchy).

\subsection{Abstract operators $\Fop$ and $\Fop$-premice}

\begin{dfn} We say that $X$ is \index{swo'd}\index{self-wellordered}\dfnemph{swo'd (self-wellordered)} iff
$X=x\un\{x,<\}$ for some transitive set $x$, and wellorder $<$ of $x$. In this
situation, $<_X$ denotes the wellorder of $X$ extending $<$, and with
last two elements $x,<$. Clearly there are uniform methods of passing from a
swo'd $X$ to a wellorder of $A={\hat{X}}$. Fix such a method, and for such $X,A$, let $<_A$ denote
the resulting wellorder of $A$.
\end{dfn}

{\color{black}The domains of our operators will be  \emph{operatic domains}, which will be certain kinds of subsets of \emph{operator backgrounds}.
The set $\her_\kappa$ of sets hereditarily of cardinality ${<\kappa}$) is a basic example of an operator background:}
\begin{dfn}\label{dfn:operator_bkgd}
We say that a set or class $\opbk$ is an \index{operator background}\dfnemph{operator background} iff (i) $\opbk$ is
transitive, rudimentarily closed and $\om\in\opbk$, (ii) for all $x\in\opbk$ and all $y,f$, if
$f\maps x^{<\om}\to\trancl(y)$ is a surjection then $y\in\opbk$, and (iii) $\opbk\sats\DC$. (So
$\OR(\opbk)=\rank(\opbk)$ is a cardinal; if $\om<\kappa\leq\Ord$  then $\her_\kappa$ is an operator
background {\color{black}(note that this only uses ZF, since $\kappa$ is an ordinal)}, and under $\ZFC$ these are the only operator backgrounds.) By
(iii), every element of $\opbk$ has a countable elementary substructure in $\opbk$.

Let $\opbk$ be an operator background. A set $C$ is a \index{cone}\dfnemph{cone of
$\opbk$} iff there is $a\in\opbk$ such that $C$ is the set of all
$x\in\opbk$ such that $a\in\J_1(\hat{x})$.
With $a,C$ as such, we say $C$ is \dfnemph{the cone above $a$}. If $b\in\J_1(a)$
we say $C$ is \dfnemph{above $b$}. A set $D$
is a \dfnemph{swo'd cone of $\opbk$}
iff $D=C\inter S$, for some cone $C$ of $\opbk$, and where $S$ is the
class of swo'd sets. Here $D$ is
\dfnemph{(the swo'd cone) above $a$} iff $C$ is
(the cone) above $a$. A \dfnemph{cone} is a
cone of $\opbk$ for some operator background $\opbk$. Likewise for \dfnemph{swo'd cone}.
\end{dfn}

\begin{dfn}\label{dfn:operatic_domain}
Let $\opbk$ be an operator background.
For $C\sub\opbk$, let\index{$\witri{C}$}
\[ \witri{C}=\{\hat{Y}\bigm| Y\in C\wedge Y\text{ is transitive}\}. \]
An \index{operatic domain}\dfnemph{operatic domain over $\opbk$} is a set $D=\witri{C}\un P\sub\opbk$,
where $C$ is {\color{black}a  cone of $\opbk$
or a swo'd cone of $\opbk$}, and $P$ is some class of ${<\om}$-condensing
$\om$-sound
opms, each over
some $A\in \witri{C}$. (We do not make any closure requirements on $P$.) Write
\index{$C^D$}\index{$P^D$}$C^D=C$ and $P^D=P$. Note that $\witri{C}\inter P=\emptyset$.

An \dfnemph{operatic domain} is an operatic domain over some such $\mathscr{B}$.
\end{dfn}

We can now define \emph{operators}:
\begin{dfn}\label{dfn:operator}
Let $\opbk$ be an operator background. An \index{operator}\dfnemph{operator over $\opbk$ with
domain $D$} is a function ${\Fop}:D\to\opbk$ such that (i) $D$ is an operatic
domain over $\opbk$;
(ii) for all $X\in D$, $\M={\Fop}(X)$ is a successor opm with $\M^-=X$ (so if $X\in\witri{C^D}$ then
$\wl(\M)=1$ and $\hmb^\M=X$,
whereas if $X\in P^D$ then $\wl(\M)=\wl(X)+1$
and $X\pins\M$).
Write \index{$C^{\Fop}$}\index{$P^{\Fop}$}$C^{\Fop}=C^D$ and $P^{\Fop}=P^D$.
\end{dfn}

\begin{rem}
The argument $X$ to an operator should be thought of as having one of two possible types.
It is a \emph{coarse object} if $X\in\witri{C^{{\Fop}}}$; it is an opm if $X\in P^{\Fop}$.
Some natural operators ${\Fop}$ have the
property that, given
$\N\in P^{\Fop}$ (so $\widehat{\N}\in \witri{C^{\Fop}}$), ${\Fop}(\widehat{\N})$ is inter-computable with
${\Fop}(\N)$. But operators producing strategy mice in the ``least branch'' (or ``least tree'') hierarchy do not have this property. {\color{black}(For in that case, $\Fop(\N)$
is defined by first identifying the ``least tree'' $\Tt_\N$ for which a branch
must be added, and then adding the correct branch through that tree;
this depends of course on $S^\N$,
which indicates which trees have already been dealt with. On the other hand, $\witri{\N}$
is treated as a coarse object, so when defining $\Fop(\witri{\N})$,
$S^\N$ is irrelevant, and
the ``least tree'' $\Tt_{\witri{\N}}$ chosen
will likely be different from $\Tt_{\N}$.)}
\end{rem}

The simplest operator is essentially the rudimentary closure operator $\J$:

\begin{dfn}\label{dfn:J^op} {\color{black}Assume $\DC$.}
Let $\hmpv\in V$.
Let $C_\hmpv$ be the class of all $x$ such that $\hmpv\in\J_1(\hat{x})$.
Let $P_\hmpv$ be the class of all ${<\om}$-condensing $\om$-sound opms $\R$ over some
$Y\in\witri{C_\hmpv}$,
with $\hmp^\R=\hmpv$.
Then $\J^\op_{\hmpv}$ denotes the operator over $V$
with domain $D=\witri{C_\hmpv}\un P_\hmpv$, where for $x\in D$, $\J^\op_\hmpv(x)$ is the passive
successor opm $\M$ with universe $\J_1(x)$ and $\M^-=x$ and $\hmp^\M=\hmpv$.\footnote{It is easy to
see that $\M$ is indeed an opm, so $\J^\op_\hmpv$ is an operator.} (So if $x\in\witri{C_\hmpv}$
then $\wl(\M)=1$ and $\hmb^\M=x$.)
Let $\J^\op=\J^\op_\emptyset$.

{\color{black}Without assuming DC,
if $p\in\mathscr{H}_\kappa$,
then we can define, in the same manner, the operator $\J_{p,\mathscr{H}_\kappa}^\op$ over  $\mathscr{H}_\kappa$.
We might also just write $\J^\op_p$ for $\J_{p,\mathscr{H}_\kappa}^\op$.}
\end{dfn}

\begin{dfn}[${\Fop}$-premouse]\label{Fop-premouse}
For ${\Fop}$ an operator, an \index{$\Fop$-premouse, $\Fop$-pm}\dfnemph{${\Fop}$-premouse
\tu{(}${\Fop}$-pm\tu{)}}
is an opm $\M$
such that $\N={\Fop}(\N^-)$ for every successor $\N\wins\M$.
\end{dfn}

Let $\M$ be an ${\Fop}$-premouse, where ${\Fop}$ is an operator over $\opbk$. Note that
$\hmb^\M\in\witri{C^{\Fop}}$, as $\M\wupto1={\Fop}(\M\wupto0)$ and $\M\wupto0=\hmb^\M=\hat{x}$ for some
$x$, and
$\hat{x}\notin P^{\Fop}$. Note also that
$\OR(\M)\leq\OR(\opbk)$.

We now define ${\Fop}$-iterability for ${\Fop}$-premice $\M$, using Definition \ref{dfn:iteration_trees} (and hence continuing to follow \cite{steel2010outline}).
The main point is that the iteration strategy should produce ${\Fop}$-premice.
One needs to be a little careful, however,
because the background $\opbk$ for ${\Fop}$ might only be a set.
To simplify things, we restrict our attention to the case that $\M\in\opbk$.

\begin{dfn}\label{dfn:F-iterability}
Let ${\Fop}$ be an operator over $\opbk$.
Let $\M$ be an $\Fop$-premouse and let $\Tt$ be a putative iteration tree on $\M$.
We say that $\Tt$ is a \dfnemph{putative ${\Fop}$-iteration tree} iff $M^\Tt_\alpha$ is an
${\Fop}$-premouse for all $\alpha+1<\lh(\Tt)$.
We say that $\Tt$ is an \index{$\Fop$-iteration tree}\dfnemph{${\Fop}$-iteration tree}
iff $M^\Tt_\alpha$ is an ${\Fop}$-premouse for all $\alpha+1\leq\lh(\Tt)$.
We may drop the ``${\Fop}$-'' when it is clear from context.

Let $k<\om$ and let $\M\in\opbk$ be a $k$-sound ${\Fop}$-premouse.
Let $\theta\leq\OR(\opbk)+1$.
The iteration game
\index{$\G^{\Fop}_n(\M,\theta)$}$\G^{{\Fop}}_k(\M,\theta)$ has the rules of $\G_k(\M,\theta)$, except for the following
difference. Let $\Tt$ be the putative tree being produced.
For $\alpha+1\leq\theta$, if both players meet their requirements at all stages $<\alpha$, then,
in stage $\alpha$, player $\plII$ must first ensure that {\color{black}$M^\Tt_\alpha$ is wellfounded, and if $\alpha<\OR^{\opbk}$, that $M^\Tt_\alpha$ is an
${\Fop}$-premouse}. (Given this, if $\alpha+1<\theta$, player $\plI$ then selects
$E^\Tt_\alpha$.)
Thus, if we reach stage $\OR(\opbk)$, then after selecting a branch, player $\plII$ wins iff
$M^\Tt_{\OR(\opbk)}$ is wellfounded. (We cannot in general expect $M^\Tt_{\OR(\opbk)}$ to be an
${\Fop}$-premouse in this situation. For example, suppose that $\opbk=\HC$ and $\theta=\om_1+1$
and $\lh(\Tt)=\om_1+1$.
Then $M^\Tt_{\om_1}$ cannot be an ${\Fop}$-premouse, since all ${\Fop}$-premice have height
$\leq\om_1$.
But in applications such as comparison, we only need to know that $M^\Tt_{\om_1}$
is wellfounded. So we still decide the game in favour of player $\plII$ in this situation.)

Let $\lambda,\alpha\leq\OR(\opbk)$, and suppose that either $\OR(\opbk)$ is regular or
$\lambda<\OR(\opbk)$. Let $\theta\leq\lambda+1$.
The iteration game
\index{$\G^{*\Fop}_n(\M,\alpha,\theta)$}${\color{black}\G}^{{\color{black}*}{\Fop}}_k(\M,\alpha,\theta)$ is defined just as ${\color{black}\G}^{\color{black}*}_k(\M,\alpha,\theta)$,
with the differences that (i) the rounds are runs of $\G^{{\Fop}}_q(\Q,\theta)$ for some
$\Q,q$, and
(ii) if $\alpha$ is a limit and neither player breaks
any rule, and $\Ttvec$ is the sequence of trees played, then player $\plII$ wins iff
$M^{\Ttvec}_\infty$ is well-defined, wellfounded,\footnote{It follows
that {\color{black}if $\alpha=\OR(\opbk)$} then $M^{\Ttvec}_\infty\wupto\OR(\opbk)$ is an ${\Fop}$-premouse.} and if $\alpha<\OR(\opbk)$ then $M^{\Ttvec}_\infty$ is an
${\Fop}$-premouse.
(By some straightforward calculations using the restrictions on $\alpha,\theta$,
one can see that for any $\gamma<\alpha$,
if neither player has lost the game after the first $\gamma$
rounds,
and $\Ttvec\rest\gamma$ is the sequence of trees played thus far,
then $M^{\Ttvec\rest\gamma}_\infty\in\opbk$ and $M^{\Ttvec\rest\gamma}_\infty$ is an
${\Fop}$-premouse, so $\G^{{\Fop}}_q(\Q,\theta)$ is defined for the relevant $(\Q,q)$. This uses the
rule that if one of the rounds produces a tree of length $\theta$, then the game terminates.)

\index{$\G^{*\Fop}_{\opt,n}(\M,\alpha,\theta)$}${\color{black}\G}_{\opt,k}^{{\color{black}*}{\Fop}}(\M,\alpha,\theta)$ is likewise defined by analogy with ${\color{black}\G}^{\color{black}*}_{\opt,k}(\M,\alpha,\theta)$.

An \index{$\Fop$-iterability}\dfnemph{${\Fop}$-$(k,\theta)$-iteration strategy} for $\M$ is
a winning strategy for player $\plII$ in $\G^{{\Fop}}(\M,k,\theta)$.
An \dfnemph{${\Fop}$-$(k,\alpha,\theta)^*$-{\color{black}optimal} iteration strategy} for $\M$ is likewise for
$\G^{*{\Fop}}_{\opt}(\M,k,\alpha,\theta)$. And an
\dfnemph{${\Fop}$-$(k,\alpha,\theta)^*$-iteration strategy} is likewise for
$\G^{*{\Fop}}(\M,k,\alpha,\theta)$.

Now \dfnemph{${\Fop}$-$(k,\theta)$-iterability}, etc, are defined in the obvious manner.
\end{dfn}

\subsection{Coarse condensation of operators}
In order to prove that ${\Fop}$-premice built by background constructions are ${\Fop}$-iterable, we
will need to know that ${\Fop}$ has good \emph{condensation} properties, which roughly  demand that  elementary hulls of structures $\Fop(\M)$ should have form $\Fop(\M')$. (But we will also need to consider variants thereof.)

\begin{dfn}
 Let $\pi:\M\to\N$ be an embedding and $b$ be transitive. We say that $\pi$ is \index{above}\dfnemph{above $b$}
iff
$b\un\{b\}\sub\dom(\pi)$ and $\pi\rest b\un\{b\}=\id$.
\end{dfn}

\begin{dfn}\label{dfn:condenses_coarsely}
 Let ${\Fop}$ be an operator over $\opbk$ {\color{black}with domain $D$. Suppose $C^D$ is the cone above some transitive $p\in\mathscr{B}$.}
We say that ${\Fop}$ \index{above}\index{coarse condensation}\index{condenses coarsely}\dfnemph{condenses coarsely} (or \dfnemph{${\Fop}$ has  coarse
condensation}) \dfnemph{above $p$} iff for
every successor ${\Fop}$-pm $\N$ {\color{black}(so  $p\in\J_1(\hmb^\N)$)},
every set-generic extension $V[G]$ of $V$
and all $\M,\pi\in V[G]$, if $\M$ is a successor opm, $\M^-\in V$ and
$\pi:\M\to\N$ is fully elementary and is above $p$,
then $\M$ is an
${\Fop}$-pm (so in particular,
{\color{black}$\hmb^\M\in\witri{C^D}$ and} $\M^-\in\dom({\Fop})$ and $\M={\Fop}(\M^-)\in V$).

We say that ${\Fop}$ \index{almost condenses coarsely}\dfnemph{almost condenses coarsely above $p$} iff the preceding holds for
$G=\emptyset$.
\end{dfn}

\begin{dfn}
 An operator ${\Fop}$ over $\opbk$ is \index{total operator}\dfnemph{total} iff $P^{\Fop}$ includes all ${<\om}$-condensing
$\om$-sound
${\Fop}$-pms in $\opbk$.
\end{dfn}

The following lemma is a standard kind of observation:
\begin{lem}\label{lem:almost_redundancy_coarsely}
 Let ${\Fop}$ be a total operator over $\mathscr{B}$ {\color{black} with domain $D$. Suppose that $C^D$ is the cone above some $p\in\HC$,
 and} that $\Fop$ almost condenses coarsely above
  $p$.
 Then ${\Fop}$ condenses coarsely above $p$.
\end{lem}
\begin{proof}[Proof sketch]
Suppose the lemma fails and let $\PP$ be a poset, and
$G\sub\PP$ be $V$-generic, such that in $V[G]$ there is a counterexample $\pi:\M\to\N$.
So $\N\in V$  is a successor $\Fop$-pm (so $\N\in\mathscr{B}$),
$\M\in V[G]$ is a successor opm,
 $\M^-\in V$, and $\pi\in V[G]$ with $\pi:\M\to\N$
 is fully elementary and is above $p$,
 but  $\M$ is not an $\Fop$-pm.
 {\color{black}By passing to proper segments of $\M,\N$ if needed, we
may assume that $\M^-$ is an ${\Fop}$-pm, and therefore that $\M^-\in\dom({\Fop})$.
So letting $\M'=\Fop(\M^-)$,
we have $\M\neq\M'$.}

Let $\PP'=\Coll(\om,\M'\un\N)$.
By $\Sigma^1_1$-absoluteness, we may assume that $\PP=\PP'$.
{\color{black}That is, if $H$ is $(V,\PP')$-generic then in $\N,\M',\M^-,p\in\HC^{V[H]}$, and since in $V[G]$ there is $(\P,\sigma)$
such that $\P$ is a successor opm such that $\P^-=\M^-$ and $\P\neq\M'$ and $\sigma:\P\to\N$ is elementary and above $p$, there is also such a pair $(\P,\sigma)\in V[H]$.}

{\color{black}Let $\varrho:\om\to p$ be a surjection.
Let $X=\M'\cup\N\cup p\cup\varrho\cup\{\M',\N,p,\varrho\}$, so $X\in\mathscr{B}$.
We can fix $\eta\in\OR$ such that $L_\eta(X)\sats\ZF^{-\varepsilon}$,
and in fact by condensation,
taking the least such $\eta$, we have
 $L_\eta(X)\in\mathscr{B}$.

So $\PP=\Coll(\om,\M'\cup\N)\in L_\eta(X)$ and $L_\eta(X)\sats$ ``It is forced by $\PP$ that there
is $(\P,\sigma)$ such that $\P$ is a successor opm with $\P^-=\M^-$ but $\P\neq\M'$ and $\sigma:\P\to\N$ is elementary and above $p$.'' Because
$\opbk\sats\DC$, we can take a countable elementary hull of $L_\eta(X)$, such that letting
$\tau:L_{\bar{\eta}}(\bar{X})\to L_\eta(X)$ be the uncollapse map, $\tau(\bar{X})=X$ and $\rg(\tau)$ includes all relevant objects and all points
in $p\cup\{p\}$.
Write $\pi(\bar{\PP})=\PP$, etc.
Fix
 $g\in V$ which is $(L_{\bar{\eta}}(\bar{X}),\bar{\PP})$-generic. Then in $L_{\bar{\eta}}(\bar{X})[g]$,
 we have some opm $\P$ such that $\P^-=\bar{\M}^-$ but $\P\neq\bar{\M}'$, and some elementary $\sigma:\P\to\bar{\N}$
 which is above $\bar{p}=p$.
 Since $\tau\rest\bar{\N}:\bar{\N}\to\N$ and $\tau\rest\bar{\M'}:\bar{\M}'\to\M'=\Fop(\M^-)$ are elementary and above $p$,
 and $\Fop$ almost condenses coarsely above $p$, $\bar{\N}$ and $\bar{\M}'$ are $\Fop$-premice.
 But then similarly, because we have $\sigma$, $\P$ is an $\Fop$-premouse, so $\P=\Fop(\bar{\M}^-)=\bar{\M}'$,
 a contradiction.}
 \end{proof}
\subsection{Operators which (don't) condense well}
\begin{rem}\label{rem:problems_with_condenses_well}
So far the only example of an operator we have formally defined is that of $\J$.
In Definition \ref{dfn:mouse_op} below, we will introduce a more general class of examples, \emph{mouse operators}.
This will be a modification of the notion \emph{model operator} from (see \cite[Definition 2.1.8]{wilson2012contributions},
where such objects are denoted $F_K$).  But as we describe below,
the \emph{model operators} defined in \cite[2.1.8]{wilson2012contributions} have a minor problem, which we rectify here. (The notion \emph{mouse operator} as defined in \cite[Definition 2.1.7]{wilson2012contributions} is distinct from both of these.)

We will then proceed toward the central notion of \emph{condenses finely},
a refinement of \emph{condenses coarsely}.
This notion is based on that of \emph{condenses well}, from \cite[Definition 1.3.2]{cmi} and \cite[Definition 2.1.10]{wilson2012contributions}. We will modify \emph{condenses well} in several respects, for multiple reasons. The main changes will be motivated by the following discussion.
We can demonstrate a concrete
problem with \emph{condenses well}, at least when it is used in concert with
other definitions in \cite{wilson2012contributions}.
The following discussion uses the definitions and notation of
\cite[\S2]{wilson2012contributions},
without further explanation here; the terminology differs from this paper. (The remainder of
this remark is for motivation only; nothing in it is needed later.)

Let $K$ be the function $x\mapsto\J_2(x)$. Clearly $K$ is a mouse operator (see
\cite[2.1.7]{wilson2012contributions}). Let $F=F_K$
(see \cite[2.1.8]{wilson2012contributions}). Then we claim:
\begin{enumerate}
\item {\color{black}assuming that ``$n$th master code'' has a standard interpretation in \cite[2.1.8]{wilson2012contributions}, $F$ is not well-defined\footnote{This is a minor point, and is easily rectified, by following the form of \emph{mouse operator} from \cite[bullet (1) after Definition 1.3.1]{cmi}.}
(contrary to \cite[2.1.8, 2.1.12]{wilson2012contributions}),} and
\item {\color{black}modifying the definition of $F$ in the natural way so as to produce a (well-defined) model operator $F'$,} $F'$ does not
condense well
(contrary to {\color{black}the spirit of} \cite[2.1.12]{wilson2012contributions}).
\end{enumerate}
Let us verify this.

{\color{black}The fact that $F$ is not well-defined is just because in \cite[2.1.8]{wilson2012contributions},
in case 2 of the definition of $F(\M)$, the universe of $F(\M)$ is taken to be the $n$th master code of $\J_2(\M)|\xi$, for the relevant $n<\om$. Here, as we are in case 2, we have $\J_2(\M)|\xi=\J_1(\M)$. Now it can be that $n>0$ and $\rho(\M)<\OR^\M$,
and then the $n$th master code
(if this is interpreted in a standard kind of fashion) has ordinal height $\rho(\M)$, and its universe does not even include all of the universe of $\M$. But then it does not make sense to define $\dot{S}^{F(\M)}=S^\M\conc\left<\M\right>$,
as is written in \cite[2.1.8]{wilson2012contributions}.

So let us consider the modification
of the definition of $F=F_K$,
where instead of using a master code in case 2, we define
\[ F'(\M)=(\J_2(\M)|\xi;{\in},\emptyset,\emptyset,\dot{S}^{\M}\conc\left<\M\right>,\ell(\M)+1,a).\]
(This wouldn't  work for
 mouse operators in general,
 but we only consider the mouse operator $K$ for this discussion.)
In case 1, keep $F'(\M)=F(\M)$
as defined in \cite[2.1.8]{wilson2012contributions}.}

Then $F'$ is a model operator,
and seems to carry the  meaning
intended in \cite[2.1.8]{wilson2012contributions}.
{\color{black}(The adjustment in the definition brings it, moreover,
 in line with the definition of \emph{mouse operator} in \cite[bullet (1) after Definition 1.3.1]{cmi}.)}
But $F'$ does not in general  condense well.
For clearly regular premice $\M$ whose ordinals are closed under ``$+\om$'' can be
arranged as models $\tilde{\M}$ with parameter $\emptyset$ (see
\cite[Definition 2.1.1]{wilson2012contributions}), such that for each
$\alpha<\l(\tilde{\M})$,
$\tilde{\M}|(\alpha+1)=F'(\tilde{\M}|\alpha)$.
Now let $\M$ be a premouse such that for some
$\kappa<\OR(\M)$, $\kappa$ is measurable in $\M$, via some measure on
$\es=\es^\M$, and $\M\sats$``$\lambda=\kappa^{+\kappa}$ exists'',
$\rho_\om^\M=\lambda$, and $\M=\J_1(\M_0)$ where $\M_0=\J^\es_\lambda$.
Let $\M^*=\J(\tilde{\M_0})$, arranged as a model with parameter
$\emptyset$ extending $\tilde{\M_0}$. We have
$\rho_\om^\M=\lambda=\rho(\M_0)$ and $\tilde{\M_0}\in\M^*\in
F'(\tilde{\M_0})$ and
$\l(\M^*)=\lambda+1$ and $(\M^*)^-=\tilde{\M_0}$ (see
\cite[Definition 2.1.3]{wilson2012contributions}).
(We can't say
$\M^*=\tilde{\M}$, because $\tilde{\M}$ is not defined.)

Let $E\in\es$ be $\M$-total with $\crit(E)=\kappa$. Let $\N=\Ult_0(\M,E)$
and $\pi=i_E$. Then
$\rho_1^\N=\sup\pi``\lambda<\pi(\lambda)$. Let $\N_0=\pi(\M_0)$ and
$\N^*=\J_1(\tilde{\N_0})$, arranged as a model with parameter $\emptyset$
extending $\tilde{\N_0}$. Then $\rho_1(\N^*)<\pi(\lambda)=\rho(\tilde{\N_0})$,
and therefore $\N^*=F'(\tilde{\N_0})$. But $\pi:\M^*\to\N^*$ is a $0$-embedding
(and $\pi(\tilde{\M_0})=\tilde{\N_0}$). Since $\M^*\neq F(\tilde{\M_0})$, $F'$
does not condense well (see \cite[2.1.10(1)]{wilson2012contributions}).
{\color{black}Note that, in fact,
$\Ult_1(\M,E)=\Ult_0(\M,E)$ and
$\pi$ is both a $0$-embedding
and a $1$-embedding,
since for all $\bfrSigma_1^{\M}$ functions $f:\kappa\to\M$ there
is a  measure one set $X\in\M$ such that $f\rest X\in\M$.
So $\pi$ is also $\rSigma_2$-elementary, even though $\M^*\neq F'(\tilde{\M_0})$.)}

But note  that in the example above, $\M$ is not $0$-relevant,
nor $k$-relevant for any $k<\om$.
This motivates our focus on $k$-relevant opms.
We now give a second example, and one in which the embedding is the kind that can arise in the proof
of solidity of the standard parameter -- certainly in this context we would want to make
use of \emph{condenses well}. We claim there are (consistently) mice $\M$, containing large
cardinals, and $\rho,\alpha\in\Ord^\M$ such that:
\begin{enumerate}[label=--]
 \item $\M=\J(\N)$ where $\N=\M|\rho^{+\M}$,
 \item $\M$ is $1$-sound,
 \item $\rho_1^\M=\rho<\alpha<\rho^{+\M}$,
 \item $p_1^\M=\{\rho^{+\M},\alpha\}$, and
 \item letting $\Hh=\cHull_1^\M(\alpha\un\{\rho^{+\M}\})$, we have $\rho_\om^\Hh=\alpha$.
\end{enumerate}
Given such  $\M,\rho,\alpha,\Hh$, note that
$\alpha=\rho^{+\Hh}$ and $\Hh=\J(\M||\alpha)$.
Then $\Hh$ is a $1$-solidity witness for $\M$, and the $0$-embedding $\pi:\Hh\to \M$ is the one
that
would be used in the proof of the $1$-solidity of $\M$.
Moreover, with $F'$ as before, ``$\M=\J(\N)=F'(\N)$'' (since $\M$ projects below $\Ord^\N$) but
``$\Hh\neq F'(\M||\alpha)=\J(\J(\M||\alpha))$''. So we again have a failure of \emph{condenses
well}, and
one
which is arising in the context of the proof of solidity. (Of course, in the example we are already
assuming $1$-solidity, but the example seems to indicate that we cannot really expect to use
\emph{condenses well} in the proof of solidity for $F'$-mice.)

Now let us verify that such an $\M$ exists. Let $\P$ be any mouse (with large cardinals) and $\rho$
a
cardinal of $\P$ such that $\rho^{++\P}<\Ord^\P$. Let $\gamma=\rho^{+\P}+1$. For
$\alpha<\rho^{+\P}$
let
\[ \Hh_\alpha=\cHull_1^{\P|\gamma}(\alpha\un\{\rho^{+\P}\}). \]
Because $\rho_\om^{\P|\gamma}=\rho^{+\P}$, there is $\alpha$ with
$\rho<\alpha<\rho^{+\P}$ and such that the
uncollapse map $\pi_\alpha:\Hh_\alpha\to \P|\gamma$ is fully
elementary, and so $\rho_\om^{\Hh_\alpha}=\alpha=\rho^{+\Hh_\alpha}$.
{\color{black}(In $\P$, there is a club $C\sub\rho^{+\P}$ of ordinals $\alpha$ such that $\crit(\pi_\alpha)=\alpha$ and $\pi(\alpha)=\rho^{+\P}$. But $\P|\gamma=\Hull_1^{\P|\gamma}(\rho^{+\P}\cup\{\rho^{+\P}\})$,
so considering Tarski-Vaught
a straightforward closure argument
yields  a club $C'\sub C$
such that $\pi_\alpha$ is fully elementary for each $\alpha\in C'$.)\label{c41}}
Fix such an $\alpha$. Let $\Hh=\Hh_\alpha$ and
\[\M=\cHull_1^{\P|\gamma}(\rho\un\{\rho^{+\P},\alpha\}). \]
We claim that $\M,\rho,\alpha$ are as required. For we have $\M\in
\P$, which easily gives that $\rho_1^\M=\rho$. Clearly $\M=\J(\N)$ where $\N=\M|\rho^{+\M}$. The
$1$-solidity witness associated to $\rho^{+\M}$ is
$\cHull_1^\M(\rho^{+\M})$,
which is just $\M|\rho^{+\M}$, as $\M|\rho^{+\M}\elem_1 \M$, as $\M|\rho^{+\M}\sats\ZFmin$.
And the $1$-solidity witness associated to $\alpha$ is
$\cHull_1^\M(\alpha\un\{\rho^{+\M}\})$,
which is just $\Hh=\J(\P||\alpha)\in\M$. All of the
required properties follow.

The preceding examples seem to extend to any (first-order) mouse operator $K$
such that $\J(x)\in K(x)$ for all $x$.

To get around the problem just described, we will need to weaken the conclusion of \emph{condenses
well}, as will be seen.

{\color{black}Our second modification to the definition of \emph{condenses well}\label{c43}} is not based on a definite problem, but on a suspicion. It relates to,
in the notation used in clause (2) of \cite[2.1.10]{wilson2012contributions}, the embedding
$\sigma:F(\P_0)\to\M$.
In at least the basic situations in which one would want to use this clause
(or its analogue in \emph{condenses finely}), $\sigma$ actually arises from something like an
iteration map. But in \emph{condenses well}, no hypothesis along these lines regarding $\sigma$ is
made.
It seems that this could be a deficit, as it might be that $F(\P_0)$ is lower than $\M$ in the
mouse order (if one can make sense of this); we might have $F(\P_0)\pins\M$. Thus, it seems that
in proving an operator condenses well, one might struggle to make use of the existence of
$\sigma$.
So, in \emph{condenses finely}, we make stronger demands on $\sigma$.

A third change is that we do not require that $\pi\com\sigma\in V$ (with $\pi,\sigma$ as in
\cite[2.1.10]{wilson2012contributions}). This is explained toward the end of
\ref{rem:condenses_finely_beyond_almost}.

Motivation for the remaining details will be provided by how they arise later, in our proof of the
fundamental fine structural properties for ${\Fop}$-mice for operators ${\Fop}$ which condense finely,
and in our proof that mouse operators condense finely. We now return to our terminology and
notation. Before we can define \emph{condenses finely}, we need to set up some terminology in order
to describe the demands on $\sigma$.
\end{rem}

\subsection{Mouse operators}
In this section, in Definition \ref{dfn:mouse_op}, we will define mouse operators,
as an instance of a somewhat more general kind of operators (those of form $\Fop_G$; see \ref{dfn:operator_Fop_G}). These are variants of  the \emph{model operators} of \cite[2.1.8]{wilson2012contributions},
but in view of Remark \ref{rem:problems_with_condenses_well}, the details must be modified somewhat.
Our definition of mouse operators will be based on  \emph{op-$\J$-structures}. An op-$\J$-structure
will be used to form one step in the $\Fop_G$-hierarchy. Being a $\J$-structure, it has its own internal hierarchy, which will provide the stratification needed for opms:

\begin{dfn}[op-$\J$-structure]\label{dfn:op-J-structure}\index{op-$\J$-structure}
Let $\alpha\in\Ord\cut\{0\}$, and let $Y$ be such that either $Y=\hat{Z}$ for a transitive $Z$,
or $Y$ is a {\color{black}${<\om}$-condensing} $\om$-sound opm. Let
\[ D=\Lim\inter[\OR^Y+\om,\OR^Y+\om\alpha) \]
and let $\hmPvec=\left<\hmP_\beta\right>_{\beta\in D}$ be given.

We define $\J_\beta^{\hmPvec}(Y)$ for $\beta\in[1,\alpha]$, if possible, by recursion on
$\beta$, as follows. We set
$\J_1^{\hmPvec}(Y)=\J(Y)$ and take unions at limit $\beta$.
For $\beta+1\in[2,\alpha]$, let $R=\J_\beta^{\Pvec}(Y)$ and suppose that
$\hmP\eqdef\hmP_{\OR^R}\sub R$ and is amenable to $R$. In this case we define
\[
\J_{\beta+1}^{\hmPvec}(Y)=\J(R,\hmPvec\rest R,\hmP).
 \]
Note then that by induction,
$\hmPvec\rest R\sub R$ and $\hmPvec\rest R$ is amenable to
$R$.

Let \index{$\Ll_\J$}$\Ll_{\J}$ be the language with binary relation
symbol
$\dot{\in}$,
predicate
symbols $\dot{\hmPvec}$ and
$\hmPdot$, and constant symbol $\hmbdot$.

For $Y$ as above, an \dfnemph{op-$\J$-structure over $Y$} is an amenable
$\Ll_{\J}$-structure
\[ \M=(\J_{\alpha}^\hmPvec(Y),\in^\M,\hmPvec,\hmP,Y), \]
where $\alpha\in\Ord\cut\{0\}$ and $\hmPvec=\left<\hmPvec_{\gamma}\right>_{\gamma\in D}$
with domain $D$ defined as above,
$\univ{\M}=\J_\alpha^{\hmPvec}(Y)$ is defined,  $\dot{\hmPvec}^\M=\hmPvec$,
$\hmPdot^\M=\hmP$, $\hmbdot^\M=Y$.

Let $\M$ be an op-$\J$-structure, and adopt the notation above. Let
$\l(\M)$ denote $\alpha$. For
$\beta\in[1,\alpha]$ and $R=\J_\beta^{\hmPvec}(Y)$ and $\gamma=\OR^R$, let
\[
\M\Jupto\gamma=(R,\in^R,\hmPvec\rest R,\hmP_\gamma,Y). \]
Write $\N\Jins\M$ iff
$\N=\M\Jupto\gamma$ for
some $\gamma$. Clearly if $\N\Jins\M$ then
$\N$ is
an op-$\J$-structure over $Y$.
Write $\N\Jpins\M$ iff
$\N\Jins\M$
but $\N\neq\M$.

Let $\M$ be an op-$\J$-structure. Note that $\M$ is pre-fine {\color{black}(see Definition \ref{dfn:pre-fine})}.
We define the \index{fine structural notions}\dfnemph{fine structural notions} for
$\M$ using \ref{dfn:fine_structure}.
\end{dfn}

 From now on we omit ``$\in$'' from our notation for op-$\J$-structures.
 In what follows,
recall that \emph{operator background}, \emph{operatic domain},  $\witri{C^D}$
and $P^D$ were introduced in
Definitions \ref{dfn:operator_bkgd} and \ref{dfn:operatic_domain}.

\begin{dfn}[Pre-operator]\label{dfn:pre-op}
Let $\opbk$ be an operator background.
A \index{pre-operator}\dfnemph{pre-operator over $\opbk$}
is a function
$G:D\to\mathscr{B}$,
for some operatic domain $D$ over $\opbk$
such that for each $Y\in D$, $G(Y)$ is an op-$\J$-structure $\M$
over $Y$
such that
(i) every $\N\ins\M$ is $\om$-sound, and (ii) for some
$n<\om$,
$\rho_{n+1}^\M=\om$.\footnote{Recall from \ref{dfn:fine_structure} that $\rho_{n+1}^\M=\om$
does not mean that there is a new subset of $\om$ definable from parameters over $\M$, but just a new subset of $\om\cross(Y\cup\{Y\})^{<\om}$.}
Recalling that $D=\witri{C^D}\cup P^D$,
let $C^G=C^D$ and $P^G=P^D$.
\end{dfn}

{\color{black}We now want to derive an operator $\Fop_G$ from a pre-operator $G$. Say $\R$ is a sound $\Fop_G$-premouse,
over some set $A$, and we want to define $\Fop_G(\R)$. The initial hope is that this structure should be essentially equivalent to $G(\R)$, but with predicates reorganized appropriately. But this might not work, for two reasons. Most importantly, the resulting structure might fail projectum amenability;
that is, $G(\R)$ might contain subsets of $\rho_\om^\R\cross A^{<\om}$ which are not in $\R$.
In this case, we need to first replace $G(\R)$ with the largest $\J$-initial segment $G'(\R)$ of $G(\R)$ which does satisfy projectum amenability. And then, although $\rho_{n+1}^{G'(\R)}=\om$ for some $n<\om$, we cannot expect
that $\rho_1^{G'(\R)}=\om$. So we need to replace $G'(\R)$ with its $n$th reduct, for the appropriate $n$, and then code this as a successor opm.}

\begin{dfn}[Operator ${\Fop}_G$]\label{dfn:operator_Fop_G}
Let $G$ be a pre-operator over an operator background $\opbk$, with domain $D=\witri{C^D}\cup P^D$. We define
a corresponding operator \index{$\Fop_G$}\index{operator $\Fop_G$}${\Fop}={\Fop}_G$, also with domain $D$, as follows.

Let $X\in\witri{C^{D}}$ and
$\N=G(X)=(\univ{\N},\hmPvec^\N,\hmP^\N,X)$.
Let $n<\om$ be {\color{black}least} such
that $\rho_{n+1}^\N=\om$, {\color{black}so} $\OR^X<\sigma$ where $\sigma=\rho_{n}^\N$.
If $n=0$ then let $\M=\N$.
If $n>0$ then let $\Q=\N\Jupto\sigma$ and let $\M$ be the
op-$\J$-structure
\[ \M=(\univ{\Q},\hmPvec^\N\rest\sigma,T,X),\]
where $T\sub\univ{\Q}$
codes
$\Th_{n}^\N(\univ{\Q}\un\pvec_{n}^\N)$
in some uniform fashion, amenably to $\univ{\Q}$, such as with mastercodes.\footnote{\label{ftn:sub_constant_symbol}For
concreteness, we take $T$ to be the set of pairs $(\alpha,t')$ such that for some $t$,
$(\pvec_{n}^\M,\alpha,t)\in T_{n}^\M$, and $t'$ results from the theory $t$
by replacing each instance of $\pvec_{n}^\M$ {\color{black}in statements in $t$ with $\alpha$, interpreted as a constant symbol; note that if $(\pvec_n^\M,\alpha,t)\in T_n^\M$ then $\alpha$ does not already occur as a parameter in $t$, and this substitution neither obscures nor creates information.}}
Note that in either case, $\M=(\univ{\M},\hmPvec^\M,\hmP^\M,X)$ is an $\om$-sound
op-$\J$-structure over $X$ and $\rho_{1}^{\M}=\om$.
Now define ${\Fop}(X)$ as the hierarchical model $\K$ over $X$, of length $1$, with $\univ{\K}=\univ{\M}$,
$\hmE^{\K}=\emptyset=\hmp^\K$,\footnote{A natural generalization of this definition
would set $\hmp^\K$ to be some fixed non-empty object. For example, if one uses operators to
define strategy mice, one might set $\hmp^\K$ to be the structure that the iteration
strategy is for.} and
$\hmP^{\K}=\{X\}\cross(\hmPvec^\M\oplus\hmP^\M)$.
(We use $\{X\}\cross\cdots$ to ensure that $\hmP^\K\sub\K\cut\K^-$.
Recall that $\K$ having length $1$ requires that $\hmS^{\K}=\emptyset$.)

Now let $\R\in P^D$; we define ${\Fop}(\R)$.
Let
$A=\hmb^\R$ and $\rho=\rho_\om^\R$.
Let $\P=G(\R)$. Let $\N\ins\P$ be largest such
{\color{black}if $\rho>\om$ (so $\rho>\rank(A)$) then} for all $\alpha<\rho$, we have
$\pow(A^{<\om}\cross\alpha^{<\om})^\N=\pow(A^{<\om}\cross\alpha^{<\om}
)^\R$.
{\color{black}(Such an $\N$ exists, since $\J(\R)=\P|^\J(\OR^\R+\om)$ satisfies the requirements, by choice of $\rho$. Note that if $\rho=\om$ then $\N=\P$.)}
Let $n<\om$ be {\color{black} least} such that
$\rho_{n+1}^\N=\om$, {\color{black}so} $\OR^\R<\rho_{n}^\N$. Define $\M$ from $(\N,n)$ as in the
definition of ${\Fop}(X)$ for $X\in
\witri{C^{D}}$, but with $\hmb^\M=\R$.
Much as there, $\M=(\univ{\M},\hmPvec^\M,\hmP^\M,\R)$ is an $\om$-sound
op-$\J$-structure over $\R$ and $\rho_{1}^{\M}=\om$.

Now set ${\Fop}(\R)$ to be the unique
hierarchical model $\K$ of length $\wl(\R)+1$ with $\univ{\K}=\univ{\M}$,
$\R\wpins\K$ (so $\hmS^\K=\hmS^\R\conc\left<\R\right>$), $E^{\K}=\emptyset$, and
$\hmP^{\K}=\{\R\}\cross(\hmPvec^\M\oplus\hmP^\M)$. {\color{black}Let us also say that $\Fop(\R)$ \index{projects early}\dfnemph{projects early} if $\N\pins\P$ (in this case,
$\Fop(\R)$ does not ``reach'' the full $\P=G(\R)$, but just its initial segment $\N$).}
This completes the definition.
\end{dfn}

With notation as above, let $\R\in D$.
Note that ${\Fop}(\R)$ easily codes $G(\R)$,
\emph{unless} $\R\in P^D$
and $\Fop(\R)$ projects early. Let us verify that
${\Fop}_G$ is indeed an operator:

\begin{lem}\label{lem:operator_Fop_G}
 Let $G$ be a pre-operator over an operator background $\opbk$, with domain $D$. Then
 ${\Fop}_G$ is an operator over $\opbk$. Moreover, for any ${\Fop}_G$-premouse $\M$ of
length $\alpha+\om$, for all sufficiently large $n<\om$,
${\Fop}_G(\M\wupto(\alpha+n))$
does not project early.
\end{lem}
\begin{proof}[Proof sketch]
We first show that ${\Fop}_G$ is an operator.
Let ${\Fop}={\Fop}_G$ and $X\in D=\dom({\Fop})$. We must verify that
$\M={\Fop}(X)$ is an opm. This follows from (i) the choice of $\univ{{\Fop}(X)}$ (i.e.~the choice
of $\N\ins G(X)$ in the definition of ${\Fop}(X)$, which gives, for example, projectum
amenability
for ${\Fop}(X)$), (ii) if $X\in P^{D}$ then $X$ is
an $\om$-sound opm (acceptability follows from this and projectum amenability), (iii) standard
properties of
$\J$-structures (for example, to establish stratification), and (iv) with {\color{black}$\M$} as in the definition
${\Fop}(X)$ (either in case $X\in\witri{C^D}$ or in case $X\in P^D$), the fact that {\color{black}$\M$} is $\om$-sound
and $\rho_{1}^{\color{black}\M}=\om$ (for
sound projection).

Now let $\M$ be an ${\Fop}$-premouse of limit length $\alpha+\om$. Then for all $m$,
\[ \rho_\om^{\M\wupto(\alpha+m+1)}\leq\rho_\om^{\M\wupto(\alpha+m)}, \]
because $\M\wupto(\alpha+m+1)$
is soundly projecting and $\M\wupto(\alpha+m)$ is $\om$-sound. So if $n<\om$ is such that
$\rho_\om^{\M\wupto(\alpha+n)}$ is as small as possible, then $\Fop(\M|(\alpha+n))$ does not project early.
\end{proof}

So any limit length ${\Fop}_G$-premouse $\M$ is ``closed under $G$'' in the sense that for
$\in$-cofinally many $X\in\M$, we have $G(X)\in\M$.

We can now define mouse operators.

\begin{dfn}\label{dfn:mouse_op}
Let $\varphi\in\Ll_0$. Let $\opbk$ be an operator background. Suppose that for every transitive
structure $x\in\opbk$ there is $\M\pins\Lp(x)$ such that $\M\sats\varphi$, and let $\M_x$ be the
least such. Let ${\color{black}G_\varphi}:\opbk\partialto\opbk$ be the pre-operator
where for $x\in\opbk$ a transitive structure, ${\color{black}G_\varphi}(\hat{x})$
is the op-$\J$-structure over $\hat{x}$ naturally coding $\M_x$,
and for $x\in\opbk$ a ${<\om}$-condensing $\om$-sound opm, ${\color{black}G_\varphi}(x)$
is the op-$\J$-structure over $x$ naturally coding $\M_x$.
The \index{mouse operator}\dfnemph{mouse operator} ${\Fop}_\varphi$ determined by $\varphi$ is ${\Fop}_{G_\varphi}$.\end{dfn}

\begin{rem}
For example, suppose that $\M_1^\#(X)$ is defined and fully iterable for all sets $X$. 
Then $X\mapsto\M_1^\#(X)$ is a pre-operator $G_\varphi$,
for the obvious formula $\varphi$,
and $\Fop=\Fop_\varphi$ the induced mouse operator.
Let $\M$ be the least $\Fop$-premouse
which models $\ZFC^-$; so $\M\models$``Every set is countable'', and letting $\eta=\OR^\M$,
$\rho_\om^{\M|\alpha}=\om$ for all $\alpha<\eta$,
but $\rho_\om^{\M}=\eta$. Note that $\Fop(\M)$ projects early, and in fact $\rho_1^{\J(\M)}=\om$, so $\Fop(\M)$ is $\J(\M)$, reorganized as an $\Fop$-premouse.
But for no $\alpha<\eta$ does $\Fop(\M|\alpha)$ project early (since $\rho_\om^{\M|\alpha}=\om$ already),
so $\Fop(\M|\alpha)$ is equivalent to $\M_1^\#(\M|\alpha)$ for all $\alpha<\eta$.

There are only countably many mouse operators over a given $\mathscr{B}$, since each is determined by a formula $\varphi$. But by combining with real parameters (say specifying the base of a cone),
we obtain uncountably many operators. Assuming $\AD$ in $L(\RR)$, such an operator can be used to witness the $\Sigma_1$ truths about reals in a given $\J_\alpha(\RR)$, and that operator is in $\J_{\alpha'}(\RR)$ with an $\alpha'$ very close to $\alpha$.
\end{rem}

\subsection{Fine condensation}

{\color{black}In this section, in Definition \ref{dfn:condenses_finely}, we will define (almost) fine condensation. It will be the key property that ensures
that copying constructions
for iteration trees on $\Fop$-premice proceed in a desirable fashion; that is,
if we have $\Fop$-premice $\R,\Ss$ and an embedding $\tau:\R\to\Ss$,
and $\Uu$ on $\Ss$ is an $\Fop$-tree (that is, its models are $\Fop$-premice), and $\Uu$ is the copy of $\Tt$  under $\tau$, then we would like to know that $\Tt$ is also an $\Fop$-tree. Of course, we will
have the copy maps $\pi:\M\to\N$ from models $\M$ of $\Tt$ into models $\N$ of $\Uu$. (Almost) fine condensation will be applied to these copy maps, and this should allow us to conclude that $\M$ is an $\Fop$-pm.
The property should also guarantee similar behaviour for realization maps replacing copy maps.

We will also want to apply (almost) fine condensation to maps $\pi:\M\to\N$ such as core embeddings, or hull embeddings which arise in the proof of solidity of the standard parameter, for example.

Before giving the definition,
we will introduce some terminology allowing us to describe the kinds of embeddings $\pi:\M\to\N$ we want to consider.}

The definition of $(z_{k+1}^\M,\zeta_{k+1}^\M$) below is a direct adaptation from
\cite[Definition 2.19]{extmax}. The facts proved there about this notion generalize readily to
the
present setting, although that
paper formally works below superstrong. See also \cite[\S3]{fsfni}, where there is no superstrong restriction.

\begin{dfn}
 Let $\M$ be a $k$-sound opm. Let $\Dd$ be the class of pairs
$(z,\zeta)\in[\Ord]^{<\om}\cross\Ord$ such that  $\zeta\cap z=\emptyset$. For $x\in[\Ord]^{<\om}$ let
$f_x$ be the decreasing enumeration of $x$. For $x=(z,\zeta)\in\Dd$ let
$f_x=f_z\conc\left<\zeta\right>$.
Order $\Dd$ by $x<^*y$ iff $f_x<_{\mathrm{lex}}f_y$, with $x<^*y$ if $f_x\psub f_y$. Then \index{$z_{n}^\M,\zeta_{n}^\M$}$(z_{k+1}^\M,\zeta_{k+1}^\M)$ denotes the
$<^*$-least $(z,\zeta)\in\Dd$ such that
\[ \Th_{k+1}^\M(\hmb^\M\un z\un\zeta)\notin\M. \]
The \dfnemph{$(k+1)$-solid-core} of $\M$ is\index{$\mathfrak{S}_n(\M)$}
 \[ \mathfrak{S}_{k+1}(\M)=\cHull_{k+1}^\M(\hmb^\M\un z_{k+1}^\M\un\zeta_{k+1}^\M), \]
and the \dfnemph{$(k+1)$-solid-core map} \index{$\sigma_n^\M$}$\sigma_{k+1}^\M$ is the uncollapse map.
\end{dfn}

If $\M$ is $(k+1)$-solid then $\mathfrak{S}_{k+1}(\M)=\core_{k+1}(\M)$ and
$\sigma_{k+1}^\M$ is the core map. But we will need to consider the $(k+1)$-solid-core more
generally, in the proof of $(k+1)$-solidity.

{\color{black}Iteration maps, along a portion of a branch which does not drop in model, and is at degree $k$,
are \emph{$k$-tight} embeddings (but \emph{$k$-tight} is more general):}
\begin{dfn}\label{dfn:k-tight}
Let $k\leq\om$, let $\L,\M$ be $k$-sound opms and $\sigma:\L\to\M$. We say that $\sigma$ is
\index{tight}\dfnemph{$k$-tight} iff there is $\lambda\in\Ord$ and a sequence
$\left<\L_\alpha\right>_{\alpha\leq\lambda}$ of opms such that $\L=\L_0$ and $\M=\L_\lambda$ and
there is a sequence $\left<E_\alpha\right>_{\alpha<\lambda}$ of extenders such that each $E_\alpha$
is weakly amenable to $\L_\alpha$, with $\hmb^\L<\crit(E_\alpha)<\rho_k^{\L_\alpha}$,
\[ \L_{\alpha+1}=\Ult_k(\L_\alpha,E_\alpha), \]
and for limit $\eta$,
\[ \L_\eta=\dirlim_{\alpha<\beta<\eta}(\L_\alpha,\L_\beta;j_{\alpha\beta}) \]
where $j_{\alpha\beta}:\L_\alpha\to\L_\beta$ is the resulting ultrapower map,
and $\sigma=j_{0\lambda}$.
\end{dfn}

{\color{black}Note that $E_\alpha$ is not required to be close to $\L_\alpha$.

Copy maps and realization maps between $k$-sound structures are often \emph{$k$-factors}:}
\begin{dfn}\label{dfn:k-factor}
Let $k\leq\om$ and $\M,\N$ be $k$-sound opms and $p$ be transitive.
{\color{black}Suppose that if $k<\om$ then $\M$ is $k$-relevant.}

We say that $\pi:\core_0(\M)\to\core_0(\N)$ is a \index{above}\index{factor}\dfnemph{$k$-factor above $p$}
iff $\pi$ is a weak $k$-embedding above $p$, and if $k<\om$ then there is a $k$-tight
$\sigma:\core_0(\L)\to\core_0(\M)$ such that
\[ \pi\com\sigma\com\sigma^\L_{k+1}:\mathfrak{S}_{k+1}(\L)\to\core_0(\N) \]
is a near $k$-embedding, $\sigma$ is above $p$,
and $\L$ is $k$-relevant.

For an operator ${\Fop}$, a $k$-factor is \index{rooted}\dfnemph{${\Fop}$-rooted} iff either $k=\om$ or
we can take $\L$ to be an ${\Fop}$-premouse.

A $k$-factor is \index{good}\dfnemph{good} iff $A\eqdef\hmb^\M=\hmb^\N$ and $\pi$ is above $A$.
\end{dfn}

An $\om$-factor above $p$ is just an $\om$-embedding (i.e. fully elementary between
$\om$-sound opms) above $p$. If $k<\om$, then both $\sigma$ and $\sigma_{k+1}^\L$, and
therefore also $\sigma\com\sigma_{k+1}^\L$, are $k$-good. Any near $k$-embedding
$\pi:\M\to\N$ above $p$, between opms $\M,\N$, is a $k$-factor above $p$ (if $k<\om$, use $\L=\M$ and $\sigma=\id$), and if $\M$ is an ${\Fop}$-pm, then $\pi$ is
${\Fop}$-rooted.

\begin{dfn}
 Let $\C$ be a successor opm and $\M$ a successor Q-opm with $\C^-=\M^-$. We say that $\C$ is a
\index{universal hull}\dfnemph{universal hull} of $\M$
 iff there is an above $\C^-$, $0$-good embedding $\pi:\C\to\M$ and for every $x\in\M$,
$\Th_1^\M(\M^-\un\{x\})$ is $\bfrSigma_1^\C$ (after replacing $x$ with a constant symbol).
\end{dfn}

\begin{rem}
{\color{black}We are now ready to define \emph{\tu{(}almost\tu{)} fine condensation}. It is a variant of \emph{condenses well} from \cite{cmi} and \cite{wilson2012contributions}. As discussed in Remark \ref{rem:problems_with_condenses_well}, we need to modify that notion.

One issue that
 Remark \ref{rem:problems_with_condenses_well} illustrates is the following: Given a $\Sigma_1$-elementary $\pi:\M\to\N=\Fop(\N^-)$,
 we should not always expect that $\M=\Fop(\M^-)$, even in the case of a mouse operator $\Fop$.
 However, for a mouse operator $\Fop$, the iterability of $\N=\Fop(\N^-)$ above $\N^-$ and the existence of $\pi$ should ensure the iterability of $\M$ above $\M^-$. (Here the \emph{iterability} we refer to is that of the ordinary mouse over $\N^-$ produced by $\Fop$; the \emph{$\Fop$-iterability} of $\N$ above $\N^-$ is trivial, as $\es_+^\N$ is empty above $\OR^{\N^-}$.)
 Secondly, the minimality of $\N$ above $\N^-$ should ensure that $\M$ does not strictly surpass $\Fop(\M^-)$. But
 Remark \ref{rem:problems_with_condenses_well}
 indicates that $\M$ might not actually reach $\Fop(\M^-)$ in general, and for example, we might have $\M\in\Fop(\M^-)$.
 We might, for example,
 have that $\M$ is a proper segment of $\Fop(\M^-)$ in the hierarchy as an op-$\J$-structure, but there are also other possibilities.

 On the other hand,
we want (almost) fine condensation to hold under appropriate circumstances,
and in particular,
we want mouse operators to almost condense finely.
 So it is allowed that $\M\in\Fop(\M^-)$
 in (one case of)  the definition of \emph{condenses finely}.

In Proposition \ref{prop:mouse_op_con_finely},
we will show that mouse operators do almost condense finely,
and the proof will help to illuminate key details of the definition.
}
\end{rem}

\begin{dfn}\label{dfn:condenses_finely}
Let ${\Fop}$ be an operator {\color{black}over $\opbk$ with domain $D$.
Suppose that $C^D$ is the cone above some transitive $p\in\mathscr{B}$.}
We say that ${\Fop}$ \index{above}\index{condenses finely}\index{fine condensation}\dfnemph{condenses finely above $p$} (or ${\Fop}$ \dfnemph{has fine
condensation above $p$}) iff (i) ${\Fop}$ condenses coarsely
above $p$; and (ii) Let
$A,\Abar,\N,\L\in V$ and let $\M,\varphi,\sigma\in V[G]$ where $G$ is
set-generic over $V$. {\color{black}Let $k<\om$.} Suppose that:
\begin{enumerate}[label=--]
 \item $p\in\J_1(\Abar)\inter\J_1(A)$,
 \item  $\L$ is a $k$-sound opm over $\Abar$ and $\N$ is a $k$-sound
opm over $A$,
\item $\M$ is a Q-opm over $\Abar$  and {\color{black} if $k>0$ then $\M$ is a $k$-sound opm},
\item $\L,\M,\N$ each have successor length,
\item $\M^-\in V$ and $\L,\M^-,\N$ are ${\Fop}$-premice {\color{black}and $\M^-\in\dom(\Fop)$ (so $\Fop(\M^-)$ is an opm),} and
\item $\varphi:\core_0(\M)\to\core_0(\N)$.
\end{enumerate}
Then:
\begin{enumerate}
 \item\label{item:if_M_is_opm} If $\M$ is a $k$-sound opm  and  either
\begin{enumerate}[label=--]
 \item $\varphi$ is $k$-good, or
 \item $\M$ is $k$-relevant and $V[G]\sats$``$\varphi$ is a $k$-factor above $p$, as witnessed by
$(\L,\sigma)$'',
\end{enumerate}
then either $\M\in{\Fop}(\M^-)$ or $\M={\Fop}(\M^-)$.
\item\label{item:if_rho_1^M<=OR^M^-} If $k=0$ and $\rho_1^\M\leq\OR^{\M^-}$ and $\varphi$ is $0$-good (hence above $p$), then there is a universal hull $\Hh$
of $\M$ such that either $\Hh\in{\Fop}(\M^-)$ or $\Hh={\Fop}(\M^-)$.
\end{enumerate}

We say ${\Fop}$ \index{almost condenses finely}\dfnemph{almost condenses finely above $p$} iff ${\Fop}$ almost condenses
coarsely above $p$ and condition (ii) above holds
for $G=\emptyset$.
\end{dfn}

{\color{black}Recall that if $\M$ is a successor opm then $\rho_1^\M\leq\OR^{\M^-}$.
So in both parts \ref{item:if_M_is_opm} and \ref{item:if_rho_1^M<=OR^M^-}
above, we have $\rho_1^\M\leq\OR^{\M^-}$,
but in part  \ref{item:if_rho_1^M<=OR^M^-},
$\M$ need not be an opm (although it is a Q-opm).}
{\color{black}Also note that
 there are cases
of \emph{condenses
finely} in which we do not assume that $\M$ is $k$-relevant, though in these, $\varphi$ is $k$-good.}

Let us observe that in certain key circumstances,
we can rule out the possibility that $\M\in\Fop(\M^-)$,
and so the conclusion of fine condensation  sharpens to
$\M={\Fop}(\M^-)$:

\begin{lem}\label{lem:strengthen_scf_conclusion}
{\color{black}Let $\Fop$ be an operator.
Let $k<\om$.
Let $\N$ be an $\Fop$-pm
which is a $k$-sound successor opm.
Let $\M$ be a successor Q-opm.
Suppose $\M^-$ is an $\Fop$-pm in $\dom(\Fop)$.
If $k>0$ then suppose also that $\M$ is a $k$-sound opm.
Suppose that $\M=\core_{k+1}(\N)$
or $\M$ is $k$-relevant.
Then $\M\notin\Fop(\M^-)$,
and if $k=0$
then there is no universal hull of $\M$ in $\Fop(\M^-)$.}
\end{lem}
\begin{proof}
{\color{black}Suppose otherwise. Suppose first that $k>0$, so $\M\in\Fop(\M^-)$.
Then by projectum amenability for $\Fop(\M^-)$, $\M$ is not $k$-relevant. So
$\M=\core_{k+1}(\N)\notin\N$. Let ${\color{black}\pi}:\M\to\N$ be the core
map. By Lemmas \ref{lem:k+1-Hull} and \ref{lem:core_embedding}, ${\color{black}\pi}$ is $k$-good, so ${\color{black}\pi}(\M^-)=\N^-$.
So since $\M$ is not $k$-relevant, $\rho_{k+1}^{\M}=\rho_k^\M=\rho_\om^{\M^-}$,
but then $\rho_{k+1}^{\N}=\rho_k^\N$, so $\N$ is $(k+1)$-sound and $\M=\N=\Fop(\N^-)=\Fop(\M^-)$,
contradicting the assumption that $\M\in\Fop(\M^-)$.

So $k=0$. (So we do not assume $\M$ is an opm, but it is a Q-opm.) Again by projectum amenability, $\M$ is not $0$-relevant, so $\M=\core_1(\N)\notin\N$. Let $\pi:\M\to\N$ be the core map. Then $\pi$ is $0$-good, so $\pi(\M^-)=\N^-$. So $\rho_1^\M=\rho_\om^{\M^-}<\rho_0^\M=\OR^\M$,   and $\rho_1^\M=\rho_1^\N=\rho_\om^{\N^-}$. But then since $\N$ is an opm
and by Lemma \ref{lem:succ_opm_top_param},
$\N$ is $1$-sound, so
$\M=\N$, again a contradiction.}
\end{proof}

So under the circumstances of the lemma above, if $\M$ is an opm, fine condensation
gives the stronger
conclusion that $\M={\Fop}(\M^-)$. But we will need to apply fine
condensation more generally, such as in the proof of solidity.
Analogously to Lemma
\ref{lem:almost_redundancy_coarsely}, we have:

\begin{lem}\label{lem:almost_redundancy_finely} Let ${\Fop}$ be a total operator {\color{black}over $\mathscr{B}$,
	with domain $D$. Suppose $C^D$ is the cone above
	some transitive $p\in\HC$,}
	and that $\Fop$  almost condenses finely
	above  $p$. Then ${\Fop}$ condenses finely above $p$.
\end{lem}

\begin{prop}\label{prop:mouse_op_con_finely}
Let
 ${\Fop}_{\varphi}$ be a mouse operator, as in Definition \ref{dfn:mouse_op}. Then $\Fop_\varphi$ almost condenses finely.
\end{prop}
\begin{proof}[Proof sketch]
We just discuss
the proof in one case, which
illustrates the main points and should clarify why \emph{almost condenses finely} is formulated as it is.
Let $\Fop=\Fop_\varphi$
and let $\N$ be a successor $\Fop$-pm. Let $\M$ be a successor Q-opm with $\rho_1^\M\leq\OR^{\M^-}$
and let $\pi:\M\to\N$ be a $0$-embedding. Then we want to verify clause \ref{item:if_rho_1^M<=OR^M^-} of almost fine condensation holds with respect to $\M,\N,\pi$. So we need to see that  there is a universal hull $\Hh$
of $\M$ such that either $\Hh\in{\Fop}(\M^-)$ or $\Hh={\Fop}(\M^-)$. (Clause \ref{item:if_rho_1^M<=OR^M^-} also assumes that $\pi$ is $0$-good,
but that isn't needed here.)
Note that although $\M$ is a Q-opm, we do not assume  it is an opm.

We have $\pi(\M^-)=\N^-$.  Let
$\N^*\pins\Lp(\N^-)$ be the premouse over $\N^-$ coded by $\N$. ({\color{black}That is, either $\rho_1^{\N^*}=\om$ and $\OR^{\N}=\OR^{\N^*}$ and $P^{\N}$ encodes
 $\es_+^{\N^*}$ directly,
 or for some $n$ such that $0<n<\om$,
 $\rho_{n+1}^{\N^*}=\om<\OR^{\N^-}<\rho_n^{\N^*}$ and $\OR^\N=\rho_n^{\N^*}$
 and $P^\N$ encodes $\Th_{n}^{\N^*}((\N^*|\rho_n^{\N^*})\cup\{\pvec_n^{\N^*}\})$.
Moreover,} $\N^*$ has no proper
segment satisfying $\varphi$, and either $\N^*\sats\varphi$ or {\color{black}$\rho_\om^{\N^-}>\om$} and {\color{black}$\N$} projects
$<\rho_\om^{\N^-}$.)  Let $n<\om$ be such that $\rho_{n+1}^{\N^*}=\om<\OR^{\N^-}<\rho_n^{\N^*}$. {\color{black}By downward extension of embeddings,
$\M$ encodes
 an} $n$-sound premouse $\M^*$ over $\M^-$,
 and {\color{black}$\pi$ determines an} $n$-embedding $\pi^*:\M^*\to\N^*$ with $\pi\sub\pi^*$. Because $\rho_1^\M\leq\OR^{\M^-}$,
$\rho_{n+1}^{\M^*}\leq\OR^{\M^-}$.

Now suppose  $\M^*$ is $(n+1)$-sound.
Then $\M^*$ is fully sound, as $\M^*$ is a premouse over $\M^-$, so
$\M^*\pins\Lp(\M^-)$. Moreover, $\M^*\ins\M'$, where $\M'$ is the premouse over $\M^-$
coded by ${\Fop}(\M^-)$.
{\color{black}For by the $\Sigma_1$-elementarity of $\pi^*$, $\M^*$ has no proper segment modelling $\varphi$,
and if $\M$ has length
$>0$
and $\om<\rho_\om^{\M^-}$
then letting $A=\hmbdot^{\M^-}$, for  $\alpha<\rho_\om^{\M^-}$,
we have
$\pow(\alpha^{<\om}\cross A^{<\om})\cap\M^*\sub\M^-$.}

Now suppose instead that $\M^*$ is not $(n+1)$-sound. Let $\Hh^*=\core_{n+1}(\M^*)$. Then $\Hh^*\ins\M'$, where $\M'$ is as before, and the
$n^\nth$ master code $\Hh$ of $\Hh^*$ is a universal hull of $\M$, and either
$\Hh\in{\Fop}(\M^-)$ or $\Hh={\Fop}(\M^-)$, as required.
\end{proof}

Note we
makes significant use here
of the assumption
that $\rho_1^\M\leq\OR(\M^-)$.

\subsection{Copying and realization}\label{sec:copying}

We next want to consider the copying construction and how it relates to operators $\Fop$ with fine condensation. {\color{black}As discussed in \cite{steel2010outline}, \cite{mim} and \cite{copy_con}, even for standard mice, the copying construction is complicated by type 3 premice $\M$, because one must handle segments $\N\pins\M$
such that $\N\npins\core_0(\M)$,
but the fine structural maps $\pi:\core_0(\M)\to\core_0(\R)$
only act directly on $\core_0(\M)$.
We first make some preparations in this regard.
The following notions are from \cite{mim} and \cite{fsfni}}:
\begin{dfn}\label{dfn:uparrow}
 Let $\M$ be an opm.
 If $\M$ is not type 3 then $\M^\uparrow$ denotes $\M$.
 If $\M$ is type 3 and $\kappa=\hmmu^\M$
then\index{$\M^\uparrow$}
 $\M^\uparrow$ dentes $\Ult_0(\M\wupto\kappa^{+\M},\actext^\M)$.

  For $\pi:\M\to\N$, a $\Sigma_0$-elementary embedding between opms of the same type, we define
\index{$\Shift(\pi)$}$\Shift(\pi):\M^\uparrow\to\N^\uparrow$ as follows.
 If $\M$ is not type 3 then $\Shift(\pi)=\pi$.
 If $\M$ is type 3 then $\Shift(\pi)$
 is the embedding induced via the Shift Lemma by $\pi$.

{\color{black}If $\M$ is not type 3, we say that $\pi$ is \emph{$\nu$-preserving},  not \emph{$\nu$-high} and not \emph{$\nu$-low}.
 Suppose $\M$ is type 3.
 Then we say that $\pi$ is
 \index{$\nu$-preserving}\emph{$\nu$-preserving}
 iff $\Shift(\pi)(\nu(F^\M))=\nu(F^\N)$,
 \index{$\nu$-high}\emph{$\nu$-high}
 iff $\Shift(\pi)(\nu(F^\M))>\nu(F^\N)$, and \index{$\nu$-low}\emph{$\nu$-low}
 iff $\Shift(\pi)(\nu(F^\M))<\nu(F^\N)$.}
\end{dfn}

{\color{black}
\begin{rem}\label{rem:elem_nu-pres}
Elementarity considerations show that
if $\pi$ is $\rSigma_1$-elementary, it is not $\nu$-low,
and if $\pi$ is $\rSigma_2$-elementary, then it is $\nu$-preserving; see  \cite{mim}. 
\end{rem}}
\begin{lem}\label{lem:type_3_F-pm_condensation}
Let $\Fop$ be an operator above $b$ which almost condenses {\color{black}coarsely} above $b$.
Let $A\in\witri{C_{\Fop}}$.
Let $\N$ be a type 3 ${\Fop}$-pm over $A$
such that $\N^\uparrow$ is an ${\Fop}$-pm.
Let $\pi:\R\to\core_0(\N)$ be a weak $0$-embedding which is above $b$.
Then $\R=\core_0(\M)$ for some ${\Fop}$-pm $\M$.
\end{lem}
\begin{proof}
Because $\pi$ is a weak $0$-embedding, $E=E^\R$ is an extender over $\R$.
So we can define $\R^\uparrow$ and $\Shift(\pi):\R^\uparrow\to\N^\uparrow$ as in
\ref{dfn:uparrow}.
By almost coarse condensation, $\R^\uparrow$ is an ${\Fop}$-pm, which yields the desired
conclusion.
\end{proof}

{\color{black}Given an iteration tree $\Uu$ on an opm $\M$,
and given $\N\ins\M$,
the next definition
sets up notation for embeddings
on $\N$ induced by the iteration maps of $\Uu$.}

{\color{black}
\begin{dfn}\label{dfn:restricted_it_maps}
 Let $\Uu$ be a $k$-maximal tree on an opm $\M$, and let $\N\ins\M$. If $\N\pins\M$ then let $\left<\N_i\right>_{i< k}$
 be the model dropdown sequence of $\N$ in $\M$. (That is,
 $\N_0=\N$, for each $i+1<k$,
 $\N_{i+1}$ is the least $\N'\pins\M$ such that $\rho_\om^{\N'}<\rho_\om^{\N_i}$, and
 $\rho_\om^{\N_{k-1}}$ is an $\M$-cardinal.) We say that $\N$ is \index{stable}\emph{$\M$-stable} iff:
 \begin{enumerate}[label=--]
 \item If $\M$ is active type 3 then either $\N=\M$ or $\N\pins\M^\sq$ (hence $\N_k\pins\M^\sq$).
 \item If $\N\pins\M$
 then for all $i+1<k$,
 if $\N_{i+1}$ is active type 3 then $\N_i\pins(\N_{i+1})^\sq$.
 \end{enumerate}

 Suppose $\N$ is $\M$-stable.
 Let $\alpha<\lh(\Uu)$.
 Let us say for the moment that $(\Uu,\alpha)$ is \emph{good} iff either
 \begin{enumerate}
 \item $\N=\M$ and $[0,\alpha]^\Uu\cap\mathscr{D}^\Uu=\emptyset$, or
 \item there are ordinals $\gamma_0\leq^\Uu\delta_0=\gamma_1\leq^\Uu\delta_1=\gamma_2\leq^\Uu\delta_2\ldots\delta_{k-1}=\gamma_k\leq^\Uu\delta_k=\alpha$
 such that:
 \begin{enumerate}
 \item $\gamma_0=0$ and $[\gamma_0,\delta_0]^\Uu\cap\mathscr{D}^\Uu=\emptyset$, and
 \item for each $i\in(0,k]$, if $\gamma_i<\delta_i$ then:
 \begin{enumerate}\item $(\gamma_i,\delta_i]^\Uu\cap\mathscr{D}^\Uu=\{\varepsilon_i\}$ where $\pred^\Uu(\varepsilon_i)=\gamma_i$ (and $\varepsilon_i\leq^\Uu\delta_i$),
 \item $\N_{k-i}\in\dom(j)$ where $j=i^{*\Uu}_{\varepsilon_{i-1}\delta_{i-1}}\com i^{*\Uu}_{\varepsilon_{i-2}\delta_{i-2}}\com\ldots\com i^{*\Uu}_{\varepsilon_1\delta_1}\com i^{\Uu}_{0\delta_0}$,
 and
 $M^{*\Uu}_{\varepsilon_i}=j(\N_{k-i})$.
 \end{enumerate}
 \end{enumerate}
 \end{enumerate}

 If $(\Uu,\alpha)$ is good
 then we define \index{$M^{\Uu}_{\N,\alpha}$}$M^{\Uu}_{\N,\alpha}\ins M^\Uu_\alpha$ and \index{$i^{\Uu}_{\N,0\alpha}$}\[i^{\Uu}_{\N,0\alpha}:\core_0(\N)\to\core_0(M^{\Uu}_{\N,\alpha}) \]
 as follows.
 Let $\gamma_k,\delta_k$ be as above. If $\gamma_k<\delta_k$ then set
 $M^{\Uu}_{\N,\alpha}=M^\Uu_\alpha$
 and $i^{\Uu}_{\N,0\alpha}=i^{*\Uu}_{\varepsilon_{k}\delta_k}\com\ldots\com i^{*\Uu}_{\varepsilon_1\delta_1}\com i^\Uu_{0\delta_0}$.
 If $\gamma_k=\delta_k$
 then set $M^{\Uu}_{\N,\alpha}=j(\N)$,
 where $j=i^{*\Uu}_{\varepsilon_{k-1}\delta_{k-1}}\com\ldots\com i^{*\Uu}_{\varepsilon_1\delta_1}\com i^{\Uu}_{0\delta_0}$,
 and set $i^{\Uu}_{\N,0\alpha}=j\rest\core_0(\N)$.

Now we say that $[0,\beta]^\Uu$
is \index{bounded}\emph{$\N$-bounded}
iff there is $\alpha\leq^\Uu\beta$ such that $[0,\alpha]^\Uu$ is good
and if $\alpha<^\Uu\beta$
then letting $\varepsilon\leq^\Uu\beta$ be such that $\pred^\Uu(\varepsilon)=\alpha$, we have $M^{*\Uu}_{\varepsilon}\pins M^{\Uu}_{\N,\alpha}$. If $[0,\beta]^\Uu$ is $\N$-bounded,
as witnessed by $\alpha$,
we say that $[0,\beta]^{\Uu}$ \index{drops below the image}\dfnemph{drops below the image of $\N$} iff $\beta$ is not good;
that is, $\alpha<\beta$.

We say that $\Uu$ is \emph{$\N$-bounded} iff $\beta$ is $\N$-bounded for all $\beta<\lh(\Uu)$.
\end{dfn}}

The following lemma is an instance of some very related material in \cite[\S2]{extmax} (for example, \cite[Lemma 2.27]{extmax}), \cite[\S6]{fsfni}, \cite[\S7]{mim} and the preprint \cite[\S5]{mouse_scales}. It shows that
when we define an embedding $\tau:\Ult_k(\R,F^\M)\to\Ult_k(\Ss,F^\N)$ via the Shift Lemma from a given embedding $\pi:\core_0(\R)\to\core_0(\Ss)$, if $\pi$
is $\nu$-preserving, so is $\tau$.

{\color{black}
\begin{lem}\label{lem:propagate_nu-pres}
Let $\R,\Ss$ be type 3 $k$-sound opms, where $k<\om$,
and $\pi:\core_0(\R)\to\core_0(\Ss)$ a $\nu$-preserving weak $k$-embedding.
Let $\M,\N$ be active $\Fop$-premice and $\psi:\core_0(\M)\to\core_0(\N)$  a weak $0$-embedding. Let $\kappa=\crit(F^\M)$. Suppose $\R||\kappa^{+\R}=\M|\kappa^{+\M}$
and $\Ss||\psi(\kappa)^{+\Ss}=\N|\psi(\kappa)^{+\N}$ and $\pi\rest(\R||\kappa^{+\R})=\psi\rest(\M|\kappa^{+\M})$. Suppose  $\kappa<\rho_k^\R$.
Suppose  $\pi$ is $\nu$-preserving.
Let
\[ \tau:\Ult_k(\R,F^\M)\to\Ult_k(\Ss,F^\N) \]
by the Shift Lemma map induced by $\pi,\psi$. Then $\tau$ is $\nu$-preserving.
\end{lem}
\begin{proof}
This kind of argument is given, for example, in \cite[\S6.1]{fsfni}. If the ultrapower maps $i^{\R,k}_{F^\M}$ and $i^{\Ss,k}_{F^\N}$ are both $\nu$-preserving,
 commutativity ($\tau\com i^{\R,k}_{F^\M}=i^{\Ss,k}_{F^\N}\com\pi$) gives the desired result. If $k>0$
 then these ultrapower maps are indeed $\nu$-preserving, by Remark \ref{rem:elem_nu-pres}).

So suppose $k=0$.

Let $\mu=\crit(F^\R)$.
We have $R^\uparrow=\Ult_0(\R|\mu^{+\R},F^\R)$.
Note that $F^\M$ is also an extender over $\R^\uparrow$.
Let $U'=\Ult_0(\R^\uparrow,F^\M)$
and let $j':\R^\uparrow\to U'$
be the ultrapower map.
Then a standard calculation shows that \[U'=U^\uparrow=\Ult_0(U|j(\mu)^{+U},F^U)\]
and $j=j'\rest\core_0(\R)$, and $j'\com i^{\R|\mu^{+\R},0}_{F^\R}=i^{U|j(\mu)^{+\R},0}_{F^U}\com j$.

Now let $a,f\in\core_0(\R)$
be such that $\nu(F^\R)=[a,f]^{\R,0}_{F^\R}$.
Then we may assume $f\in\R|\mu^{+\R}$ and we have $\nu(F^\R)=[a,f]^{\R|\mu^{+\R},0}_{F^\R}=i^{\R|\mu^{+\R},0}_{F^\R}(f)(a)$. So
\begin{equation}\label{eqn:shift_[a,f]}
\begin{split}j'(\nu(F^\R))&=j'(i^{\R|\mu^{+\R},0}_{F^\R}(f)(a))\\
&=j'(i^{\R|\mu^{+\R},0}_{F^\R}(f))(j'(a))\\
&=i^{U|j(\mu)^{+U},0}_{F^U}(j(f))(j(a))\\
&= [j(a),j(f)]^{U|j(\mu)^{+U},0}_{F^U}=[j(a),j(f)]^{U,0}_{F^U}.\end{split}\end{equation}

Now $\nu(F^U)=\sup j``\nu(F^\R)$.
So by line (\ref{eqn:shift_[a,f]}), if $j'$ is continuous at $\nu(F^\R)$ then $j$ is $\nu$-preserving, and if $j'$ is discontinuous at $\nu(F^\R)$ then $j$ is $\nu$-high.

But $j'$ is continuous at $\nu(F^\R)$ iff $\cof^{\R}(\nu(F^\R))\neq\kappa$.
And $\cof^{\R}(\nu(F^\R))=\kappa$
iff $\cof^{\Ss}(\nu(F^\Ss))=\pi(\kappa)=\psi(\kappa)$, since $\pi$ is $\nu$-preserving. So if $\cof^{\R}(\nu(F^\R))\neq\kappa$,
then $j$ is $\nu$-preserving,
and similarly, so is $i^{\Ss,0}_{F^\N}$, and so as remarked earlier, it follows that $\tau$ is also $\nu$-preserving, as desired.

So it just remains to consider the case that $\cof^\R(\nu(F^\R))=\kappa$
(so $j'$ is discontinuous at $\nu(F^\R)$),
and so $\cof^\Ss(F^\Ss)=\pi(\kappa)$ (and $i^{\Ss,0}_{F^\N}$ is discontinuous at $\nu(F^\Ss)$).
Let $f\in\R$ be such that $f:\kappa\to\nu(F^\R)$
is continuous, strictly increasing and $\sup f``\kappa=\nu(F^\R)$.
Then note that since $\nu(F^U)=\sup j``\nu(F^\R)$
and $\kappa=\crit(j)=\crit(j')$,
we have $j'(f)\rest \kappa=j\com f:\kappa\to\nu(F^U)$ and $j'(f)\rest\kappa$ is continuous, strictly increasing and $\sup j'(f)``\kappa=\nu(F^U)$.
Now let $g:[\mu]^{<\om}\to\R|\mu^{+\R}$
and $a\in[\nu(F^\R)]^{<\om}$
be such that $f=i^{\R|\mu^{+\R},0}_{F^\R}(g)(a)$ and $\kappa\in a$. Say $\kappa=\alpha_i$ where $a=\{\alpha_0,\ldots,\alpha_{k-1}\}$ and $\alpha_0<\ldots<\alpha_{k-1}$. Then define
$g':[\mu]^k\to\R|\mu^{+\R}$
to be $g'(u)=\sup g(u)``u_i$,
where $u=\{u_0,\ldots,u_{k-1}\}$
and $u_0<\ldots<u_{k-1}$. Then note that $\nu(F^\R)=[a,g']^{\R,0}_{F^\R}$.
Let $h'\in U$
be the function $h':[j(\mu)]^{k+1}\to U|j(\mu)^{+U}$
with \[ h'(u)=\sup (j(g')(u\cut\{u_i\}))``u_i,\]
where $u=\{u_0,\ldots,u_{k}\}$
and $u_0<\ldots<u_k$.
Note here that $j(g')$ has domain $[j(\mu)]^k$, whereas $h'$ has domain $[j(\mu)]^{k+1}$, and if $u=\{u_0,\ldots,u_{k}\}$ as above, then \[ h'(u)=\sup  (j(g')(\{u_0,\ldots,u_{i-1},u_{i+1},\ldots,u_{k}\}))``u_i.\]
Then $\nu(F^U)=[j(a)\cup\{\kappa\},h']^{U,0}_{F^U}$.
For note that $|a|=k$
and $\kappa$ is the $i$th element of $a$, but $j(a)\cap[\kappa,j(\kappa))=\emptyset$, and $j(\kappa)$ is the $i$th element of $j(a)$,
so $\kappa$ is the $i$th element of $j(a)\cup\{\kappa\}$.

Since $\pi$ is $\nu$-preserving,
 $\nu(F^\Ss)=[\pi(a),\pi(g')]^{\Ss,0}_{F^\Ss}$ and $[\pi(a),\pi(g)]^{\Ss,0}_{F^\Ss}$ is a function $f^*:\pi(\kappa)\to\nu(F^\Ss)$ which is continuous and strictly increasing, and $\nu(F^\Ss)=\sup f^*``\pi(\kappa)$.
So it now easily follows that $\nu(F^\Ss)=[\tau(j(a))\cup\{\tau(\kappa)\},\tau(h')]^{U^*,0}_{F^{U^*}}$, where $U^*=\Ult_0(\Ss,F^\N)$,
so $\tau$ is $\nu$-preserving,
as desired.
\end{proof}}
We now verify that fine condensation for ${\Fop}$ ensures that the copying construction proceeds
smoothly for relevant ${\Fop}$-premice.
The indexing function $\iota$ in the following lemma need not be the identity, because of the possibility of
$\nu$-high copy embeddings
 $\pi:\core_0(\M)\to\core_0(\N)$
between type 3 premice $\M,\N$.

\begin{lem}\label{lem:strategy_copying}
Let $\Fop$ be an operator above $b$ which almost condenses finely above $b$.
Let $\bar{A},A\in\witri{C_{\Fop}}$. Let $j\leq\om$ and let $\Q$ be a $j$-sound ${\Fop}$-premouse over $A$.
Let $(\N,k)\wins(\Q,j)$ {\color{black}be such that $\N$ is $\Q$-stable} \tu{(}see \ref{dfn:restricted_it_maps}\tu{)}.
Let $\M$ be a $k$-relevant ${\Fop}$-pm over $\Abar$. Let
$\pi:\core_0(\M)\to\core_0(\N)$ be an ${\Fop}$-rooted $k$-factor above $b$.
Let $\Sigma_\Q$ be an ${\Fop}$-$(j,\om_1+1)$-strategy for $\Q$.
Then there is an ${\Fop}$-$(k,\om_1+1)$-strategy $\Sigma_\M$ for $\M$ such that trees $\Tt$ via
$\Sigma_\M$ lift to trees $\Uu$ via $\Sigma_\Q$. In fact, for such pairs $(\Tt,\Uu)$, there is
$\iota:\lh(\Tt)\to\lh(\Uu)$ such that for each $\alpha<\lh(\Tt)$, there is $(N^\Uu_\alpha,\pi_\alpha)$ such that:
\begin{enumerate}
\item  $(N^\Uu_{\alpha},\deg^\Tt_\alpha)\wins (M^\Uu_{\iota(\alpha)},\deg^\Uu_{\iota(\alpha)})$.
\item  $\pi_\alpha:\core_0(M^\Tt_\alpha)\to\core_0(N^\Uu_{\alpha})$  is an
 ${\Fop}$-rooted $\deg^\Tt_\alpha$-factor which is above $b$.
\item  If $\pi$ is good then
$\pi_\alpha$ is good.

\item $\Uu$ is $\N$-bounded.

\item  $[0,\alpha]^\Tt\inter \mathscr{D}^\Tt\neq\emptyset$ iff $[0,\iota(\alpha)]^{\Uu}$ drops
below the image
of $\N$.

\item If $[0,\alpha]^\Tt\inter \mathscr{D}^\Tt=\emptyset$ then
$N^\Uu_{\alpha}=M^{\Uu}_{\N,\iota(\alpha)}$ and
\begin{equation}\label{eqn:no_drop_comm} \pi_\alpha\com
i^\Tt_{0\alpha}=i^{\Uu,I}_{\N,0,\iota(\alpha)}\com\pi. \end{equation}
\item\label{item:when_model_drop_nu-pres_near_emb}
{\color{black}If $[0,\alpha]^\Tt\cap\mathscr{D}^\Tt\neq\emptyset$
then $N^\Uu_{\alpha}=M^{\Uu}_{\iota(\alpha)}$,
$\deg^\Tt_\alpha=\deg^\Uu_{\iota(\alpha)}$ and $\pi_\alpha$ is a $\nu$-preserving near $\deg^\Tt_\alpha$-embedding.
\item\label{item:if_(N,k)=(Q,j)_and_pi_near}
If
$(\N,k)=(\Q,j)$ and $\pi$ is a near
$k$-embedding
then
 $N^\Uu_{\alpha}=M^{\Uu}_{\iota(\alpha)}$,
$\deg^\Tt_\alpha=\deg^\Uu_{\iota(\alpha)}$ and $\pi_\alpha$ is a near $\deg^\Tt_\alpha$-embedding,
and if $\pi$ is also $\nu$-preserving then so is $\pi_\alpha$.}
\end{enumerate}

The previous paragraph also holds with ``$(j,\om_1,\om_1+1)^*$-optimal'' replacing ``$(j,\om_1+1)$''
and ``$(k,\om_1,\om_1+1)^*$-optimal'' replacing ``$(k,\om_1+1)$''.
\end{lem}

\begin{proof}
We just sketch the proof, for the case that $\Tt$ is $k$-maximal. It is mostly the standard copying
construction, augmented with propagation of near embeddings (using the methods of the proof of \cite[Lemma 1.3]{fine_tame}) and the
standard extra details dealing
with type 3 premice (see \cite{steel2010outline}, \cite{mim}, \cite{copy_con}).
{\color{black}Because of how we handle type 3 premice, the tree orders of $\Tt$ and $\Uu$ need not be identical, and the indexing map $\iota$ can fail to be the identity.

The construction is by recursion on $\lh(\Tt)$.
Suppose we have determined $\Tt\rest(\alpha+1)$, $\iota(\alpha)$, $\Uu\rest(\iota(\alpha)+1)$ and
all the other objects mentioned in the lemma, satisfying the properties there, in particular with $N^\Uu_\alpha\ins M^\Uu_{\iota(\alpha)}$
and $\pi_\alpha:\core_0(M^\Tt_\alpha)\to\core_0(N^\Uu_\alpha)$.
 We now want to proceed to $\Tt\rest(\alpha+2)$, etc.

 Suppose first that  $\pi_\alpha$ is non-$\nu$-high or $E^\Tt_\alpha=F(M^\Tt_\alpha)$ or  $\lh(E^\Tt_\alpha)<\rho_0(M^\Tt_\alpha)$. Then we set $\iota(\alpha+1)=\iota(\alpha)+1$ and set $E^\Uu_{\iota(\alpha)}$ to be:
 \begin{enumerate}[label=--]
 \item $F(M^{\Uu}_{\N,\iota(\alpha)})$, if $E^\Tt_\alpha=F(M^\Tt_\alpha)$
 and $\pi_\alpha$ is non-$\nu$-low,
 \item $F(M^{\Uu}_{\N,\iota(\alpha)})\rest\Shift(\pi_\alpha)(\nu(E^\Tt_\alpha))$,
 if $E^\Tt_\alpha=F(M^\Tt_\alpha)$ and $\pi_\alpha$ is $\nu$-low,
 \item $\pi_\alpha(E^\Tt_\alpha)$, if $\lh(E^\Tt_\alpha)<\rho_0(M^\Tt_\alpha)$,
 \item $\Shift(\pi_\alpha)(E^\Tt_\alpha)$, if $\rho_0(M^\Tt_\alpha)<\lh(E^\Tt_\alpha)<\OR(M^\Tt_\alpha)$.
 \end{enumerate}

 Now suppose instead that $\pi_\alpha$ is $\nu$-high and
 $\rho_0(M^\Tt_\alpha)<\lh(E^\Tt_\alpha)<\OR(M^\Tt_\alpha)$.
 (So by induction
 with part \ref{item:when_model_drop_nu-pres_near_emb},  $[0,\alpha]^\Tt\cap\mathscr{D}^\Tt=\emptyset$
 and $\M$ is active type 3.)
 In this case let us say that $\alpha$ is an \emph{insertion stage}.
We set $\iota(\alpha+1)=\iota(\alpha)+2$,
and set $E^\Uu_{\iota(\alpha)}=F(M^{\Uu}_{\N,\iota(\alpha)})$
and $E^\Uu_{\iota(\alpha)+1}=\Shift(\pi_\alpha)(E^\Tt_\alpha)$.

Let $\beta=\pred^\Tt(\alpha+1)$
 and
 $\beta'=\pred^{\Uu}(\iota(\alpha+1))$ (determined $k$- and  $j$-maximality respectively).
Then $\beta'=\iota(\beta)$,
unless $\beta$ was an insertion stage and $\nu(F(M^\Tt_\beta))\leq\crit(E^\Tt_\alpha)$,
in which case $\beta'=\iota(\beta)+1$ and $\alpha+1\in\mathscr{D}^\Tt$
and $\iota(\alpha+1)\in\mathscr{D}^\Uu$.
We (can and do) define $\pi_{\alpha+1}$
via the Shift Lemma from $\pi_{\beta'},\pi_\alpha$  as usual.

For limit ordinals $\alpha$,
$\iota(\alpha)=\sup_{\beta<
\alpha}\iota(\beta)$,
and $[0,\alpha)^\Tt$ is the unique cofinal branch of $\Tt\rest\alpha$
such that for some $\beta<^\Tt\alpha$, we have
$\iota``[\beta,
\alpha)^\Tt\sub[0,\iota(\alpha))^{\Uu}$.
 We omit the remaining details of the definitions,
 which are routine.

Now let us observe that for each $\alpha$, $\pi_\alpha$ is
an ${\Fop}$-rooted $\deg^\Tt_\alpha$-factor above $b$  (see Definition \ref{dfn:k-factor}); that is, that there is a $\deg^\Tt_\alpha$-sound $\Fop$-pm
$\Ll$ and a $\deg^\Tt_\alpha$-tight $\sigma:\Ll\to M^\Tt_\alpha$ such that $\pi_\alpha\com\sigma\com\sigma^{\Ll}_{\deg^\Tt_\alpha+1}:\mathfrak{S}_{k+1}(\Ll)\to \core_0(N^\Uu_{\iota(\alpha)})$ is a near $k$-embedding. For
 given this,
fine condensation, together with Lemmas \ref{lem:k-stack-max_pres_rpm} (to see $M^\Tt_\alpha$ is {\color{black}$\deg^\Tt_\alpha$-}relevant) and \ref{lem:type_3_F-pm_condensation}, give that $M^\Tt_\alpha$ is
an
${\Fop}$-pm.
(If $M^\Tt_\alpha$ might be type 3 (that is, if $N^\Uu_{\alpha}$ is type 3), then
\ref{lem:type_3_F-pm_condensation} applies, because $(N^\Uu_{\alpha})^\uparrow$ is an
${\Fop}$-pm, since we can extend $\Uu\rest(\iota(\alpha)+1)$ to a tree $\Uu'$, setting
$E^{\Uu'}_{\iota(\alpha)}=F(N^\Uu_{\alpha})$.)

Fix $(\Ll_0,\sigma_0)$
witnessing  that $\pi$ is an  ${\Fop}$-rooted $k$-factor; so
$\sigma_0:\Ll_0\to\M$
is $k$-tight and
$\pi\com\sigma_0\com\sigma^{\L_0}_{k+1}:\mathfrak{S}_{k+1}(\L_0)\to\core_0(\N)$
is a near $k$-embedding.}

\begin{case}\label{case:copying_case_no_drop_in_model}
 $[0,\alpha]^\Tt$ does not drop in model in $\Tt$.

In this case, it is routine to verify that $[0,\iota(\alpha)]^{\Uu}$ does not drop below the
image of $\N$, $\pi_\alpha$ is a weak $\deg^\Tt_\alpha$-embedding and line (\ref{eqn:no_drop_comm})
holds.

If $\deg^\Tt_\alpha=k$ then  $(\Ll_0,\sigma)$ witnesses the fact that
$\pi_\alpha$ is an ${\Fop}$-rooted $k$-factor above $b$, where
$\sigma=i^\Tt_{0\alpha}\com\sigma_0$,
because $i^{\Uu}_{\N,0,\iota(\alpha)}$ and $\pi\com\sigma_0{\color{black}\com \sigma^{\L_0}_{k+1}}$ are  near $k$-embeddings,
and $\pi_\alpha\com i^\Tt_{0\alpha}=i^{\Uu}_{\N,0,\iota(\alpha)}\com\pi$.

Now suppose  that $[0,\alpha]^\Tt$ drops in degree and let $n=\deg^\Tt_\alpha$. Then letting
$\Ll=\core_{n+1}(M^\Tt_\alpha)$ and $\sigma:\core_0(\Ll)\to \core_0(M^\Tt_\alpha)$ be the core embedding,
$(\Ll,\sigma)$ witnesses the fact that $\pi_\alpha$ is an ${\Fop}$-rooted $n$-factor above $b$
(we have $\mathfrak{S}_{n+1}(\Ll)=\Ll$ and $\sigma_{n+1}^\Ll=\id$).
The fact that $\pi_\alpha\com\sigma$ is a near $n$-embedding is because
$\pi_\alpha\com\sigma=i^{\Uu}_{\N,\iota(\xi),\iota(\alpha)}\com\pi_\xi$, $\pi_\xi$ is a weak
$(n+1)$-embedding, and $i^{\Uu}_{\N,\iota(\xi),\iota(\alpha)}$ a near $n$-embedding.

{\color{black}Now consider part \ref{item:if_(N,k)=(Q,j)_and_pi_near}. Suppose
that $(\N,k)=(\Q,j)$ and $\pi$ is a near $k$-embedding. Clearly $N^\Uu_\alpha=M^\Uu_{\iota(\alpha)}$. The fact that $\pi_\alpha$ is a near $\deg^\Tt_\alpha$-embedding (in fact for every $\beta\leq\alpha$, $\pi_\beta$ is a near $\deg^\Tt_\beta$-embedding)
is by \cite[Lemma 1.3]{fine_tame}, or more literally, its proof. However, note that that proof is inductive on $\alpha$, and one should also maintain as part of the induction that $\deg^\Tt_\beta=\deg^\Uu_{\iota(\beta)}$ for every $\beta\leq\alpha$. In order to see that $\deg^\Tt_{\beta+1}=\deg^\Uu_{\iota(\beta+1)}$,
letting $\gamma=\pred^\Tt(\beta+1)$, if
$\beta+1\notin\mathscr{D}^\Tt$,
one uses the fact that $\pi_\gamma$ is a near $\deg^\Tt_\gamma$-embedding
and $\deg^\Tt_\gamma=\deg^\Uu_{\iota(\gamma)}$.
Finally, if $\pi$ is $\nu$-preserving, then using Lemma \ref{lem:propagate_nu-pres}, one easily  shows inductively that $\pi_\beta$ is $\nu$-preserving for every $\beta\leq\alpha$.}
\end{case}
\begin{case}\label{case:drop_in_model_in_[0,alpha]^T}
 $[0,\alpha]^\Tt$ drops in model in $\Tt$.

 It is straightforward to see that
$[0,\iota(\alpha)]^{\Uu}$ drops below the image of $\N$ and that $N^\Uu_{\alpha}=M^\Uu_{\iota(\alpha)}$. The
fact that $\pi_\alpha$ is an ${\Fop}$-rooted $\deg^\Tt_\alpha$-factor is
almost the same as in the dropping degree case above.
{\color{black}The fact that $\pi_\alpha$ is in fact a near $\deg^\Tt_\alpha$-embedding  and
$\deg^\Tt_\alpha=\deg^\Uu_{\iota(\alpha)}$ follows much as before, though this time it is not quite as directly by \cite[Lemma 1.3]{fine_tame} itself,
but from
an examination of its proof; one observes that the inductive argument used in \cite{fine_tame} can simply be done above any node in $\Tt$ at which there is a drop in model, instead of having to start at the root node $0$.\footnote{Note that we are not assuming that $\pi$ itself is a near embedding in this case;
it is just that above any drop in model in $\Tt$, we get near embeddings.}
 (Similar arguments were also used in \cite{copy_con}.) And the fact that $\pi_\alpha$ is $\nu$-preserving is proved similarly, by an induction above any node in $\Tt$ at which there is a drop in model.}\qedhere\end{case}
\end{proof}

\begin{dfn}
 Let $\N$ be an ${\Fop}$-pm and $k\leq\om$. Then $\N$ is
\index{$\Fop$-$n$-fine}\dfnemph{${\Fop}$-$k$-fine} iff for each
$j\leq k$:
\begin{enumerate}[label=--]
 \item $\core_j(\N)$ is a $j$-solid ${\Fop}$-pm,
 \item if $j<k$ then $\core_j(\N)$ is $(j+1)$-universal,
 \item if $k=\om$ then $\core_\om(\N)$ is ${<\om}$-condensing.\qedhere
\end{enumerate}
\end{dfn}

We next consider background constructions building ${\Fop}$-mice.

\begin{dfn}\label{dfn:Fop_con}
Let ${\Fop}$ be an operator over $\opbk$. Let $A\in\witri{C_{\Fop}}$ and $\chi\leq\OR(\opbk)+1$. An
\dfnemph{$L^{\Fop}[\es,A]$-construction (of length $\chi$)} is a
sequence $\CC=\left<\N_\alpha\right>_{\alpha<\chi}$ such that for all $\alpha<\chi$:
\begin{enumerate}[label=--]
\item $\N_0={\Fop}(A)$ and $\N_\alpha$ is an ${\Fop}$-pm over $A$.
\item If $\alpha$ is a limit then $\N_\alpha=\liminf_{\beta<\alpha}\N_\beta$.
\item If $\alpha+1<\chi$ then either (i) $\N_{\alpha+1}$ is $\hmE$-active and
$\N_{\alpha+1}\wdoubleupto\OR(\N_{\alpha+1})=\N_\alpha$, or (ii)
 $\N_\alpha$ is ${\Fop}$-$\om$-fine and $\N_{\alpha+1}={\Fop}(\core_\om(\N_\alpha))$.\qedhere
\end{enumerate}
\end{dfn}

We will now explain how fine condensation for ${\Fop}$ leads to the ${\Fop}$-iterability of
substructures $\R$ of ${\Fop}$-pms built by background construction. The basic engine
behind this is the realizability of iterates of $\R$ back into models of the
construction.

\begin{dfn}\label{dfn:pi_C_realizable}
Let $\Fop$ be an operator above $b$ which almost condenses finely. Let  $\bar{A},A\in\witri{C_{\Fop}}$. Let
$\CC=\left<\N_\alpha\right>_{\alpha\leq\lambda}$ be an
$L^{\Fop}[\es,A]$-construction. Let $k\leq\om$ and suppose that $\N_\lambda$ is
${\Fop}$-$k$-fine. Let $\R$ be a $k$-sound ${\Fop}$-pm over $\Abar$
and
$\pi:\core_0(\R)\to\core_k(\N_\lambda)$ be an above $b$ weak $k$-embedding. Let
$\Tt$ be a putative
${\Fop}$-tree
on $\R$, with $\deg^\Tt_0=k$. We say that $\Tt$
is
\index{realizable}\index{above}\dfnemph{$(\pi,\CC)$-realizable above $b$} iff for every
$\alpha<\lh(\Tt)$, $\beta=\base^\Tt(\alpha)$ {\color{black}exists (that is, $[0,\alpha]^\Tt$ drops only finitely often)} and letting $m=\deg^\Tt_\alpha$,
there are $\zeta,\tau$ such that:
\begin{enumerate}[label=--]
 \item $(\zeta,m)\leq_\lex(\lambda,k)$,
 \item if $[0,\alpha]^\Tt$ does not drop in model or degree then
$\zeta=\lambda$ and $\tau=\pi$,
 \item if $[0,\alpha]^\Tt$ drops in model or degree then
$\tau\colon \core_0(M^{*\Tt}_\beta)\to\core_m(\N_\zeta)$ is a near $m$-embedding above $b$,
 \item if $M^{*\Tt}_\beta$ is not type 3 then there is a weak $m$-embedding
$\varphi:\core_0(M^\Tt_\alpha)\to\core_m(\N_\zeta)$ such that $\varphi\com
i^{*\Tt}_{\beta\alpha}=\tau$.
 \item if $M^{*\Tt}_\beta$ is type 3 then there is a weak $m$-embedding
$\varphi:\Ss\to\core_m(N_\zeta)$ such that $\varphi\com i^{*\Tt}_{\beta\alpha}=\tau$, where $\Ss$
is ``$(M^\Tt_\alpha)^\sq$''.\footnote{\label{ftn:squash}$(M^\Tt_\alpha)^\sq$ might not make literal
sense, if say $M^\Tt_\alpha$ is not wellfounded. By ``$(M^\Tt_\alpha)^\sq$'' we mean that either
$\alpha=\xi+1$ and $\Ss=\Ult_m((M^{*\Tt}_\alpha)^\sq,E^\Tt_\xi)$ (formed without unsquashing), or
$\alpha$ is a limit and $\Ss$ is
the direct limit of the structures $(M^\Tt_\xi)^\sq$ for $\xi\in[\beta,\alpha)_\Tt$, under
the iteration maps.}\qedhere
\end{enumerate}
\end{dfn}

\begin{dfn}
A \dfnemph{putative ${\Fop}$-$(k,\theta)$-iteration strategy} for a $k$-sound ${\Fop}$-pm $\N$ is
a function $\Sigma$ such that for every $k$-maximal ${\Fop}$-tree $\Tt$ on $\N$,
with $\Tt$ via $\Sigma$ and $\lh(\Tt)<\theta$ a limit, $\Sigma(\Tt)$ is a $\Tt$-cofinal branch.
\end{dfn}

\begin{lem}\label{lem:L[E]_construction_iterability}
Let $\Fop$ be an operator above $b$ which almost condenses finely above $b$. Let  $\bar{A},A\in\witri{C_{\Fop}}$.
Let
$\CC=\left<\N_\alpha\right>_{\alpha<\chi}$ be an
$L^{\Fop}[\es,A]$-construction.
Suppose that $(\N_\alpha)^\uparrow$ is an
${\Fop}$-pm for each $\alpha<\chi$. Let $\lambda<\chi$ and $k\leq\om$ be such that $\N_\lambda$ is
${\Fop}$-$k$-fine, and let $\Ss=\core_k(\N_\lambda)$.
Let $\R$ be a $k$-relevant ${\Fop}$-pm over $\Abar$.
Let $\pi:\core_0(\R)\to\core_0(\Ss)$ be an ${\Fop}$-rooted $k$-factor above $b$.
Let $\Sigma$ be either:
\begin{enumerate}[label=--]
 \item a putative ${\Fop}$-$(k,\om_1+1)$-iteration strategy for $\R$, or
 \item a putative ${\Fop}$-$(k,\om_1,\om_1+1)^*$-optimal iteration strategy for $\R$.
 \end{enumerate}
Suppose that every putative ${\Fop}$-tree via $\Sigma$
is $(\pi,\CC)$-realizable above $b$. Then $\Sigma$
is an ${\Fop}$-$(k,\om_1+1)$, or ${\Fop}$-$(k,\om_1,\om_1+1)^*$-optimal, iteration strategy.
\end{lem}

\begin{proof}
The argument is almost that used for Lemma \ref{lem:strategy_copying},
using the  $(\pi,\CC)$-realizability maps in place of copy maps.
The hypothesis that each $(\N_\alpha)^\uparrow$ is an $\Fop$-pm is used to see that Lemma \ref{lem:type_3_F-pm_condensation} applies where needed. We leave the details to the reader.
\end{proof}

 The above proof does not work with
\emph{$(k,\om_1,\om_1+1)^*$-optimal} replaced by $(k,\om_1,\om_1+1)^*$,
{\color{black}because of the reliance on $m$-relevance in connection with fine condensation.}

\begin{rem}\label{rem:condenses_finely_beyond_almost}
We digress to mention a key application of
the extra strength that \emph{condenses finely} has compared to \emph{almost condenses finely};
this essentially comes from \cite{CMIP}, such as in \cite[\S2]{CMIP}.
Adopt the assumptions and notation of the first paragraph of
\ref{lem:L[E]_construction_iterability}.
Assume further that $\Fop$  condenses finely  (not just almost), $\opbk=V$ and ${\Fop}$ is
total.
For an ${\Fop}$-premouse $\M$, say that $\M$ is \dfnemph{${\Fop}$-full} iff there is no $\alpha\in\Ord$
such that ${\Fop}^\alpha(\M)$ projects $<\OR(\M)$.\footnote{Here ${\Fop}^\alpha(\M)$
is the unique ${\Fop}$-pm $\N$ such that $\M\ins\N$ and $\l(\N)=\l(\M)+\alpha$ and $\N|\beta$ is
$\hmE$-passive for every $\beta\in(\l(\M),\l(\N)]$.}
Assume also that there is no ${\Fop}$-full $\M$ such that $\OR(\M)$ is Woodin in
${\Fop}^{\Ord}(\M)$.
Let $\kappa$ be a cardinal.
 Suppose that {\color{black} $\R\sats$``there is no Woodin cardinal''} and  every $k$-maximal putative ${\Fop}$-tree $\Tt$ on $\R$ of length $\leq\kappa$
 is such that in some set-generic extension
$V[G]$,
either $\Tt$ is $(\pi,\CC)$-realizable, or there is a limit $\lambda\leq\lh(\Tt)$ and a
 $\Tt\rest\lambda$-cofinal branch $c$  such that {\color{black}$c$ is $\Tt$-maximal and} $(\Tt\rest\lambda)\conc c$ is
$(\pi,\CC)$-realizable.
Then $\R$ is ${\Fop}$-$(k,\kappa+1)$-iterable, via the strategy guided by Q-structures
of the form ${\Fop}^\alpha(M(\Tt))$ for some $\alpha\in\Ord$.\footnote{
It might be that the Q-structure satisfies ``$\delta(\Tt)$ is not Woodin'',
but in this case, $\alpha=\beta+1$ for some $\beta$ and ${\Fop}^\beta(M(\Tt))$
satisfies ``$\delta(\Tt)$ is Woodin''.} This follows by a straightforward adaptation of the
proof for standard premice (cf.~\cite{CMIP}), where $\Fop = \J$. In the argument one needs to apply \emph{condenses
finely} to embeddings $\varphi,\sigma$,
when $\varphi\com\sigma\notin
V$.
{\color{black}Here $\sigma$ is an iteration map arising from $\Tt$ (and possibly $c$), with codomain $M^\Tt_\infty$ (or $M^\Tt_c$), and $\varphi$ embeds the codomain of $\varphi$ into some model of $\CC$}. We can only expect $\varphi\com\sigma\in V$ if the realized branch does not drop in model
or degree (indeed, in the latter case, $\varphi\com\sigma=\pi$), or if all relevant objects are
countable. We use fine condensation to see that the Q-structure
$Q\ins \M^\Tt_b$ (where $\M(\Tt)\in V$ and $b\in V[G]$)
is in fact of the form ${\Fop}^\alpha(\M(\Tt))$; the hypothesis
of \emph{condenses finely} that $\M^-\in V$ holds
where needed since $\M(\Tt)\in V$, and so ${\Fop}^\alpha(\M(\Tt))\in V$
for each $\alpha$; $Q$ has no extenders in its sequence above $\delta(\Tt)$
by the smallness assumption.

From now on we will only deal with \emph{almost condenses finely}.
\end{rem}

{\color{black}
\subsection{Weak Dodd-Jensen}

The weak Dodd-Jensen property is defined as the obvious adaptation of the usual one (see \cite{wDJ}):
\begin{dfn}\label{dfn:weak_DJ}
Let $k\leq\om$ and $\M$ be a countable $k$-relevant opm.

We say that $(\Tt,\Q,\pi)$ is \dfnemph{$(\M,k)$-large} iff
$\Tt$ is a run of $\G^{\Fop}_{\opt}(\M,k,\om_1,\om_1)$
of countable successor length,
in which neither player has lost,
$\Q\wins M^\Tt_\infty$ and $\pi:\core_0(\M)\to\core_0(\Q)$ is a nearly $k$-good embedding.\footnote{
So $\Q$ is $k$-sound; the rules of $\G^{\Fop}_{\opt}$ therefore  imply that if
$\Q=M^\Tt_\infty$ then $\deg^\Tt_\infty\geq k$.
So we do not need to explicitly stipulate that $\deg^\Tt_\infty\geq k$, unlike in \cite{copy_con}.}

Let $\Sigma$ be an iteration strategy
for $\M$. Let $\alphavec=\left<\alpha_n\right>_{n<\om}$ enumerate $\OR^\M$. We say that $\Sigma$
has the
\index{Dodd-Jensen}\index{DJ}\dfnemph{$k$-weak Dodd-Jensen \tu{(}DJ\tu{)} property for $\alphavec$} iff for all
$(\M,k)$-large $(\Tt,\Q,\pi)$ with $\Tt$ via
$\Sigma$, we have $\Q=M^\Tt_\infty$, $b^\Tt$ does not drop in model (hence, nor degree),
 and
\[ i^\Tt\rest\OR^\M\leq^{\alphavec}_\lex\pi\rest\OR^\M \]
(that is, either $i^\Tt\rest\OR^\M=\pi\rest\OR^\M$, or $i^\Tt(\alpha_n)<\pi(\alpha_n)$ where
$n<\om$ is least such that $i^\Tt(\alpha_n)\neq\pi(\alpha_n)$).
\end{dfn}

Note that in the context above, if
$i^\Tt\rest\OR^\M=\pi\rest\OR^\M$ then $i^\Tt=\pi$, because $i^\Tt,\pi$ are both nearly $0$-good,
and $\M=\Hull_1^\M(\hmb^\M\un\OR^\M)$.

Following \cite{wDJ}, one can convert given strategies for stacks of trees  on a countable mouse into strategies with weak DJ:
\begin{lem}\label{lem:k-factor+weak_DJ}
Assume $\DC_\RR$. Let $\Fop$ be an operator above $b\in\HC$ which almost condenses finely above $b$. Let  $A\in\witri{C_{\Fop}}$.
 Let $\M\in\HC$ be an ${\Fop}$-$(k,\om_1,\om_1+1)^*$-optimally iterable $k$-relevant ${\Fop}$-pm.
 Let $\alphavec=\left<\alpha_n\right>_{n<\om}$ enumerate $\OR^\M$.
 Then there is an ${\Fop}$-$(k,\om_1,\om_1+1)^*$-optimal strategy for $\M$ with the
$k$-weak
DJ property for $\alphavec$.
\end{lem}
\begin{proof}[Proof Sketch]
The proof is like the usual one (see \cite{wDJ}),
using one minor observation: Suppose $\Tt$ is a run of $\G^{\Fop}_{\opt}(\M,k,\om_1,\om_1)$ of countable successor length,
$\Q\pins M^\Tt_\infty$ and $\pi:\core_0(\M)\to\core_0(\Q)$
is a near $k$-embedding, but $\Q$ is not $M^\Tt_\infty$-stable. In this case we can't use Lemma \ref{lem:strategy_copying} to copy trees on $\M$ to trees on $M^\Tt_\infty$ via $\pi$, which is superficially a problem for the proof of the lemma.
But because $\Q$ is $M^\Tt_\infty$-unstable, there is $E\in\es_+^{M^\Tt_\infty}$
such that $\Q\pins M^\Tt_\infty|\lh(E)$. Let $E$ be least such.
Then  note that $\Q\pins M^{\Tt\conc\left<E\right>}_\infty$
and $\Q$ is $M^{\Tt\conc\left<E\right>}_\infty$-stable.
(Here $\Tt\conc\left<E\right>$ is the run of the game for which $\Tt$ constitutes
the first $\alpha$ rounds (for some $\alpha$), followed by 1 more round, which is the tree on $M^\Tt_\infty$ which only uses $E$.) So we can apply Lemma \ref{lem:strategy_copying}
to $\pi$, $\Q$ and $M^{\Tt\conc\left<E\right>}_\infty$,
and this is enough for the proof.
\end{proof}
}

\subsection{Solidity and condensation}\label{sec:solidity}

In this final section we prove some of the basic fine structural facts (solidity, etc) hold for iterable $\Fop$-mice, assuming $\Fop$ condenses finely. The proof will be heavily based on the corresponding
proofs as presented in the union of \cite{FSIT}, \cite{steel2010outline} and \cite{deconIMT}.
Beyond extra details in connection with operator mice,
which are relatively minor,
 we need to handle some details which arise with superstrong extenders, which did not arise in the papers just mentioned. 

We also take the opportunity to discuss a couple of elements of the proof which are not made explicit in \cite{FSIT}, \cite{steel2010outline}, \cite{deconIMT}.
These elements are also relevant for standard premice,
not just operator premice.
They deal with two issues.

First,  \cite[Lemma 6.1.5]{FSIT} (on closeness of extenders) establishes that if $\Tt$ is a $k$-maximal tree on a $k$-sound premouse $\M$,
then for every $\alpha+1<\lh(\Tt)$, $E^\Tt_\alpha$ is close to $M^{*\Tt}_{\alpha+1}$. The proof of solidity, etc,
involves trees on phalanxes, to which this lemma does not literally apply, though the arguments analysing the comparisons seem to assume that the lemma does apply to them. So we discuss the generalization of the lemma to  trees on phalanxes which arise in the proof. This kind of thing has also been discussed in \cite{copy_con} and \cite{fsfni}, for example.

Second, and somewhat more importantly, we discuss why weak Dodd-Jensen is enough to rule out certain situations in the analysis of comparisons, in which it might not be entirely obvious that it is enough, and which are not addressed explicitly in \cite{FSIT}, \cite{steel2010outline}, \cite{deconIMT}.

The main issue regarding weak Dodd-Jensen arises in the following situation.
Suppose that $\M$ is a $k$-sound, $(k,\om_1,\om_1+1)^*$-iterable premouse, and we want to prove that $\M$ is $(k+1)$-solid.
Let $\alpha\in p_{k+1}^\M$.
Let  $\Ww$ be the $(k+1)$-solidity witness
\[ \Ww=\cHull_{k+1}^\M(\alpha\cup\{\pvec_k^\M,p_{k+1}^\M\cut(\alpha+1)\}) \]
at $\alpha$,
and $\pi:\Ww\to\M$ the uncollapse map.
We need to see that $\Ww\in\M$.
Suppose that $\alpha$ is not an $\M$-cardinal. Let $\kappa=\card^\M(\alpha)$. Let $\R\pins\M$ be least such that $\rho_\om^\R=\kappa$ and $\alpha\leq\OR^\R$.
The proof given in \cite{steel2010outline} that $\Ww\in\M$ proceeds by comparing the phalanx $\mathfrak{P}=((\M,{<\kappa}),(\R,\kappa),\Ww)$
versus $\M$, producing trees $\bar{\Tt}$ on $\mathfrak{P}$ and $\Uu$ on $\M$. An iteration strategy $\Sigma$
for $\M$ with the weak Dodd-Jensen property with respect to some enumeration of $\M$ is used to form $\Uu$, and $\bar{\Tt}$ is formed by simultaneously lifting $\bar{\Tt}$ to a tree $\Tt$ on $\M$ via $\Sigma$. Let $r$ be such that $\rho_{r+1}^\R=\kappa<\rho_r^\R$.
The tree $\bar{\Tt}$ is $(k,r,k)$-maximal,
meaning that in $\mathfrak{P}$, $\M$ is at degree $k$,
$\R$ is at degree $r$, and $\Hh$ at degree $k$.
The main issue does not arise in the \emph{anomalous} case,
i.e.~when $\R$ is active type 3 with $\OR^\R=\alpha$,
so we will ignore this case, and therefore we have $\kappa<\rho_0^\R$, so $r<\om$ is well-defined.

Now the process for forming $\Tt$ in this context at a step for which $\crit(E^{\bar{\Tt}}_\eta)=\kappa$ is not made clear in \cite{FSIT}.\footnote{\label{ftn:p.78}In fact, it is ill-defined, because on p.~47 of \cite{FSIT}, clause (4) of Definition 5.0.6 requires $\pow(\kappa)\cap \M^*_{\alpha+1}=\pow(\kappa)\cap\N$,
whereas on p.~78, in Subcase A of Case 2, it is required that $\bar{\P}^*_{\eta+1}=\P^*_{\eta+1}$. But in case $\kappa=\crit(E^{\bar{\Tt}}_\eta)=\crit(E^\Tt_\eta)$  (where $\kappa^{+\R}=\alpha$), then $E^{\bar{\Tt}}_\eta$ measures only $\pow(\kappa)\cap\R$, so is not $\M$-total, but $E^\Tt_\eta$ \emph{is} $\M$-total,
so according to Definition 5.0.6, in order for $\Tt$ to be an iteration tree, we cannot set $\P^*_{\eta+1}=\bar{\P}^*_{\eta+1}$.} The process is, however, clarified in \cite{deconIMT}:\footnote{See p.~728 of \cite{deconIMT}, within the proof of Theorem 3.3, in the paragraph beginning ``There is one wrinkle in the copying argument''.} we lift $M^{\bar{\Tt}}_{\eta+1}$ to $i^{*\Tt}_{\eta+1}(\R)$, defining $\pi_{\eta+1}:M^{\bar{\Tt}}_{\eta+1}\to i^{*\Tt}_{\eta+1}(\R)$; there is a natural definition for this map $\pi_{\eta+1}$, analogous to Shift Lemma maps. Then $\pi_{\eta+1}$ is a weak $r$-embedding, but not clear that it is a near $r$-embedding;
that is, not clear that it is $\rSigma_{r+1}$-elementary.
In fact, if there is an $\bfrSigma_r^\R$-definable function $f:\kappa\to\rho_r^\R$, and $\kappa$ is the least such ordinal,
then  $\pi_{\eta+1}$ is \emph{not} a near $r$-embedding,
as then $\pi_{\eta+1}(\rho_r^{M^{\bar{\Tt}}_{\eta+1}})<\rho_r^{i^{*\Tt}_{\eta+1}(\R)}$.

We write $M^{\bar{\Tt}}_\infty$ and $M^\Uu_\infty$ for the last models of $\bar{\Tt}$ and $\Uu$ respectively. Suppose $M^{\bar{\Tt}}_\infty=M^\Uu_\infty$ and neither $b^{\bar{\Tt}}$ nor $b^\Uu$ drops in model, $b^\Uu$ does not drop in  degree, and $b^{\bar{\Tt}}$ is above $\R$ (but maybe $b^{\bar{\Tt}}$ drops in degree strictly below $r$), so the first extender used along $b^{\bar{\Tt}}$
is some  $E^{\bar{\Tt}}_\eta$ as above. If $b^{\bar{\Tt}}$ does not drop in degree (strictly below $r$),
then one can use some standard fine structural arguments
to reach a contradiction. But if it drops in degree (and it follows that it drops to degree exactly  $k<r$), then standard fine structural arguments do not seem to suffice, and the most obvious tool appears to be weak Dodd-Jensen. But for this, we need to know that $\pi_\infty\com i^\Uu$ is a near $k$-embedding.

While \cite{FSIT} appears to ignore this situation,
\cite{steel2010outline} and \cite{deconIMT}
do not say much further regarding its analysis outside of the anomalous case, which we are presently ignoring.
In \cite{steel2010outline},
it is mentioned briefly that \cite[1.3]{fine_tame}
shows the copying construction gives rise to near $k$-embeddings (see \cite[Remark 4.3]{steel2010outline}, and also \cite[\S4.3]{FSIT}, immediately following the definition of \emph{$(\M,k)$-large}). But the results in \cite{fine_tame} themselves
do not really suffice themselves to cover all cases,
especially the one described above.
We will show in Claim \ref{clm:pi_infty_near_k} below, at least under some contradictory assumptions which we are free to make,\footnote{We will assume that $\Ww\notin\M$, that $\M$ is $(k+1)$-solid with respect to $p_{k+1}^\M\cut(\alpha+1)$,
and  that the enumeration of $\M$ we use for specifying weak Dodd-Jensen begins with $p_{k+1}^\M\cut(\alpha+1)$, in descending order.} that in the above situation,
$\pi_\infty$ is a near $k$-embedding, and hence so is $\pi_\infty\com i^\Uu$. The proof will involve an extension of the methods of \cite{fine_tame}.
Part of these calculations were used, for example, in
\cite{copy_con} and \cite{fsfni}. But in the most subtle case,
they were not, and
the argument we give to handle it seems to be new.

There are actually multiple options available to either complete the proof as it is, or to modify the proof somewhat and thereby prove solidity slightly differently:
\begin{enumerate}
\item\label{item:wDJ_for_k-lifting} Instead of formulating the weak Dodd-Jensen property
for near $k$-embeddings, formulate it for $k$-lifting embeddings (or maybe cardinal-preserving $k$-lifting embeddings), as defined in \cite{premouse_inheriting}. The copying construction routinely yields such embeddings.
\item\label{item:z,zeta-pres} Use the results of \cite{extmax} on $(z,\zeta)$-preservation
(and their generalizations in \cite{fsfni} to the superstrong level) and some of the arguments from \cite{fsfni};
these make the troublesome appeals to weak Dodd-Jensen unnecessary.
\item\label{item:ult_of_small_structure} We have not thought carefully about the following option, but it seems it should work: instead of formulating the weak Dodd-Jensen property with respect to stacks in which the rounds are each $n$-maximal for the relevant $n$, consider stacks in which the rounds can be ``weakly $n$-maximal'', in which degrees of nodes in the tree need not be taken as large as possible. Then we would be free to enforce agreement over degrees of nodes in $\bar{\Tt}$ (as above) with corresponding nodes in $\Tt$ at certain points of the construction, thus avoiding the difficulties we encounter in the proof of Claim \ref{clm:pi_infty_near_k} later.

A variant of this, which should also work,
would be  to allow the equation ``$\P^*_{\eta+1}=\bar{\P}^*_{\eta+1}$'' on page 78 of \cite{FSIT} (see Footnote \ref{ftn:p.78}). So the lift tree $\Tt$ would apply $E^\Tt_\eta$ to a model with strictly fewer subsets of its critical point than what is measured by $E^\Tt_\eta$. But then  the resulting map $\pi_{\eta+1}$ should in fact be a $\deg^{\bar{\Tt}}_{\eta+1}$-embedding, and it seems this should also remove the problem.
\end{enumerate}

Option \ref{item:wDJ_for_k-lifting} is probably
mathematically the most natural option.
But it involves
a significant change to the general setup for the proof,
and seems it might require using a different iteration strategy than that selected in \cite{steel2010outline}
(since we need one that has the resulting Dodd-Jensen property).
Option \ref{item:ult_of_small_structure} also involves changes to the general development, and selection of iteration strategy. Option \ref{item:z,zeta-pres} does not involve such changes
to the setup, so can be used to show that the original comparison argument (as in \cite{steel2010outline} and \cite{deconIMT})
works. But it involves a significant change in method for analysing aspects of the comparison, and new ideas for this.

We wanted to keep the setup of the original proof the same,
excluding  the alternatives \ref{item:wDJ_for_k-lifting} and \ref{item:ult_of_small_structure},
and as far as possible, to stick to the same basic methods of the original proof. The proof we give achieves the latter better than  option \ref{item:z,zeta-pres}. It also gives some  information which does not come out of the other methods.
Let us begin.

We  now state the central result of the paper -- the fundamental fine structural facts for
${\Fop}$-mice. An \dfnemph{${\Fop}$-pseudo-premouse}
is just the $\Fop$- version of a pseudo-premouse (see \cite[\S10]{FSIT}),
an \dfnemph{${\Fop}$-bicephalus}
is that of a bicephalus (see \cite[\S9]{FSIT}),
and the \dfnemph{${\Fop}$-iterability} of such structures
is defined in the
obvious manner. Likewise the definition of \dfnemph{${\Fop}$-iterability} for phalanxes of ${\Fop}$-pms.

\begin{tm}\label{thm:k+1-fineness}
Let $\Fop$ be an operator above $b\in\HC$, over $\mathscr{B}$, which almost condenses finely.
Then:
\begin{enumerate}
\item\label{item:k+1-fine} For $k<\om$, every $k$-sound, ${\Fop}$-$(k,\om_1,\om_1+1)^*$-optimally
iterable ${\Fop}$-premouse  is
${\Fop}$-$(k+1)$-fine.
 \item\label{item:condensing} Every $\om$-sound, ${\Fop}$-$(\om,\om_1,\om_1+1)^*$-optimally iterable
${\Fop}$-premouse  is ${<\om}$-condensing.
\item\label{item:pseudo-pm} Every ${\Fop}$-$(0,\om_1,\om_1+1)^*$-optimally iterable
${\Fop}$-pseudo-premouse is an
${\Fop}$-premouse.
\item\label{item:bicephalus} There is no non-trivial ${\Fop}$-$(0,\om_1,\om_1+1)^*$-optimally iterable
${\Fop}$-bicephalus.
\end{enumerate}
\end{tm}
\begin{proof}[Proof sketch]
The proof is heavily based on that for standard premice, as given by the combination of
 \cite{steel2010outline}, \cite{deconIMT} and \cite{FSIT}, and with which the reader should be reasonably familiar.
But we will give quite a detailed proof of the solidity aspect of part \ref{item:k+1-fine}, in order that we can describe the new features which arise for operator-premice and in particular superstrong extenders,\footnote{Various other papers have dealt in various ways with fine structure for superstrong extenders, for example \cite{zeman}, \cite{premouse_inheriting}, \cite{fsfni},  \cite{fs_plus-one}, \cite{kappa-plus}. The extra considerations that we need to handle superstrong extenders are very similar to some of those which appear in those papers. But because we are working with Mitchell-Steel indexing and generally following the original proof setup from \cite{steel2010outline}, \cite{deconIMT}, \cite{FSIT}, we can't directly cite those works.}  and also in order to discuss the details in the situation described in the introduction to this section.\footnote{The reader who has not gone through the rest of this paper, and is just interested in the proof of solidity for standard premice and the issues mentioned in the introduction, should be able to read the proof, ignoring the details specifically regarding $\Fop$-premice. In that case, the structure $\Hh$ we  introduce is just $\Ww$, and the uncollapse map $\sigma:\Hh\to\Ww$ is just the identity.}
We will omit some more routine calculations.
  We also sketch enough of the rest of parts  \ref{item:k+1-fine} and \ref{item:condensing}, focusing
on aspects new for operator-premice,
that combined with \cite{steel2010outline}, \cite{deconIMT}, \cite{FSIT} and the proof of solidity we give, one obtains a complete proof
 of parts  \ref{item:k+1-fine} and \ref{item:condensing}.
Part
\ref{item:pseudo-pm} involves similar modifications to the standard proof,
and part \ref{item:bicephalus} is an immediate transcription.

Part \ref{item:k+1-fine}:
Let $\M$ be a $k$-sound, ${\Fop}$-$(k,\om_1,\om_1+1)^*$-optimally iterable ${\Fop}$-premouse
over  $A\in\witri{C_{\Fop}}$.
We may assume that $\rho_{k+1}^\M<\rho_k^\M$,
and by Lemma \ref{lem:non_k-rpm_sound}, that $\M$ is $k$-relevant.
We may assume that $\M$ is countable (otherwise we can replace $\M$ with a countable
elementary substructure, because ${\Fop}$ almost condenses coarsely above $b\in\HC$ and
$\opbk\sats\DC$).

Let $\Sigma_0$ be an ${\Fop}$-$(k,\om_1,\om_1+1)^*$-optimal iteration strategy for $\M$.
We would like to use Lemma \ref{lem:k-factor+weak_DJ}, but that lemma assumes $\DC_\RR$.
But we may assume $\DC_\RR$.
For we can work in $W=L^{{\Fop},\Sigma_0}[x]$,
where $x\in\RR$ codes $\M$. (The hypotheses of the
theorem hold in $W$ regarding $b,A,\M,{\Fop}^W,\Sigma_0^W,\opbk^W$, where $\opbk^W,{\Fop}^W,\Sigma_0^W$ are the  restrictions of
$\opbk,{\Fop},\Sigma_0$ to $W$.)

Now using
\ref{lem:k-factor+weak_DJ}, let $\Sigma$ be an ${\Fop}$-$(k,\om_1+1)$ iteration strategy
for $\M$ with the weak DJ property for some enumeration of $\OR^\M$.

We will first establish $(k+1)$-universality and that $\Cc=\core_{k+1}(\M)$ is an ${\Fop}$-pm. For this, let us assume for better focus that $\M$ is a successor, since the limit case is easier and much closer to the standard proof.
Let $\pi:\Cc\to\M$ be the core map.\label{pg:repeat_k-relevant}

First suppose $k=0$, and consider $1$-universality. 
Because $\pi$ is $0$-good and by \ref{lem:Qformula_pres_weak0}, $\Cc$ is a
Q-opm, $\Cc$ is a successor and $\pi(\Cc^-)=\M^-$.
By fine condensation and \ref{lem:strengthen_scf_conclusion}, $\Hh={\Fop}(\Cc^-)$ is a universal
hull of $\Cc$, as witnessed by $\sigma:\Hh\to\Cc$.

We claim that $\Cc$ is $0$-relevant. For suppose otherwise,
so $\rho_1^\Cc=\rho_\om^{\Cc^-}$.
But $\rho_1^{\Cc}=\rho_1^\M$ and $\M$ is $0$-relevant, so $\rho_1^{\M}<\rho_\om^{\M^-}$.
Since $\pi$ is $0$-good,
$\pi(\Cc^-)=\M^-$
and $\pi(\rho_1^{\Cc})=\pi(\rho_\om^{\C^-})=\rho_\om^{\M^-}$.
So by ${<\om}$-condensation
for $\M^-$
and since $\rho_\om^{\Cc^-}=\rho_1^\M$ is an $\M$-cardinal,
we have $\Cc^-\pins\M^-$, so $\Hh=\Fop(\Cc^-)\pins\M^-$ also.
But $\Hh$
is a universal hull of $\Cc$,
and therefore $\Cc\in\M$
also, contradicting the fact that $\Cc=\core_1(\M)$.

So letting $\rho=\rho_1^\M$,
we have $\rho=\rho_1^\Cc<\rho_\om^{\Cc^-}$.
Since $\Hh^-=\Cc^-$, therefore $\Cc|\rho^{+\Cc}=\Hh|\rho^{+\Hh}$.
So it suffices to see that $\M|\rho^{+\M}=\Hh|\rho^{+\Hh}$.

The phalanx
$\ph=((\M,{<\rho}),\Hh)$
is ${\Fop}$-$((0,0),\om_1+1)$-maximally iterable.\footnote{A $(k_0,k_1,\ldots,k)$-maximal tree on a
phalanx $((M_0,\rho_0),(M_1,\rho_1),\ldots,H)$, is one formed according to the usual rules for
$k$-maximal trees, except that an extender $E$ with $\rho_{i-1}\leq\crit(E)<\rho_i$ (where
$\rho_{-1}=0$) is applied to $M_i$, at degree $k_i$.} Moreover, we get an
${\Fop}$-$((0,0),\om_1+1)$-iteration strategy for $\ph$ by lifting to $0$-maximal
trees on $\M$ via $\Sigma$.
This is proved by using $\pi\com\sigma$ to lift $\Hh$ to $\M$
and the identity to lift $\M$ to $\M$,
 {\color{black}combining the usual methods for lifting trees on phalanxes (as in \cite{FSIT} and \cite{steel2010outline}), handling various details much as in the proof of Lemma
\ref{lem:strategy_copying}, 
in particular} to see  that the strategy is indeed an ${\Fop}$-strategy. We can therefore compare $\ph$ with $\M$.
The analysis of the comparison is mostly routine, using  weak DJ
to rule out various possibilities.
The only, small, difference is when $b^\Tt$ is above $\Hh$ without drop and
$M^\Tt_\infty\wins M^\Uu_\infty$. Because $\Hh$ is a universal hull of
$\Cc=\core_1(\M)$, this implies that $b^\Uu$ does not drop and $M^\Tt_\infty=M^\Uu_\infty$;
now deduce that $\M\wdoubleupto\rho^{+\M}=\Hh\wdoubleupto\rho^{+\Hh}$ as usual, completing the
proof.\footnote{\color{black}There are some minor details involved here which are not explicitly discussed in \cite{steel2010outline}, etc,
and which we have not mentioned. We will cover such details in   the solidity to proof to follow.}

We now show that $\Cc=\Hh$, and therefore that $\Cc$ is an ${\Fop}$-pm.
{\color{black}Recall that $\Hh$ is an $\Fop$-pm, $\Hh^-=\Cc^-$,
 $\sigma:\Hh\to\Cc$
is a $0$-embedding which is above $\Hh^-$,
and every $\bfrSigma_1^{\Hh}$
subset of $\Hh^-=\Cc^-$
is
 $\bfrSigma_1^{\Hh}$.}
Therefore
$\rho_1^\Hh=\rho=\rho_1^{\Cc}<\rho_\om^{\Hh^-}$,
and since $\Hh$
is $\Hh^-$-sound,
also $p_1^\Cc\leq\sigma(p_1^{\Hh})$.
{\color{black}Recall from Definition \ref{dfn:opm} that $q^\Hh=p_1^\Hh\cap(\OR^{\Hh^-},\OR^{\Hh})$, and likewise for $q^\Cc$.}\label{page:q}
Since $\Hh$ is $(1,q^\Hh)$-solid,
$\Cc$ is $(1,\sigma(q^\Hh))$-solid (since by stratification, $\sigma$ preserves solidity), so $\sigma(q^\Hh)\wins p_1^\Cc$.
And since $\sigma$ is above $\Cc^-$ {\color{black}and $\Hh|\rho^{+\Hh}=\Cc|\rho^{+\C}$}, it follows that $\sigma(p_1^\Hh)=p_1^\Cc$. But by
$1$-universality, $\pi(p_1^\Cc)=p_1^\M$, so $\Cc=\Hull_1^\Cc(A\un\rho\un p_1^\Cc)$,
so $\Hh=\Cc$ and $\sigma=\id$, completing the proof.

Now suppose $k>0$. Then $\Cc=\core_{k+1}(\M)$ is an
opm by \ref{lem:Sigma_1_hull_reflects_opm}, and is $k$-relevant as
$\rho_{k+1}^\Cc<\rho_k^\Cc\leq\rho_\om^{\Cc^-}$. So by fine condensation and
\ref{lem:strengthen_scf_conclusion}, $\Cc={\Fop}(\Cc^-)$ is an ${\Fop}$-pm. The rest is a
simplification of the argument for $k=0$.

Now consider $(k+1)$-solidity. {\color{black}Here we will not assume that $\M$ is a successor, since there are some subtleties we want to deal with explicitly, which do not arise in the successor case.} Let
$q=p_{k+1}^\M$, $i<\lh(q)$, $\Ww=\mathcal{W}^\M_{k+1}(q_i)$ be the $(k+1)$-solidity witness
for $\M$ at $q_i$ (see \ref{dfn:fine_structure}),
and  $\sigma:\Ww\to\M$ the uncollapse map. We have
$\rho_{k+1}^\Ww\leq\mu$ where $\mu=\crit(\sigma)=q_i$.
By \ref{lem:core_embedding} we may assume that $\sigma$ is $k$-good, so:
\begin{enumerate}[label=--]\item $\Ww$ is a $k$-sound 
Q-opm,
\item if $\M$ is a successor then $\sigma(\Ww^-)=\M^-$,  and
\item
{\color{black} if $\M$ is a limit or $k>0$ then $\Ww$ is an opm}.
\end{enumerate}
{\color{black}We may also assume that
$\Ww^\M_{k+1}(q_j)\in\M$ for each $j<i$, or in other words,
\begin{equation}\label{eqn:solidity_induction}\M\text{ is }(k+1)\text{-solid with respect to }q\rest i.\end{equation}
}

{\color{black}Suppose for the moment that $\M$ is a successor.} Then by \ref{lem:succ_opm_top_param} we may assume that
$\mu<\rho_\om^{\M^-}$, so $\mu\leq\rho_\om^{\Ww^-}$.\label{page:mu<=rho_om^Ww^-} Suppose $\mu=\rho_\om^{\Ww^-}$.
Then since $\M^-$ is ${<\om}$-condensing, {\color{black}either
$\Ww^-\pins\M^-$
or $\M|\mu$ is active
and $\Ww^-\pins\Ult(\M|\mu,F^{\M|\mu})$,
and in either case,}
 ${\Fop}(\Ww^-)\in\M^-$. But by the fine condensation of
${\Fop}$, $\Ww$ is computable from ${\Fop}(\Ww^-)$, and so $\Ww\in\M^-$, which suffices.
{\color{black}(That is, if
 $\Ww\in\Fop(\Ww^-)$
or $\Ww=\Fop(\Ww^-)$,
we are done. But otherwise,
by fine condensation,
$k=0$ and there is a universal hull $\Hh$ of $\Ww$ such that
$\Hh\in\Fop(\Ww^-)$ or $\Hh=\Fop(\Ww^-)$,
and so $\Hh\in\M^-$.
But $\Ww=\Hull_1^{\Ww}(\Ww^-\cup\{x\})$ for some $x\in\Ww^-$,
so working in $\M^-$,
we can recover $\Ww$ from $\Hh$.)}
 So we may assume that
$\mu<\rho_\om^{\Ww^-}$, so $\Ww$ is $k$-relevant, so $\Ww\notin{\Fop}(\Ww^-)$ and if $k=0$ then
$\Ww$ has no universal hull in ${\Fop}(\Ww^-)$.
If $k=0$, let $\Hh={\Fop}(\Ww^-)$; by fine condensation, $\Hh$ is an ${\Fop}$-pm, and is a
universal hull of $\Ww$; let $\sigma':\Hh\to\Ww$ be a map witnessing this, with $\sigma'=\id$ if $\Hh=\Ww$. If $k>0$ then $\Ww$ is an opm, so by fine condensation, $\Ww={\Fop}(\Ww^-)$
is an ${\Fop}$-pm. If $k>0$, let $\Hh=\Ww$.

\label{page:solidity_proof_changed}{\color{black}Now if $\M$ is in fact a limit,} then so is $\Ww$;
in this case let $\Hh=\Ww$.

We now return to the general case (successor or limit).
We just need to see that $\Hh\in\M$.
(For if $\Ww\neq\Hh$ then $\M$ is a sucessor, $\mu<\rho_\om^{\Ww^-}$, $k=0$ and $\Ww=\Hull_{1}^\Ww(\mu\cup\{x\})$ for some $x\in\Ww$, and since $\Hh$ is a universal hull, $\Th_{\Sigma_1}^\Ww(\mu\cup\{x\})\in\J(\Hh)\sub\M$.) So let us assume from now on, for simplicity, that
\[ \Hh\notin\M;\]
we will derive a contradiction. It easily follows that
\[ \mu=\kappa^{+\Hh}<\OR^{\Hh}.\]
If $\Hh=\Ww$ let $\sigma'=\id:\Hh\to\Hh$.
Let $\pi=\sigma\com\sigma':\Hh\to\M$. Note that $\crit(\pi)=\mu$.

Let us also assume from now on, for focus, that $\mu$ is not an $\M$-cardinal, since the $\M$-cardinal case is easier and more routine.
So we have $\mu=\kappa^{+\Hh}=\kappa^{+\Ww}$ for some $\M$-cardinal $\kappa$.
Let $\R\wpins\M$ be least such that $\mu\leq\OR^\R$ and $\rho_\om^\R=\kappa$.
Let \[\ph=((\M,{<\kappa}),(\R,\kappa),\Hh).\] Then $\ph$ is $(k,r,k)$-maximally {\color{black}$\Fop$-}iterable,
where $r$ is least such that $\rho_{r+1}^\R=\kappa$, by lifting to $k$-maximal trees $\Vv$ on $\M$
(possibly $r=-1$, which holds iff $\R$ is active type 3 with $\mu=\OR^{\R}$). In fact,  fix an enumeration $e$ of $\M$ in ordertype $\om$, with $e\rest i=q\rest i$ (recall $q=p_{k+1}^\M$ and $q_i=\mu$). Fix a $(k,\om_1+1)$-iteration strategy $\Sigma$ for $\M$ with weak DJ with respect to $e$. Then we will form $\Vv$ according to $\Sigma$.

Note here that since $\R\pins\M$ and $\rho_\om^\R=\rho_{r+1}^\R=\kappa$ is an $\M$-cardinal,
the model dropdown sequence of $\R$ (see \ref{dfn:restricted_it_maps}) contains only one element, $\R$ itself. Moreover, $\R$ is $\M$-stable (see \ref{dfn:restricted_it_maps}),
since $\kappa<\rho_0^\M$. So the usual methods for lifting trees on phalanxes (as in \cite{steel2010outline}, \cite{deconIMT} and \cite{FSIT}),
combined with the methods in the proof of Lemma \ref{lem:strategy_copying}, do work here. In particular,
note that we are in the situation described in \cite[p.~728, proof of Theorem 3.3, paragraph ``There is one wrinkle in the copying argument...'']{deconIMT}, and we use that method mentioned there, defining $\pi_{\eta+1}:M^\Tt_{\eta+1}\to \R'=M^{\Vv}_{\R,\iota(\eta+1)}\pins M^\Vv_{\iota(\eta+1)}$ when $\crit(E^\Tt_\eta)=\kappa$. Note that $\kappa'=i^\Vv_{0,\iota(\eta+1)}(\kappa)$ is an $M^\Tt_{\iota(\eta+1)}$-cardinal
and $\rho_{r+1}^{\R'}=\kappa'<\rho_r^{\R'}$.
As was discussed in the analogous situation in
%conf
\cite[Proof of Claim 9, \S14.2]{fsfni},
letting $E^\Vv_{\eta'}$ be the copy of $E^\Tt_\eta$,
if $E^\Vv_{\eta'}$ is superstrong,\label{page:superstrong_copying} then there is a further small wrinkle. For in this case, $E^\Tt_\eta$ is also superstrong and $\pi_{\eta+1}(\lambda(E^\Tt_\eta))=\kappa'<(\kappa')^{+\R'}=\pi_{\eta+1}(\lh(E^\Tt_\eta))$, and by commutativity, it is easy enough to see that $\pi_{\eta+1}$ is non-$\nu$-high,
so $E^\Vv_{\eta'+1}$ is the $\pi_{\eta+1}$-copy of $E^\Tt_{\eta+1}$, and $\lh(E^\Tt_\eta)<\lh(E^\Tt_{\eta+1})$, so  $(\kappa')^{+\R'}<\lh(E^\Vv_{\eta'+1})\leq\OR^{\R'}<\lh(E^\Vv_{\eta'})$,
and in particular, $\lh(E^\Vv_{\eta'+1})<\lh(E^\Vv_{\eta'})$.
So $\Vv$ itself is not in fact a $k$-maximal tree;
it is \emph{essentially-$k$-maximal}
%conf
in the sense of \cite[\S6.2]{fsfni}.
But this does not actually cause a problem, because one can easily replace $\Vv$ with an essentially equivalent $k$-maximal tree $\Vv'$, in particular with $M^{\Vv'}_\infty=M^\Vv_\infty$ and corresponding iteration maps. In $\Vv'$, $E^\Vv_{\eta'}$ does not get used; instead, the next extender used in $\Vv'$ is $E^{\Vv}_{\eta'+1}$, unless that extender is also superstrong with critical point $\kappa$, in which case $\lh(E^\Vv_{\eta'+2})<\lh(E^\Vv_{\eta'+1})$, etc,
producing a (finitely long) strictly decreasing sequence.
Then the next extender used in $\Vv'$
is simply the next one used in $\Vv$ which is \emph{not} superstrong with critical point $\kappa$,
if there is one, or otherwise the last extender used in $\Vv$.
See \cite[\S6.2]{fsfni} for more details. We will continue to work with $\Vv$ itself, not $\Vv'$, however.

 Let $(\Tt,\Uu)$ be the successful comparison of $(\ph,\M)$ and let $\Vv$ be the lift of $\Tt$ to a tree on $\M$,
 formed with $\Tt$ and $\Vv$ according to $\Sigma$. In order that we can make use of standard fine structural preservation arguments (like \cite[Lemma 4.5]{FSIT}), we want to know:

\begin{clm}
$E^\Tt_\alpha$ is close to $M^{*\Tt}_{\alpha+1}$
for every $\alpha+1<\lh(\Tt)$.
\end{clm}
Recall here we are assuming that $\Hh\notin\M$;
this will be used in the proof of the claim.
Since $\mathfrak{P}$ is a phalanx, the claim does not literally follow from \cite[Lemma 6.1.5]{FSIT}, but needs a slight variant of that argument. Such arguments were also given in \cite{copy_con} and
%conf
\cite[Proof of Claim 10, \S14.2]{fsfni}, but we include an argument here for better self-containment. We will follow an inductive proof much as in \cite[6.1.5]{FSIT}, and the reader should have that proof in mind; we will point out the key differences.
\begin{proof}[Proof of Claim]Let $b\in[\nu(E^\Tt_\alpha)]^{<\om}$ and consider the measure $(E^\Tt_\alpha)_b$. We must see that \begin{equation}\label{eqn:E_is_close_to_N}(E^\Tt_\alpha)_b\text{ is }\bfrSigma_1^{M^{*\Tt}_{\alpha+1}}.\end{equation}

Suppose first that $(E^\Tt_\alpha)_b\in M^\Tt_\alpha$. Things are as usual by induction as usual unless $\pred^\Tt(\alpha+1)=\M$ or $\pred^\Tt(\alpha+1)=\R$.

Suppose $\pred^\Tt(\alpha+1)=\M$,
and so $M^{*\Tt}_{\alpha+1}=\M$ (as $E^\Tt_\alpha$ is $\M$-total in this case).
If $\mathrm{root}^\Tt(\alpha)=\M$ then line (\ref{eqn:E_is_close_to_N}) is shown  as usual.
If $\mathrm{root}^\Tt(\alpha)=\Hh$ then we get as usual
that $(E^\Tt_\alpha)_b$ is $\bfrSigma_1^\Hh$, but then since $\pi:\Hh\to\M$ is $\Sigma_1$-elementary and $\crit(\pi)=\mu$, it follows that $(E^\Tt_\alpha)_b$ is $\bfrSigma_1^\M$ also. And if $\mathrm{root}^\Tt(\alpha)=\R$
then it is similar, but this time because $\R\in\M$.

So suppose that $\pred^\Tt(\alpha+1)=\R$, so $M^{*\Tt}_{\alpha+1}=\R$ and $\crit(E^\Tt_\alpha)=\kappa$. Note then that since $\kappa^{+\Hh}=\mu\leq\lh(E^\Tt_0)$ and
$\mu<\OR^{\Hh}$ and $E^\Tt_0\in\es_+^{\Hh}$, in fact  $\kappa^{+\Hh}<\lh(E^\Tt_0)$, and so $(E^\Tt_\alpha)_b\in\Hh$.
Since also $(E^\Tt_\alpha)_b\sub\kappa^{+\Hh}=\crit(\pi)$,
we can fix $\Yy\pins\Hh$ such that $\rho_\om^\Yy=\kappa^{+\Hh}$
and $(E^\Tt_\alpha)_b\in\Yy$.
By condensation applied to  $\pi\rest\Yy:\Yy\to\pi(\Yy)$, either $\M|\kappa^{+\Hh}$ is passive and $\Yy\pins\M$, or $\M|\kappa^{+\Hh}$ is active with extender $F$
and $\Yy\pins\Ult(\M|\lh(F),F)$. But if $\M|\kappa^{+\Hh}$ is passive, then since $\rho_\om^{\R}=\kappa<\kappa^{+\Hh}=\rho_\om^\Yy$, we must have $\Yy\pins\R$, so $(E^\Tt_\alpha)_b\in\R$. And if $\M|\kappa^{+\Hh}$ is active with $F$, then $\R=\M|\lh(F)$,
so $\Yy\pins\Ult(\R,F^\R)$, so $(E^\Tt_\alpha)_b$ is $\bfrSigma_1^{\R}$. This establishes line (\ref{eqn:E_is_close_to_N})
in this case.

Now suppose that $(E^\Tt_\alpha)_b\notin M^\Tt_\alpha$.
So $E^\Tt_\alpha=F(M^\Tt_\alpha)$
and $M^\Tt_\alpha$ is active type 1 or 2,
and $\rho_1^{M^\Tt_\alpha}\leq\theta^{+\Hh}$ where $\theta=\crit(E^\Tt_\alpha)$.
If $\crit(E^\Tt_\alpha)>\kappa$, then one can argue
essentially as in \cite[6.1.5]{FSIT}, so suppose $\crit(E^\Tt_\alpha)\leq\kappa$. Suppose first that $\crit(E^\Tt_\alpha)<\kappa$, so $\pred^\Tt(\alpha+1)=\M$
and $M^{*\Tt}_{\alpha+1}=\M$ and $\rho_1^{M^\Tt_\alpha}\leq\kappa$. Letting $\xi=\mathrm{root}^\Tt(\alpha)$, the argument in  \cite{FSIT} therefore shows that $(\xi,\alpha]^\Tt$ does not drop in model, and that $\crit(E^\Tt_\alpha)<\crit(i^\Tt_{\xi\alpha})$. By induction,
all extenders used along $(\xi,\alpha]^\Tt$ are close to their target model,
and it follows that $(E^\Tt_\alpha)_b$ is $\bfrSigma_1^{\N}$,
where $\N=M^\Tt_\xi$. Like before, since $\R\in\M$ and $\pi:\Hh\to\M$ is $\Sigma_1$-elementary, this suffices to give line (\ref{eqn:E_is_close_to_N}). 

So now suppose that $\crit(E^\Tt_\alpha)=\kappa$, so $\pred^\Tt(\alpha+1)=\R$.
Supposing $\Tt$ drops in model in $(\xi,\alpha]^\Tt$, let $\gamma+1$ be the last such drop.
Then as in \cite{FSIT}, $\kappa<\crit(i^{*\Tt}_{\gamma+1,\alpha})$, so $\pred^\Tt(\gamma+1)\neq\M$ and $\pred^\Tt(\gamma+1)\neq\R$,
  and $(E^\Tt_\alpha)_b$ is $\bfrSigma_1^{M^{*\Tt}_{\gamma+1}}$.
So $(E^\Tt_\alpha)_b\in M^\Tt_\delta$ where $\pred^\Tt(\gamma+1)=\delta$, but then $(E^\Tt_\alpha)_b\in\Hh$,
and we can deduce that $(E^\Tt_\alpha)_b$ is $\bfrSigma_1^\R$ as before. So suppose $\Tt$ does not drop in model in $(\xi,\alpha]^\Tt$.
Then $\kappa<\crit(i^\Tt_{\xi\alpha})$,
and so $\xi=\Hh$.
Since $M^\Tt_\alpha$ is active type 1 or 2
and $\kappa=\crit(E^\Tt_\alpha)=\crit(F(M^\Tt_\alpha))$,
 $\Hh$ is therefore also active type 1 or 2
 and $\kappa=\crit(F^\Hh)$,
 so $\M$ is active type 1 or 2 and $\kappa=\crit(F^\M)$.
But $\Hh$ is determined by $F^\Hh$
and $\Hh|\kappa^{+\Hh}$,
and $F^\Hh$ is finitely generated (that is, generated by finitely many generators),
since $\Hh=\Hull_1^{\Hh}(\mu\cup\{x\})$ for some finite $x$.
 And $\kappa^{+\Hh}=\mu<\kappa^{+\M}$,
so $F^\Hh\in\M$,  so $\Hh\in\M$, a contradiction.
\end{proof}

\begin{clm}\label{clm:near_emb_after_drop}For all $\delta<\lh(\Tt)$, if either $\delta$ is above $\Hh$ or
above $\M$ or $(\mathrm{root}^\Tt(\delta),\delta]^\Tt$ drops in model
then $\deg^\Tt_\delta=\deg^\Vv_{\iota(\delta)}$ and $\pi_\delta:M^\Tt_\delta\to M^\Vv_{\iota(\delta)}$ is a near $\deg^\Tt_\delta$-embedding.\end{clm}
(Note here that if $\delta$ is above $\R$ and $(\R,\delta]^\Tt$ drops in model then $[0,\iota(\xi+1)]^\Vv$ drops below the image of $\R$,
so $M^{\Vv}_{\R,\iota(\xi+1)}=M^\Vv_{\iota(\xi+1)}$.)
\begin{proof}
If $\delta$ is above $\M$ and there is no drop in model or degree in $(\M,\delta]^\Tt$ then $(0,\iota(\delta)]^\Vv$ also does not drop in model or degree,
and $\pi_\delta$ is in fact a $\deg^\Tt_\delta$-embedding, by its commutativity with the iteration maps.
This applies in particular in the case that $\delta$ is a successor and $\pred^\Tt(\delta+1)=\M$, since $\kappa<\mu<\rho_k^\M$. And since $\Hh\notin\M$,
$\pi``\rho_k^\Hh$ is cofinal in $\rho_k^\M$,
so similar remarks apply to the case that $\delta$ is above $\Hh$ and there is no drop in model or degree in $(\Hh,\delta]^\Tt$. In the remaining cases, the claim
now follows from an inspection of the proof of \cite[Lemma 1.3]{fine_tame}.\end{proof}

We now begin to analyse the comparison. We will make use of the claim implicitly, in the usual fashion.

\begin{clm}$M^\Tt_\infty=M^\Uu_\infty$ and $b^\Uu$ does not drop in model or degree.\end{clm}
\begin{proof}
 We can't have $M^\Uu_\infty\pins M^\Tt_\infty$, by weak DJ.
 Suppose $M^\Tt_\infty\pins M^\Uu_\infty$.
 Then by weak DJ, $b^\Tt$ is not above $\M$,
 so it is above $\R$ or $\Hh$.
 And $M^\Tt_\infty$ is sound, which implies $\Tt$ is trivial and $M^\Tt_\infty=\Hh$.
But then it follows that $\Hh\in\M$, a contradiction.
 So $M^\Tt_\infty=M^\Uu_\infty$.
 
 Suppose $b^\Uu$ drops in model or degree; so $b^\Tt$ does not drop in model or degree. If $b^\Tt$ is above $\M$, then $M^\Tt_\infty$ is $k$-sound,
 so $b^\Uu$ must drop in model, and then $i^\Tt_{\M\infty}:\M\to M^\Tt_\infty=M^\Uu_\infty$ is nearly $k$-good,
 contradicting weak DJ (as $b^\Uu$ drops).
If $b^\Tt$ is above $\R$, then because $\R$ is sound,
it is just as if $b^\Tt$ drops in model, and so we reach the usual contradiction via compatible extenders. So $b^\Tt$ is above $\Hh$, but then the usual calculations yield that $\Hh\in\M$, a contradiction.
\end{proof}

\begin{clm}\label{clm:b^Tt_not_above_M}
$b^\Tt$ is not above $\M$.\end{clm}
\begin{proof} Suppose $b^\Tt$ is above $\M$.
If $b^\Tt$ does not drop in model or degree, then by weak DJ, $i^\Tt=i^\Uu$, giving the usual contradiction to comparison. So $b^\Tt$ drops in model or degree, and (since $M^\Tt_\infty=M^\Uu_\infty$ is $k$-sound)
therefore in model, and hence $b^\Vv$ also drops in model. By Claim \ref{clm:near_emb_after_drop}, $\pi_\infty:M^\Tt_\infty\to M^\Vv_\infty$ is a near $\deg^\Tt_\infty$-embedding, and since $\deg^\Tt_\infty\geq k$
(as $M^\Uu_\infty$ is $k$-sound),
therefore $\pi_\infty$ is a near $k$-embedding.
But then
 $\pi_\infty\com i^\Uu:\M\to M^\Vv_{\infty}$ is a near $k$-embedding, contradicting weak DJ.
\end{proof}

\begin{clm}\label{clm:b^Tt_not_above_H}
 $b^\Tt$ is not above $\Hh$.\end{clm}
 \begin{proof}
Suppose it is above $\Hh$. Just as in the proof of Claim \ref{clm:b^Tt_not_above_M}, it does not drop in model or degree.
By Claim \ref{clm:near_emb_after_drop},
$\pi_\infty:M^\Tt_\infty\to M^\Vv_\infty$ is a near $k$-embedding. So note that by our choice of $\Sigma$ and enumeration of $e$ (with $e\rest i=q\rest i$, where $q=p_{k+1}^\M$),
letting $\pi(\bar{q})=q\rest i$,
we have $i^\Tt_{\Hh\infty}(\bar{q})\leq i^\Uu(q\rest i)$.
(For $\pi_\infty(i^\Tt_{\Hh\infty}(\bar{q}))=i^\Vv(\pi(\bar{q}))=i^\Vv(q\rest i)\leq \pi_\infty(i^\Uu(q \rest i))$, by weak DJ.)

On the other hand, $i^\Uu(q\rest i)\leq i^\Tt_{\Hh\infty}(\bar{q})$.
For by our inductive hypothesis (\ref{eqn:solidity_induction}),
$\M$ is $(k+1)$-solid with respect to $q\rest i$,
so $M^\Uu_\infty$ is $(k+1)$-solid with respect to $i^\Uu(q\rest i)$, so if $i^\Tt_{\Hh\infty}(\bar{q})<i^\Uu(q\rest i)$
then 
\[ \Th_{k+1}^{M^\Uu_\infty}(\mu\cup\{\pvec_k^{M^\Uu_\infty},i^\Tt_{\Hh\infty}(\bar{q})\})\in M^\Uu_\infty,\]
but $\Hh$ and $M^\Uu_\infty$ have the same subsets of $\mu$,
so then
\[ \Th_{k+1}^{\Hh}(\mu\cup\{\pvec_k^\Hh,\bar{q}\})\in\Hh,\]
whereas $\Hh=\Hull_{k+1}^\Hh(\mu\cup\{\pvec_k^\Hh,\bar{q}\})$,
which is impossible.

So we have established that $i^\Tt_{\Hh\infty}(\bar{q})=i^\Uu(q\rest i)$.

Now if $\rho_{k+1}^\M\leq\crit(i^\Uu)$ then
$\rho_{k+1}^\M=\rho_{k+1}^{M^\Uu_\infty}=\rho_{k+1}^{M^\Tt_\infty}=\rho_{k+1}^\Hh\leq\crit(i^\Tt)$,
and  $t=\Th_{k+1}^\M(\rho_{k+1}^\M\cup\{\pvec_{k+1}^\M\})$
is $\bfrSigma_{k+1}^{M^\Uu_\infty}$,
hence $\bfrSigma_{k+1}^\Hh$, and $\bfrSigma_{k+1}^{\Ww}$,
which contradicts the minimality of $p_{k+1}^\M$.
So  $\crit(i^\Uu)<\rho_{k+1}^\M$.
A straightforward calculation shows that $\sup i^\Uu``\rho_{k+1}^\M\leq\rho_{k+1}^{M^\Uu_\infty}$.\footnote{In fact $\rho_{k+1}^{M^\Uu_\infty}=\sup i^\Uu``\rho_{k+1}^\M$, by \cite[Lemma 3.8]{fsfni} or the methods of  proof of \cite[Corollary 2.24]{extmax},  but as discussed at the start of \S\ref{sec:solidity}, we are avoiding using those results in this proof.} But $\rho_{k+1}^\Hh\leq\mu\leq\crit(i^\Tt)$,
and it follows that $\rho_{k+1}^\Hh=\mu=\lh(E^\Uu_0)$
and $E^\Uu_0$ is superstrong,
and $\crit(E^\Uu_0)^{+\M}=\rho_{k+1}^\M$,
and $\rho_{k+1}^{M^\Uu_1}=\mu$.\footnote{In fact, letting $t'=\Th_{k+1}^{M^\Uu_1}(\mu\cup\{i^\Tt_{01}(\pvec_{k+1}^\M)\})$,
then $t'\notin M^\Uu_1$,
for otherwise, $t'\in\M$, and  $E^\Uu_0\in\M$,
but from $t'$ and $E^\Uu_0$, one can easily compute $t$,
so $t\in M$, a contradiction. Here $t$ was defined above; 
a statement $\varphi$ is in $t$ iff $i_{E^\Uu_0}(\varphi)\in t'$.}  But since $q_i=\mu\in p_{k+1}^\M$,
we have
\[ u=\Th_{k+1}^\M(\rho_{k+1}^\M\cup\{\pvec_k^\M,q\rest i\})\in \M.\]
So $i^\Uu(u)\in M^\Uu_\infty$, and the usual arguments with solidity witnesses (and preservation of the standard parameter under iteration maps) show that from $i^\Uu(u)$, we can recover
\[ u'=\Th_{k+1}^{M^\Uu_\infty}(\mu\cup\{\pvec_k^{M^\Uu_\infty},i^\Uu(q\rest i)\})\in M^\Uu_\infty.\]
But then since $i^\Tt_{\Hh\infty}(\bar{q})=i^\Uu(q\rest i)$,
we again get that $\Th^\Hh_{k+1}(\mu\cup\{\pvec_k^\Hh,\bar{q}\})\in\Hh$, again a contradiction.
\end{proof}

\begin{clm}$b^\Tt$ is above $\R$, does not drop in model,
$0\leq k=\deg^\Tt_\infty< r$.
\end{clm}
\begin{proof} $b^\Tt$ is above $\R$ by Claims \ref{clm:b^Tt_not_above_M} and \ref{clm:b^Tt_not_above_H}. Now if $b^\Tt$ drops in \emph{model},
then by Claim \ref{clm:near_emb_after_drop}, $\pi_\infty$ is nearly $k$-good, so we can apply weak DJ for a contradiction.
We have $r\geq 0$,  since if $r=-1$ then $F^{M^\Tt_\infty}$ fails the ISC, and hence $M^\Tt_\infty$ is not an opm. Since $M^\Tt_\infty=M^\Uu_\infty$ is $k$-sound,
$k\leq\ell\leq r$ where $\ell=\deg^\Tt_\infty$.
The final copy map
$\pi_\infty:M^\Tt_\infty\to M^{\Vv}_{\R,\infty}$ is a weak $\ell$-embedding. If $k<\ell$ then
$\pi_\infty$ is a near $k$-embedding,
and so $\pi_\infty\com i^{\Uu}:\M\to  M^{\Vv}_{\R,\infty}$
 is also a near $k$-embedding,  and either $M^{\Vv}_{\R,\infty}\pins M^\Vv_\infty$ or $b^\Vv$ drops in model, contradicting  weak DJ.
So $k=\ell\leq r$.

Now if $k=\ell=r$ then some fairly standard fine structural calculations give a contradiction:
We have $\rho_{k+1}^{\R}=\kappa<\rho_k^{\R}$,
and as $\crit(i^\Tt)=\kappa$. Using closeness,  \cite[Lemma 4.5]{FSIT} (adapted to our context) now gives that 
 $\rho_{k+1}^{M^\Tt_\infty}=\kappa$
and $i^\Tt(p_{k+1}^\R)=p_{k+1}^{M^\Tt_\infty}$. As earlier, $\sup i^\Uu``\rho_{k+1}^\M\leq\rho_{k+1}^{M^\Uu_\infty}=\rho_{k+1}^{M^\Tt_\infty}=\kappa$.
But since $\kappa<\lh(E^\Uu_\alpha)$ for all $\alpha+1<\lh(\Uu)$, it follows that $\kappa\leq\crit(i^\Uu)$,
and so in fact $\kappa=\rho_{k+1}^\M$.
But now as usual
(as in the proof of
\cite[Lemma 4.5]{FSIT}), it follows  that $\Th_{\rSigma_{k+1}}^{\M}(\rho_{k+1}^\M\cup\pvec_{k+1}^\M)$
is definable from parameters over $\R$,
and hence an element of $\M$,
a contradiction.
\end{proof}

Now 
$\pi_\infty$ can't be a near $k$-embedding,
since otherwise $\pi_\infty\com i^\Uu:\M\to M^{\Vv}_{\R,\infty}$
is a near $k$-embedding,
and as $\R\pins\M$,
this contradicts weak DJ.  So the following claim reaches a contradiction, completing the proof of solidity:

\begin{clm}\label{clm:pi_infty_near_k}$\pi_\infty$ is a near $k$-embedding.\end{clm}
\begin{proof}
The proof will be a variant of the proof of \cite[Lemma 1.3]{fine_tame}.
Let $\alpha_0$ be least such that $\alpha_0+1\in b^\Tt$,
so $M^{*\Tt}_{\alpha_0+1}=\R$ and $\deg^\Tt_{\alpha_0+1}=r$.
Since $k=\ell<r$ where $\ell=\deg^\Tt_\infty$, there is  $\alpha_1+1\in(\alpha_0+1,\infty]^\Tt$ such that, letting $\beta_1=\pred^\Tt(\alpha_1+1)$,
we have $\deg^\Tt_{\beta_1}=r$ and $k\leq\deg^\Tt_{\alpha_1+1}<r$. Let $\pi_\gamma:M^\Tt_\gamma\to M^{\Vv}_{\R,\iota(\gamma)}$ be the copy map; we have $M^{\Vv}_{\R,\iota(\gamma)}\ins M^\Vv_{\iota(\gamma)}$.
Then $\pi_{\beta_1}$ is a weak $r$-embedding,
and so a near $(r-1)$-embedding. Note moreover that $i^{\Tt}_{\R,\beta_1}(\kappa)<\rho_r^{M^\Tt_{\beta_1}}\leq\crit(E^\Tt_{\alpha_1})$,
so $\pi_{\beta_1}(i^\Tt_{\R,\beta_1}(\kappa))=i^{\Uu}_{\R,0,\iota(\beta_1)}(\kappa)<\pi_{\beta_1}(\crit(E^\Tt_{\alpha_1}))$.
So either $M^{\Vv}_{\iota(\beta_1)}=M^{\Vv}_{\R,\iota(\beta_1)}$
or $\Vv$ drops in model at $\iota(\alpha_1+1)$
and $M^{\Vv}_{\iota(\alpha_1+1)}=M^{\Vv}_{\R,\iota(\alpha_1+1)}$.

If $\rho_r^{M^{\Vv}_{\R,\iota(\beta_1)}}\leq\pi_{\beta_1}(\crit(E^\Tt_{\alpha_1}))$ then $\deg^\Tt_{\alpha_1+1}=\deg^\Vv_{\alpha_1+1}$, and because $\pi_{\beta_1}$ is a near $\deg^\Tt_{\alpha_1+1}$-embedding,
an inspection of the proof of \cite[Lemma 1.3]{fine_tame}
shows that for all $\xi\in[\alpha_1+1,\infty]^\Tt$, we have
$\deg^\Tt_{\xi}=\deg^\Vv_{\iota(\xi)}$ and
$\pi_\xi:M^\Tt_\xi\to M^{\Vv}_{\iota(\xi)}=M^\Vv_{\R,\iota(\xi)}$
is a near $\deg^\Tt_{\xi}$-embedding, so $\pi_\infty$ is a near $k$-embedding, as desired.

So suppose from now on that $\pi_{\beta_1}(\crit(E^\Tt_{\alpha_1}))<\rho_r^{M^{\Vv}_{\R,\iota(\beta_1)}}$,
and so $\deg^\Vv_{\iota(\alpha_1+1)}=r$ (since $\rho_{r+1}^\R=\kappa$,
we have $\deg^\Vv_{\iota(\alpha_1+1)}\leq r$).
Note then that $\crit(E^\Tt_{\alpha_1})<\rho_{r-1}^{M^\Tt_{\beta_1}}$, since either $\rho_{r-1}^{M^\Tt_{\beta_1}}=\rho_0^{M^\Tt_{\beta_1}}$ and $\rho_{r-1}^{M^\Vv_{\R,\iota(\beta_1)}}=\rho_0^{M^\Vv_{\R,\iota(\beta_1)}}$,
or $\pi_{\beta_1}(\rho_{r-1}^{M^\Tt_{\beta_1}})=\rho_{r-1}^{M^\Vv_{\R,\iota(\beta_1)}}$. So $\deg^\Tt_{\alpha_1+1}=r-1$.

For  $r$-sound opm $\N$ and $\rho\leq\rho_0^\N$,
the \emph{weak $\bfrSigma_r^\N$-cofinality of $\rho_{r-1}^\N$},
denoted $\mathrm{wcof}^{\bfrSigma_r^\N}(\rho_{r-1}^\N)$,
is the least $\theta\leq\rho_r^\N$ such that there is $x\in\core_0(\N)$ with $\rho_{r-1}^\N\cap\Hull_{r}^\N(\theta\cup\{x\})$  cofinal in $\rho_{r-1}^\N$.

A degree $r$ iteration map $j:\N\to\N'$ preserves $\mathrm{wcof}^{\bfrSigma_r}(\rho_{r-1}^\N)$, to the extent that if $\theta=\mathrm{wcof}^{\bfrSigma_r^\N}(\rho_{r-1}^\N)<\rho_r^\N$
then $\mathrm{wcof}^{\bfrSigma_r^{\N'}}(\rho_{r-1}^{\N'})=j(\theta)$,
and if $\theta=\rho_r^\N$ then $\mathrm{wcof}^{\bfrSigma_r^{\N'}}(\rho_{r-1}^{\N'})=\rho_r^{\N'}$ (but it might be that $j(\rho_r^\N)>\rho_r^{\N'}$).  Moreover, say $\Xx$
is an iteration tree and $\N=M^{*\Xx}_{\alpha+1}$
and $\N'=M^\Xx_\infty$ and $j=i^{*\Xx}_{\alpha+1,\infty}$,
where $(\alpha+1,\infty]^\Xx\cap\mathscr{D}^\Xx=\emptyset$ and $\deg^\Xx_{\alpha+1}=\deg^\Xx_\infty=r$.
Then the following are equivalent:
\begin{enumerate}[label=--]
\item $j$ is discontinuous at $\rho_{r-1}^\N$,
\item $\theta<\rho_n^\N$ and $j$ is discontinuous at $\theta$,
\item $\theta<\rho_n^\N$ and either $\theta=\crit(j)$
or there is $\gamma$ such that $\alpha+1\leq^\Xx\gamma<^\Xx\infty$
and $i^{*\Xx}_{\alpha+1,\gamma}(\theta)=\crit(i^\Xx_{\gamma,\infty})$.
\end{enumerate}
These facts follow from some straightforward arguments;
there are very similar calculations (and more) in \cite[\S6.1]{fsfni}.

Let $\beta'\leq^\Tt\beta_1$ be largest
such that $M^{\Vv}_{\R,\beta'}\pins M^\Vv_{\iota(\beta')}$.
So $i^{\Vv}_{\R,0,\iota(\beta')}:\R\to M^\Vv_{\R,\beta'}$ is fully elementary and $\pi_{\beta'}\com i^\Tt_{\R,\beta'}=i^\Vv_{\R,0,\iota(\beta')}$. Note that $\beta'$
is the least $\beta''\leq^\Tt\beta_1$
such that $\crit(i^\Tt_{\beta'',\infty})\geq i^\Tt_{0\beta''}(\kappa)$.
Let $\beta_2\in(\beta_1,\infty]^\Tt$ be largest such that
either $\beta_2=\infty$ or $\deg^\Vv_{\iota(\beta_2)}=r$.
So $\beta'\leq^\Tt\beta_1<^\Tt\beta_2$ and $i^\Vv_{\R,\iota(\beta'),\iota(\beta_2)}$ is a degree $r$ iteration map,
to which we can apply the preceding remarks on preservation of 
weak $\bfrSigma_r$-cofinality. We are interested here in both $i^\Vv_{\R,\iota(\beta'),\iota(\beta_1)}$ and $i^\Vv_{\R,\iota(\beta_1),\iota(\beta_2)}$.
\begin{sclm}\label{sclm:rho_r-1_cont}
$j=i^{\Vv}_{\R,\iota(\beta_1),\iota(\beta_2)}$ is continuous at $\rho=\rho_{r-1}^{M^{\Vv}_{\R,\iota(\beta_1)}}$.
\end{sclm}
\begin{proof}
Suppose first that $\mathrm{wcof}^{\bfrSigma_r^\R}(\rho_{r-1}^\R)=\rho_r^\R$. Then by the elementarity of $i^{\Vv}_{\R,0,\iota(\beta')}$,
$\mathrm{wcof}^{\bfrSigma_r^{M^\Vv_{\R,\iota(\beta')}}}(\rho_{r-1}^{M^\Vv_{\R,\iota(\beta')}})=\rho_r^{M^\Vv_{\R,\iota(\beta')}}$. But then by the preceding remarks on preservation of weak $\bfrSigma_r$-cofinality, $\mathrm{wcof}^{\bfrSigma_r^{M^\Vv_{\R,\iota(\beta_1)}}}(\rho_{r-1}^{M^\Vv_{\R,\iota(\beta_1)}})=\rho_r^{M^\Vv_{\R,\iota(\beta_1)}}$, and also by those remarks, it follows that $j$ is continuous at $\rho$.

Now suppose otherwise, so $\theta=\mathrm{wcof}^{\bfrSigma_r^\R}(\rho_{r-1}^\R)<\rho_r^\R$. By full elementarity,
$i^\Vv_{\R,0,\iota(\beta')}(\theta)=\mathrm{wcof}^{\bfrSigma_r^{M^\Vv_{\R,\iota(\beta')}}}(\rho_{r-1}^{M^\Vv_{\R,\iota(\beta')}})$, and so by the remarks,
$i^\Vv_{\R,0,\iota(\beta_1)}(\theta)=\mathrm{wcof}^{M^\Vv_{\R,\iota(\beta_1)}}(\rho_{r-1}^{M^\Vv_{\R,\iota(\beta_1)}})$.
But by commutativity, $i^\Vv_{\R,0,\iota(\beta_1)}(\theta)=\pi_{\beta_1}(i^\Tt_{\R,\beta_1}(\theta))<\sup\pi_{\beta_1}``\rho_r^{M^\Tt_{\beta_1}}$.
But $\crit(j)\geq\sup \pi_{\beta_1}``\rho_r^{M^\Tt_{\beta_1}}$,
so again by the remarks, $j$ is continuous at $\rho$, as desired.
\end{proof}

Now recall that $\pred^\Tt(\alpha_1+1)=\beta_1$ and $\deg^\Tt_{\beta_1}=r$
(so $\pi_{\beta_1}$ is a near $(r-1)$-embedding),  but $\deg^\Tt_{\alpha_1+1}=r-1$, 
whereas $\deg^\Vv_{\alpha_1+1}=r$.

\begin{sclm}$\pi_{\alpha_1+1}$ is a near $(r-1)$-embedding.\end{sclm}
\begin{proof}Since $\pi_{\alpha_1+1}$ is a weak $(r-1)$-embedding, we just have to verify $\rSigma_{(r-1)+1}$-elementarity. The proof is much as in the proof of \cite[Lemma 1.3]{fine_tame}. Let $b\in[\nu(E^\Tt_{\alpha_1})]^{<\om}$
 and let $f:[\crit(E^\Tt_{\alpha_1})]^{<\om}\to \core_0(M^\Tt_{\beta_1})$ be $\bfrSigma_{r-1}^{M^\Tt_{\beta_1}}$. Say $f$ is so defined from the parameter $q\in\core_0(M^\Tt_{\beta_1})$, and write this as $f=f_q$. Let $\varphi$ be an $\rSigma_{(r-1)+1}$ formula. We want to see that $M^\Tt_{\alpha_1}\sats\varphi([b,f_q])$ iff $M^\Vv_{\R,\iota(\alpha_1+1)}\sats\varphi([\psi(b),f_{\pi_{\beta_1}(q)}])$,
 where $\psi:M^\Tt_{\alpha_1}|\lh(E^\Tt_{\alpha_1})\to \N$
 is the relevant extender lifting map. As in \cite{fine_tame},
 we can find some $p\in\core_0(M^\Tt_{\beta_1})$
 such that $(E^\Tt_{\alpha_1})_b$ is $\rSigma_1^{M^\Tt_{\beta_1}}(\{b\})$, via a certain $\Sigma_1$ formula $\varrho$, and such that $(F^N)_{\psi(b)}$
 is $\rSigma_1^{M^\Vv_{\iota(\beta_1)}}(\{\pi_{\beta_1}(p)\})$,
 via the same formula $\varrho$.
 Say $\varphi(u)$ is the formula ``there is $t\in T_{r-1}$
 such that $\tau(t,u)$'', where $\tau$ is some $\Sigma_1$ formula; this assumes $r-1>0$, but if $r-1=0$, then one uses the usual kind of variant.
 Then $M^\Tt_{\alpha_1+1}\sats\varphi([b,f_q])$
 iff $M^\Tt_{\beta_1}\sats\varphi'(p,q,\pvec_{r-1}^{M^\Tt_{\beta_1}})$, where $\varphi'(p,q,\pvec_{r-1}^{M^\Tt_{\beta_1}})$ asserts ``there is $t\in T_{r-1}$
 such that $t$ is a theory in parameters $\alpha\cup\{p,q,\pvec_{r-1}^{M^\Tt_{\beta_1}}\}$ for some $\alpha<\rho_{r-1}^{M^\Tt_{\beta_1}}$, and there is some $X\in (E^\Tt_{\alpha_1})_b$, such that for all $x\in X$,
 $t$ codes a sub-theory $t'$ (of the appropriate form for elements of $T_{r-1}$, and with truth corresponding to truth exhibited directly in $t$) and $t$ exhibits that $\tau(t',f_q(x))$ holds''. By the $\rSigma_r$-elementarity
 of $\pi_{\beta_1}$,
 this holds iff $M^\Vv_{\R,\iota(\beta_1)}\sats\varphi'(\pi_{\beta_1}(p),\pi_{\beta_1}(q),\pvec_{r-1}^{M^\Vv_{\R,\iota(\beta_1)}})$,
 and note that by choice of $p,q$, and because $i^{*\Vv}_{\iota(\alpha_1+1)}$ is continuous at $\rho_{r-1}^{M^\Vv_{\R,\iota(\beta_1)}}$ (by Subclaim \ref{sclm:rho_r-1_cont}),
 this holds iff $M^\Vv_{\R,\iota(\alpha_1+1)}\sats\varphi([\psi(b),f_{\pi_{\beta_1}(q)}])$, so $\pi_{\alpha_1+1}$ is a near $(r-1)$-embedding, as desired.
 \end{proof}

Generalizing the previous argument directly, we have:
\begin{sclm}
For each $\xi\in[\alpha_1+1,\beta_2]^\Tt$,
 $\pi_\xi$ is a near $(r-1)$-embedding.\end{sclm}
 
 But now for nodes $\xi\in b^\Tt$ beyond $\beta_2$, we can argue just as before: an inspection of the proof of \cite[Lemma 1.3]{fine_tame}
 shows that $\deg^\Tt_\xi=\deg^\Vv_{\iota(\xi)}$
 and $\pi_\xi$ is a near $\deg^\Tt_\xi$-embedding
 for each such $\xi$.
 So $\pi_\infty$ is a near $k$-embedding, a completing the proof of the claim.
\end{proof} 
As mentioned just prior to Claim \ref{clm:pi_infty_near_k},
the claim yields a contradiction, completing the proof of solidity.

\label{page:condensation}Now consider part \ref{item:condensing} of the theorem,
regarding condensation. Let $k<\om$ and
let $\Hh$ be a $(k+1)$-sound potential opm which is soundly projecting. Let
$\pi:\Hh\to\M$ be nearly $k$-good, with $\rho=\rho_{k+1}^\Hh<\rho_{k+1}^\M$. Then $\Hh$ is in fact
an opm. Let us assume that $\Hh,\M$ are both successors, so $\pi(\Hh^-)=\M^-$. By fine condensation
of ${\Fop}$, $\Hh^-$ is an ${\Fop}$-pm, and either $\Hh\in{\Fop}(\Hh^-)$ or $\Hh={\Fop}(\Hh^-)$.
If $\Hh$ is not $k$-relevant then the result follows from the fact that $\M^-$ is
${<\om}$-condensing and
$\Hh^-$ is an ${\Fop}$-pm. So assume $\Hh$ is $k$-relevant, so $\Hh={\Fop}(\Hh^-)$.

We now use $\om$-weak DJ  and the usual phalanx comparison argument to reach the desired
conclusion.
Say $\ph=((\M,{<\rho}),\Hh)$ is the phalanx. Then $\ph$ is ${\Fop}$-$((\om,k),\om_1+1)$-iterable,
lifting to ${\Fop}$-$(\om,\om)$-maximal trees $\Vv$ on $\M$. (It could be that $\M$ is not
$k$-relevant. So we want to keep the degrees of nodes of $\Vv$ at $\om$ where possible,
to ensure that each $M^\Vv_\alpha$ is an ${\Fop}$-pm.) Suppose $\Tt$ is non-trivial. Because
$k<\om$,
if $M^\Tt_\infty$ is above $\Hh$ without drop in model or degree, $\pi_\infty$ need only be a weak
$k$-embedding.
But in this case, $M^\Tt_\infty$ is not $\om$-sound, which implies
$M^\Uu_\infty\wpins M^\Tt_\infty$, which contradicts $\omega$-weak DJ. The rest is routine.
\end{proof}

\section*{Acknowledgements}

First author partly funded by the Austrian Science Fund (FWF) [10.55776/Y1498].

\printindex

\bibliographystyle{plain}
\bibliography{operators}

\begin{thebibliography}{10}

\bibitem{adolf2024ideals}
Dominik Adolf, Grigor Sargsyan, Nam Trang, Trevor Wilson, and Martin Zeman.
\newblock Ideals and strong axioms of determinacy.
\newblock {\em Journal of the American Mathematical Society}, 37(4):1203--1273,
  2024.

\bibitem{busche2009strength}
Daniel Busche and Ralf Schindler.
\newblock The strength of choiceless patterns of singular and weakly compact
  cardinals.
\newblock {\em Annals of Pure and Applied Logic}, 159(1-2):198--248, 2009.

\bibitem{FSIT}
William~J. Mitchell and John~R. Steel.
\newblock {\em Fine structure and iteration trees}, volume~3 of {\em Lecture
  Notes in Logic}.
\newblock Springer-Verlag, Berlin, 1994.

\bibitem{wDJ}
Itay Neeman and John Steel.
\newblock A weak {D}odd-{J}ensen lemma.
\newblock {\em Journal of Symbolic Logic}, 64(3):1285--1294, 1999.

\bibitem{sargsyan2014non}
Grigor Sargsyan and Nam Trang.
\newblock Non-tame mice from tame failures of the unique branch hypothesis.
\newblock {\em Canadian Journal of Mathematics}, 66(4):903--923, 2014.

\bibitem{sargsyan2016tame}
Grigor Sargsyan and Nam Trang.
\newblock Tame failures of the unique branch hypothesis and models of
  $\sf{AD}_\mathbb{R}$ + {$\Theta$} is regular.
\newblock {\em Journal of Mathematical Logic}, 16(02):1650007, 2016.

\bibitem{fine_tame}
E.~Schimmerling and J.~R. Steel.
\newblock Fine structure for tame inner models.
\newblock {\em The Journal of Symbolic Logic}, 61(2):621--639, 1996.

\bibitem{deconIMT}
Ralf-Dieter Schindler, John Steel, and Martin Zeman.
\newblock Deconstructing inner model theory.
\newblock {\em The Journal of Symbolic Logic}, 67(2):721--736, 2002.

\bibitem{scales_in_hybrid_mice_over_R}
F.~Schlutzenberg and N.~Trang.
\newblock Scales in hybrid mice over $\mathbb{R}$.
\newblock arXiv:1210.7258v4.

\bibitem{fsfni}
Farmer Schlutzenberg.
\newblock Fine structure from normal iterability.
\newblock To appear in Journal of Mathematical Logic. Preprint
  arXiv:2011.10037v4.

\bibitem{kappa-plus}
Farmer Schlutzenberg.
\newblock The initial segment condition for $\kappa^+$-supercompactness.
\newblock arXiv:2306.13827v2.

\bibitem{mouse_scales}
Farmer Schlutzenberg.
\newblock Mouse scales.
\newblock arXiv:2310.19764v2.

\bibitem{copy_con}
Farmer Schlutzenberg.
\newblock Reconstructing copying and condensation.
\newblock Notes available at
  https://sites.google.com/site/schlutzenberg/home-1/research/papers-and-preprints.

\bibitem{mim}
Farmer Schlutzenberg.
\newblock {\em Measures in mice}.
\newblock PhD thesis, University of California, Berkeley, 2007.
\newblock arXiv:1301.4702.

\bibitem{premouse_inheriting}
Farmer Schlutzenberg.
\newblock A premouse inheriting strong cardinals from ${V}$.
\newblock {\em Annals of Pure and Applied Logic}, 171(9), 2020.

\bibitem{iter_for_stacks}
Farmer Schlutzenberg.
\newblock Iterability for (transfinite) stacks.
\newblock {\em Journal of Mathematical Logic}, 21(2), 2021.

\bibitem{extmax}
Farmer Schlutzenberg.
\newblock The definability of $\es$ in self-iterable mice.
\newblock {\em Annals of Pure and Applied Logic}, 174(2), 2023.

\bibitem{V=HODX}
Farmer Schlutzenberg.
\newblock The definability of the extender sequence $\mathbb{E}$ from
  $\mathbb{E}\upharpoonright\aleph_1$ in {$L[\mathbb{E}]$}.
\newblock {\em The Journal of Symbolic Logic}, 89(2):427--459, 2024.

\bibitem{CMIP}
J.~R. Steel.
\newblock {\em The core model iterability problem}, volume~8 of {\em Lecture
  Notes in Logic}.
\newblock Springer-Verlag, Berlin, 1996.

\bibitem{cmi}
J.~R. Steel and R.~D. Schindler.
\newblock The core model induction.
\newblock Unpublished notes, available at
  \url{https://ivv5hpp.uni-muenster.de/u/rds/}.

\bibitem{cmwmwc}
John Steel.
\newblock Core models with more {W}oodin cardinals.
\newblock {\em The Journal of Symbolic Logic}, 67(3):1197--1226, 2002.

\bibitem{fs_plus-one}
John Steel and Itay Neeman.
\newblock Fine structure for plus one premice.
\newblock Unpublished notes available at
  \url{https://math.berkeley.edu/~steel}.

\bibitem{PFA_implies_ADLR}
John~R. Steel.
\newblock $\sf{PFA}$ implies $\sf{AD}$$^{L(\RR)}$.
\newblock {\em Journal of Symbolic Logic}, 70(4):1255–1296, 2005.

\bibitem{steel2010outline}
John~R. Steel.
\newblock An outline of inner model theory.
\newblock In {\em Handbook of set theory. {V}ols. 1, 2, 3}, pages 1595--1684.
  Springer, Dordrecht, 2010.

\bibitem{trang2016pfa}
Nam Trang.
\newblock {PFA} and guessing models.
\newblock {\em Israel Journal of Mathematics}, 215(2):607--667, 2016.

\bibitem{trang2021determinacy}
Nam Trang and Trevor~M Wilson.
\newblock Determinacy from strong compactness of $\omega_1$.
\newblock {\em Annals of Pure and Applied Logic}, 172(6):102944, 2021.

\bibitem{wilson2012contributions}
Trevor~Miles Wilson.
\newblock {\em Contributions to Descriptive Inner Model Theory}.
\newblock PhD thesis, University of California, Berkeley, 2012.
\newblock Pro{Q}uest {ID}:{W}ilson$\_$berkeley$\_$0028{E}$\_$13013,
  \url{https://escholarship.org/uc/item/8fg2x6cr}.

\bibitem{zeman}
Martin Zeman.
\newblock {\em Inner models and large cardinals}, volume~5 of {\em de Gruyter
  Series in Logic and its Applications}.
\newblock Walter de Gruyter \& Co., Berlin, 2002.

\end{thebibliography}
\end{document}